\documentclass[a4paper, 12pt]{report}
\usepackage{amsmath,amsxtra,amssymb,latexsym, amscd,amsthm}
\usepackage{indentfirst}
\usepackage[mathscr]{eucal}
\usepackage{indentfirst}
\usepackage[mathscr]{eucal}
\usepackage{bbm} 
\usepackage{graphicx}
\usepackage{color}
\usepackage{enumerate}
\usepackage{float}
\usepackage{amsfonts}
\numberwithin{equation}{section}
\usepackage{enumerate}
\usepackage{subfig}
\usepackage[T2A,T5]{fontenc}
\usepackage[utf8]{inputenc}
\usepackage[english,vietnamese]{babel}
\usepackage{tcolorbox}

\theoremstyle{plain}
\newtheorem{Theorem}{Theorem}[section]

\newtheorem{prop}[Theorem]{Proposition}

\newtheorem{Proposition}[Theorem]{Proposition}

\newtheorem{lemma}[Theorem]{Lemma}
\newtheorem{Lemma}[Theorem]{Lemma}

\newtheorem{corollary}[Theorem]{Corollary}

\theoremstyle{definition}

\newtheorem{Definition}[Theorem]{Definition}

\newtheorem{remark}[Theorem]{Remark}

\def\HH{{\bf H}}
\def\NN{{\bf N}}
\def\hh{{\bf h}}
\def\t{{\bf t}}
\def\s{{\bf s}}
\def\aa{{\bf a}}
\def\bb{{\bf b}}
\def\ii{{\bf i}}
\def\xx{{\bf x}}
\def\yy{{\bf y}}

\def\uu{{\bf u}}
\def\vv{{\bf v}}
\def\kk{{\bf k}}

\def\rr{{\bf r}}
\def\cc{{\bf c}}

\newcommand{\Prob}{\mathbb{P}}
\newcommand{\E}{\mathbb{E}}

\newcommand{\R}{\mathbb{R}}

\begin{document}

   \begin{titlepage}

 \begin{figure}
    \centering
    \subfloat{\includegraphics[scale=0.12]{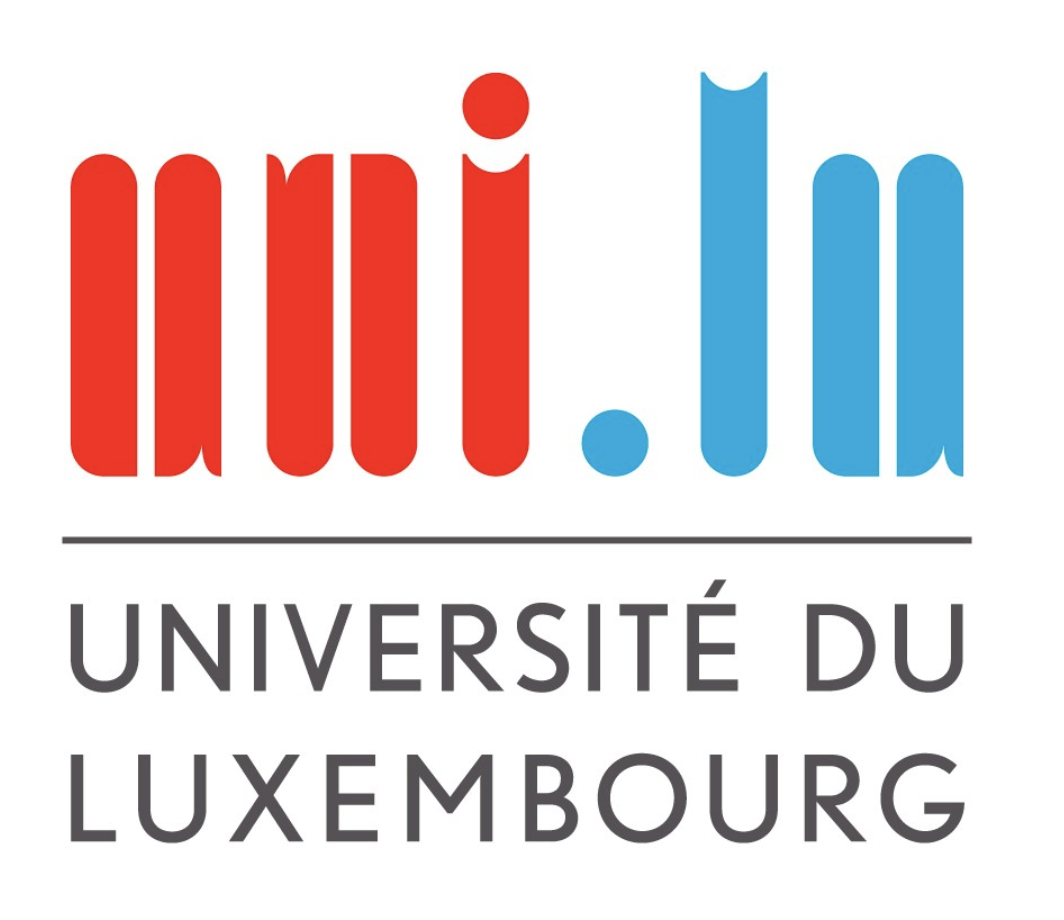} }%
    \qquad\qquad\qquad
\qquad
    \qquad
 \qquad
    \subfloat{\includegraphics[scale=0.6]{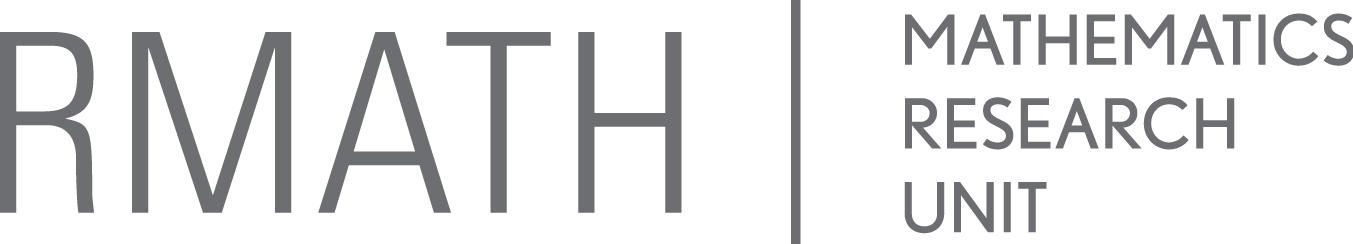} }
 \end{figure}
\noindent\hrulefill

\vspace{0.5cm}
\begin{center}
A thesis presented for the degree of  \\~\\
{\Large Doctor of Philosophy of Luxembourg University \\~\\
in Mathematics}\\~\\
by \\~ \\
\LARGE{\textbf{Thi Thanh Diu TRAN}}~\\~\\
      
\noindent\rule[0.5ex]{\linewidth}{1pt}

{\LARGE
      \textbf{Contributions to the asymptotic study of Hermite driven processes} }
  \noindent\rule[0.5ex]{\linewidth}{1pt} 
  \end{center}

\vspace{0.1cm}
\begin{center}
Defense on 25th January 2018\\
\end{center}

\vspace{0.1cm}

\begin{flushleft}
\noindent{\normalsize
\hspace{0.5cm} Members of jury:}
\begin{equation*}
\begin{aligned}
  &\textbf{Prof. Ivan Nourdin}  \\
&\textbf{Prof. Giovanni Peccati}  \\
&\textbf{Prof. Anthony R\'{e}veillac} \\
&\textbf{Prof. Anton Thalmaier} \\
&\textbf{Prof. Ciprian A. Tudor} 
\end{aligned}
\qquad
\begin{aligned}
  &\text{Supervisor}  \\
&\text{Vice Chairman}  \\
&\text{Member} \\
&\text{Chairman} \\
&\text{Member} 
\end{aligned}
\qquad
\begin{aligned}
  &\text{University of Luxembourg}  \\
&\text{University of Luxembourg}  \\
&\text{INSA de Toulouse} \\
&\text{University of Luxembourg} \\
&\text{University of Lille 1} 
\end{aligned}
\end{equation*}
\end{flushleft}

\vspace{1.2cm}
\noindent\hrulefill \\
{\footnotesize Universit\'{e} du Luxembourg, Unit\'{e} de Recherche en Math\'{e}matiques, Maison du Nombre, 6 avenue de la Fonte, L-4364 Esch-sur-Alzette, Luxembourg.}

\newpage
\thispagestyle{empty}

\vspace{3.5cm}

\begin{flushright}
To my parents and brothers for all their  love!\\
\vspace{0.8cm}

\selectlanguage{vietnamese}
Gửi cha Điển, mẹ Hiền và hai em Khoa, Quang

\end{flushright}

\newpage
 
\end{titlepage}

\pagestyle{plain}
\setlength{\baselineskip}{20pt}

\selectlanguage{english}
\tableofcontents
\chapter*{Acknowledgement}
\addcontentsline{toc}{chapter}{Acknowledgement}
\pagenumbering{roman} 

First and foremost, I would like to express my sincere gratitude and special appreciation to my PhD advisor Prof. Ivan Nourdin for his continuous and enthusiastic support of my Ph.D study and related research, for his amazing guidance and discussion in all the time of research and writing of this thesis, for his encouragement and patience during these past three years. I will forever be thankful to Ivan for inspiring me to the large mathematical world and encouraging me step by step to grow up as a research scientist. I also want to thank him for always noticing and helping me to complete the necessary procedures during the academic years.

Besides my advisor, my sincere thanks also go to the committee members: Prof. Giovanni Peccati, Prof. Anthony R\'{e}veillac, Prof. Anton Thalmaier and Prof. Ciprian Tudor, for generously giving their time to review my work and for their insightful questions.

I am greatly indebted to my former supervisors Prof. Nguyen Van Trao, to the directors of the Vietnamese Institute of Mathematics Prof. Ngo Viet Trung and Prof. Nguyen Viet Dung, to Prof. Patricia Reynaud-Bouret and to my PhD supervisor Prof. Ivan Nourdin who created the big turning-points both in my research and in my life. They provided me opportunities to experience the study environment abroad, to learn valuable mathematical knowledge, and to approach the modern research mathematical world. Without their precious support, it would not have been possible to achieve this PhD diploma. 

A very special gratitude goes to all professors in the Vietnamese Institute of Mathematics, University of Uppsala and Nice-Sophia-Antipolis, especially to my former advisors Prof. Ho Dang Phuc, Prof. Maciej Klimek, and Prof. C\'{e}dric Bernardin, for their great lectures and guidance during my studies there.

I gratefully acknowledge the University of Luxembourg, not only for providing the funding and comfortable working environment, but also for giving me the opportunity to attend conferences and to meet so many interesting people. It was fantastic to work in such a university.

I warmly thank Prof. M\'{a}ty\'{a}s Barczy, Prof. Christophe Ley, Prof. Gyula Pap and Prof. Frederi Viens for interesting discussions, ideas and several helpful comments and open questions about the research in this thesis.

With a special heartfelt thank to Guangqu for discussion around Malliavin calculus and Stein's method, to anh Hoang for his help on my first paper and bac Lam for his encouragements. Thanks to my labmates for the stimulating discussions and for all the fun we have had during my time at Luxembourg. Thanks to all my friends who accompanied me on every trips to discover Europe and shared with me these funny and unforgetable moments of life.

Last but not least, I would like to thank my family, my parents and my brothers, for all their love and encouragement, for supporting me spiritually throughout my study and my life in general. 
\selectlanguage{vietnamese}
Con cảm ơn bố mẹ vì những hi sinh vất vả để con có ngày hôm nay. Con cảm ơn gia đình luôn bên cạnh con, cỗ vũ, động viên ủng hộ con những lúc khó khăn. 
\selectlanguage{english}

The most surprising outcome from the last PhD year is to meet you, Jimmy. This thesis is also dedicated to you $\heartsuit$. Wish to see you again!

\chapter*{Organization of the thesis}
\addcontentsline{toc}{chapter}{Organization of the thesis}

This thesis consists of two parts.
\vspace{0.5cm}

Part I is an introduction to Hermite processes, Hermite random fields, Fisher information and to the papers constituting the thesis. More precisely, in Section 1 we introduce Hermite processes in a nutshell, as well as some of its basic properties. It is the necessary background for the articles [a] and [c]. In Section 2 we consider briefly the multiparameter Hermite random fields and we study some less elementary facts which are used in the article [b]. In section 3, we recall some terminology about Fisher information related to the article [d]. Finally, our articles [a] to [d] are summarised in Section 4.

\vspace{0.5cm}
Part II consists of the articles themselves:

[a] T.T. Diu Tran (2017): Non-central limit theorem for quadratic functionals of Hermite-driven long memory moving average processes. \textit{Stochastic and Dynamics}, \textbf{18}, no. 4.

[b] T.T. Diu Tran (2016): Asymptotic behavior for quadratic variations of non-Gaussian multiparameter Hermite random fields. Under revision for \textit{Probability and Mathematical Statistics}.

[c] I. Nourdin,  T.T. Diu Tran (2017): Statistical inference for Vasicek-type model driven by Hermite processes. Submitted to \textit{Stochastic Process and their Applications}.

[d] T.T. Diu Tran (2017+): Fisher information and multivariate Fouth Moment Theorem. Main results have already been obtained. It should be submitted soon.

\pagenumbering{arabic}
\setcounter{page}{1}
\chapter[Introduction]{Introduction}

\section{Hermite processes in a nutshell}

\subsection{Historical definition of Hermite processes } 

Hermite processes form a family of self-similar stochastic processes with long-range dependence. It includes the well-known \textit{fractional Brownian motion} (fBm in short) as a particular case, which is the only Hermite process to be Gaussian. Apart for Gaussianity, Hermite processes share a number of basic properties with the fBm, such as self similarity, stationary increments, long-range dependence and covariance structure. The lack of Gaussianity makes the Hermite process an interesting alternative candidate for modelling purposes. For instance, it can serve to understand how much a given fractional model relies on the Gaussian assumption, because we may use it to test the robustness of the model with respect to the Gaussian feature.

Originally, Hermite processes have first appeared as limits of correlated stationary Gaussian random sequences. It is, roughly speaking, what the so-called Non-Central Limit Theorem proved by Taqqu \cite{dTaqqu1975, dTaqqu1979} and Dobrushin, Major \cite{dDobrushinMajor} states. Before being in position to describe this result in more details, we first need to recall the important notion of Hermite rank.

Denote by $H_k(x)$ the Hermite polynomial of degree $k$, given by 
$$H_k(x) = (-1)^k e^{\frac{x^2}{2}} \frac{d^k}{dx^k}e^{-\frac{x^2}{2}}.$$
The first few Hermite polynomials are $H_1(x) = x, H_2(x) = x^2-1$ and $H_3(x) = x^3 -3x$. Assume on the other hand that $g$ belongs to $L^2(\mathbb{R}, \frac{1}{\sqrt{2\pi}}e^{-\frac{x^2}{2}}dx)$ and satisfies $\int_{\mathbb{R}} g(x)e^{-\frac{x^2}{2}}dx =0$. As such the function $g$ can be expressed as a linear sum of Hermite polynomials as follows
\begin{equation}\label{deq:01}
g(x) = \sum_{k=1}^\infty c_k H_k(x),
\end{equation}
where $c_k = \frac{1}{k!}E[g(N)H_k(N)]$ with $N \sim \mathcal{N}(0, 1)$. The Hermite rank of $g$ is then, by definition, the index $q$ of the first non-zero coefficient in the previous expansion (\ref{deq:01}):
$$q = \min \{k, c_k \ne 0 \}.$$

In the series of papers \cite{dDobrushinMajor, dTaqqu1975, dTaqqu1979} by Dobrushin, Major and Taqqu, the authors investigated the asymptotic behavior, as $N \to \infty$, of the following family of stochastic processes :
\begin{equation}\label{deq:TaqquDobrushinMajor}
\frac{1}{N^H} \sum_{i=1}^{[Nt]} g(X_i)
\end{equation}
where $X=(X_i)_{i \in \mathbb{Z}}$ is a stationary Gaussian sequence with mean $0$ and variance $1$ that displays long-range dependence. More precisely, let us assume that $X$ is such that its correlation function $r(n)= E[X_0X_n]$ satisfies
$$r(n)= n^{2H_0-2}L(n)$$
for some $H_0 \in (1 - \frac{1}{2q}, 1)$, with $q$ the Hermite rank of $g$ and $L$ a slowly varying function. The main result of \cite{dDobrushinMajor, dTaqqu1975, dTaqqu1979} is that the sequence (\ref{deq:TaqquDobrushinMajor}) converges, in the sense of finite-dimensional distributions, to a self-similar stochastic process with stationary increments belonging to the $q$-th Wiener chaos, called \emph{Hermite process} of order $q$ and self-similar parameter $H = q(H_0 -1) + 1$. Since $1 - \frac{1}{2q} < H_0 < 1$, note that the parameter $H$ belongs to $(\frac{1}{2}, 1)$ for all $q \geq 1$. 

The Hermite process of order $q=1$ is nothing but the fBm; it is the only Hermite process to be defined for any value of $H \in (0, 1]$. The Hermite process of order $q=2$ is called \textit{the Rosenblatt process}; it was introduced in Rosenblatt \cite{dRosenblatt} but its current name comes from Taqqu \cite{dTaqqu1975}.

Recently, Hermite processes have received a lot of attention, not only from a theoretical point of view but also because of their great potential for applications. We would liek to highlight the following references.
\begin{enumerate}
\item In Tudor and Viens \cite{dTudorViens} and Chronopoulou, Tudor and Viens \cite{dChronopoulouTudorViens2009}, the authors constructed strong consistent statistical estimators for the self-similar parameter of the Rosenblatt process, by means of discrete observations after a careful analysis of the asymptotic behavior of its quadratic variations. Later, Chronopoulou, Tudor and Viens \cite{dChronopoulouTudorViens2011} extended the study in \cite{dTudorViens} to cover the case of all Hermite processes.
\item Maejima and Tudor \cite{dMaejimaTudor} introduced Wiener-It\^{o} integrals with respect to the Hermite process. As an application, they studied stochastic differential equations with this process as driving noise. They proved the existence and investigated some properties of the so-called Hermite Ornstein-Uhlenbeck process, which is a natural generalization of the celebrated fractional Ornstein-Uhlenbeck process.
\item Bertin, Torres and Tudor \cite{dBertinTorresTudor} were among the first to do some statistical inference for a model involving the Rosenblatt process. They constructed a strong consistent maximum likelihood estimator for the drift parameter. To do so, they used a method based on the random walk approximation of the Rosenblatt process. 
\end{enumerate}   

\subsection{Hermite processes viewed as multiple Wiener-It\^{o} integrals}

We now define Hermite processes by means of their time-indexed representation. We only focus on the definition and properties that will be needed throughout this thesis. For an in-depth introduction to Hermite processes, we refer the reader to the recent book by Tudor \cite{dTudorbook}. 

Let $ q \geq 1$ be an integer. Denote by $B = (B_t)_{t\in \mathbb{R}}$ a two-sided Brownian motion defined on some probability space $(\Omega, \mathcal{F}, P)$. The $q$-th \textit{multiple Wiener-It\^{o} integral} of kernel $f \in L^2(\mathbb{R}^q)$ with respect to $B$ is written in symbols as
\begin{equation}\label{deq:multipleintegral}
I_q^B(f) = \int_{\mathbb{R}^q} f(\xi_1, \ldots, \xi_q)dB_{\xi_1} \ldots dB_{\xi_q}.
\end{equation}
For the construction of (\ref{deq:multipleintegral}) and its main properties, we refer the reader to \cite{dNourdinbook} or \cite{dNualartbook}. Here, we only point out some basic facts. For any $f \in L^2(\mathbb{R}^q)$ and $g \in L^2(\mathbb{R}^p)$, we have $E[I_q^B(f) ] = 0$ and
\begin{equation}\label{deq:orthogonal}
E[I_q^B(f)I_p^B(g)] =
\begin{cases}
 & q! \big\langle \widetilde{f}, \widetilde{g} \big\rangle_{L^2(\mathbb{R}^p)} \qquad\text{if } q=p\\
& 0 \qquad\qquad\qquad\quad \text{ if } q \ne p,
\end{cases}
\end{equation}
where $\widetilde{f}$ is the symmetrization of $f$ defined by
$$\widetilde{f}(\xi_1, \ldots, \xi_q) = \frac{1}{q!}\sum_{\sigma \in \mathfrak{S}_q} f(\xi_{\sigma(1)}, \ldots, \xi_{\sigma(q)}).$$
Furthermore, $I_q^B(f)$ satisfies the so-called \textit{hypercontractivity} property:
 \begin{equation}\label{dhypercontractivity1}
E[|I_q^B(f)|^k]^{1/k} \leq (k-1)^{q/2}E[|I_q^B(f)|^2]^{1/2} \text{ for any } k \in [2, \infty).
 \end{equation}
The set $\mathscr{H}_q^B$ of random variables of the form $I_q^B(f), f \in L^2(\mathbb{R}^q)$, is called the $q$-th \textit{Wiener chaos} associated with $B$. As a convention, we set $\mathscr{H}_0^B = \mathbb{R}$.  

\begin{Definition}\label{dDef1.1.1}
Let $(B_t)_{t \in \mathbb{R}}$ be a two-sided standard Brownian motion. \textit{The Hermite process} $(Z_t^{q, H})_{t \geq 0}$ of order $q \geq 1$ and self-similarity parameter $H \in (\frac{1}{2}, 1)$ is defined as
\begin{equation}\label{dH1}
Z^{q, H}_t= c(H, q) \int_{\mathbb{R}^q} \bigg( \int_0^t \prod_{j=1}^q(s- \xi_j)_+^{H_0 - \frac{3}{2}}ds\bigg) dB_{\xi_1}\ldots dB_{\xi_q},
\end{equation}
where 
\begin{equation}\label{dH2}
c(H, q) = \sqrt{\frac{H(2H - 1)}{q! \beta^q(H_0 - \frac{1}{2}, 2-2H_0)}} \quad \text{and} \quad H_0 = 1+\frac{H-1}{q} \in \Big(1-\frac{1}{2q}, 1\Big).
\end{equation}
The integral (\ref{dH1}) is a multiple Wiener-It\^{o} integral of order $q$ with respect to the Brownian motion $B$, as considered in (\ref{deq:multipleintegral}). The positive constant $c(H, q)$ in (\ref{dH2}) is calculated to ensure that $E[(Z_1^{q, H})^2] = 1$. A random variable which has the same law as $Z^{q, H}_1$ is called a \textit{Hermite random variable}.
\end{Definition}

\subsection{Basic properties of Hermite processes}

Apart for Gaussianity, Hermite processes of any order $q \geq 2$ share many basic properties with the fractional Brownian motion. We make this statement more precise in the following result.

\begin{Proposition} Let $Z^{q, H}$ be a Hermite process of order $q \geq 1$ and self-similarity parameter $H \in (\frac{1}{2}, 1)$. Then,
\begin{enumerate}[(i)] 
\item  $[$ Self-similarity $]$  For all  $c > 0, (Z^{q, H}_{ct})_{t \geq 0} \overset{law}{=}  (c^H Z^{q, H}_t)_{t \geq 0}$. 
\item $[$ Stationarity of increments $]$ For any $h >0$, $(Z^{q, H}_{t+h} - Z^{q, H}_h)_{t \geq 0} \overset{law}{=} (Z^{q, H}_t)_{t \geq 0}$.
\item $[$ Covariance function $]$ For all $s, t \geq 0$, $E[Z_t^{q, H}Z_s^{q,H}]= \frac{1}{2}(t^{2H} + s^{2H} - |t-s|^{2H})$.
\item  $[$ Long-range dependence $]$ $\sum_{n \geq 1}|E[Z_1^{q, H}(Z_{n+1}^{q, H} - Z_n^{q, H})]| = \infty$.
\item $[$ H\"{o}lder continuity $]$ For any $\beta \in (0, H)$, Hermite process $Z^{q, H}$ admits a version with H\"{o}lder continuous sample paths of order $\beta$ on any compact interval.
\item $[$ Finite moments $]$ For every $p \geq 1, t\geq 0$, $E[|Z^{q, H}_t|^p] \leq C_{p, q} t^{pH}$, where $C_{p,q}$ is a positive constant depending on $p$ and $q$.
\end{enumerate}
\end{Proposition}

\begin{proof}
Point $(i)$ follows from the self-similarity of $B$ with index $1/2$, that is, $dB_{c\xi}$ has the same law as $c^{1/2}dB_\xi$ for all $c > 0$. Indeed, as a process,
\begin{align*}
Z^{q, H}_{ct}&= c(H, q) \int_{\mathbb{R}^q} \bigg( \int_0^{ct} \prod_{j=1}^q(s- \xi_j)_+^{H_0 - \frac{3}{2}}ds\bigg) dB_{\xi_1}\ldots dB_{\xi_q}\\ 
& = c(H, q) \int_{\mathbb{R}^q} \bigg( c \int_0^{t} \prod_{j=1}^q(cs- \xi_j)_+^{H_0 - \frac{3}{2}}ds\bigg) dB_{\xi_1}\ldots dB_{\xi_q}\\
&= c(H, q) \int_{\mathbb{R}^q} \bigg( c \int_0^{t} \prod_{j=1}^q(cs- c\xi_j)_+^{H_0 - \frac{3}{2}}ds\bigg) dB_{c\xi_1}\ldots dB_{c\xi_q}\\
&\overset{(d)}{=} c c^{q(H_0-3/2)}c^{q/2} c(H, q) \int_{\mathbb{R}^q} \bigg( c \int_0^{t} \prod_{j=1}^q(s- \xi_j)_+^{H_0 - \frac{3}{2}}ds\bigg) dB_{\xi_1}\ldots dB_{\xi_q}\\
&= c^H Z^{q, H}_t.
\end{align*}
Point $(ii)$ is as a consequence of the definition (\ref{dH1}) of Hermite process. In fact, for any $h >0$ we have, as a process,
{\allowdisplaybreaks
\begin{align*}
Z^{q, H}_{t+h} - Z^{q, H}_h & = c(H, q) \int_{\mathbb{R}^q} \bigg( \int_h^{t+h} \prod_{j=1}^q(s- \xi_j)_+^{H_0 - \frac{3}{2}}ds\bigg) dB_{\xi_1}\ldots dB_{\xi_q} \\
& = c(H, q) \int_{\mathbb{R}^q} \bigg( \int_0^t \prod_{j=1}^q(s + h- \xi_j)_+^{H_0 - \frac{3}{2}}ds\bigg) dB_{\xi_1}\ldots dB_{\xi_q} \\
& = c(H, q) \int_{\mathbb{R}^q} \bigg( \int_0^t \prod_{j=1}^q(s - \xi_j)_+^{H_0 - \frac{3}{2}}ds\bigg) dB_{\xi_1 + h}\ldots dB_{\xi_q + h} \\
& \overset{(d)}{=} c(H, q) \int_{\mathbb{R}^q} \bigg( \int_0^t \prod_{j=1}^q(s - \xi_j)_+^{H_0 - \frac{3}{2}}ds\bigg) dB_{\xi_1}\ldots dB_{\xi_q } \\
& = Z^{q, H}_t.
\end{align*}
}Furthermore, all self-similar processes with stationary increments have the same covariance function, see e.g., \cite[Prop. A.1]{dTudorbook}, which is given by
$$E[Z_t^{q, H}Z_s^{q,H}]= \frac{1}{2}E[(Z^{q, H}_1)^2](t^{2H} + s^{2H} - |t-s|^{2H}), \qquad t, s \geq 0.$$
Since $E[(Z^{q, H}_1)^2] = 1$, Point $(iii)$ is proved. For any integer $n \geq 1$, we compute from $(iii)$ that
\begin{align*}
|E[Z_1^{q, H}(Z_{n+1}^{q, H} - Z_n^{q, H})]|& = \Big|\frac{1}{2} \big((n+1)^{2H} + (n-1)^{2H} -2n^{2H} \big)\Big|\\ 
& \sim H(2H-1) n^{2H-2}.
\end{align*}
Since $H > \frac{1}{2}$, the Hermite process $Z^{q, H}$ exhibits long-range dependence $(iv)$. We now turn to the proofs of $(v)$ and $(vi)$. From $(iii)$ and the hypercontractivity property (\ref{dhypercontractivity1}), it comes that, for any $p \geq 1$,
$$E[|Z_t^{q, H} - Z_s^{q, H}|^p] \leq C_{p, q} \big(E[(Z_t^{q, H} - Z_s^{q, H})^2] \big)^{\frac{p}{2}}= C_{p, q} |t-s|^{pH}.$$
It follows from Kolmogorov's continuity criterion that $Z^{q, H}$ admits a version with H\"{o}lder continuous sample paths of any order $\beta$  with $0 < \beta < H$, which proves the point $(v)$. Furthermore, it also proves $(vi)$. 
\end{proof}

\subsection{Two further stochastic representations of Hermite processes}

Hermite processes can be represented as multiple Wiener-It\^{o} integrals in at least three different ways. 

The first one is given by (\ref{dH1}); it is the \emph{time-indexed representation}, supported on the real line and in the time domain.

The second one is the \emph{spectral representation} on the real line. It was obtained by Taqqu \cite{dTaqqu1979}; his finding is that, as a process,
 \begin{equation}\label{dH3}
Z_t^{q,H}\overset{(d)}{=} A(q, H) \int_{\mathbb{R}^q} \frac{e^{(\lambda_1+\ldots + \lambda_q)t} -1}{i(\lambda_1+\ldots + \lambda_q)} |\lambda_1|^{\frac{1}{2}-H_0} \ldots |\lambda_q|^{\frac{1}{2}-H_0} W(d\lambda_1) \ldots dW(d\lambda_q),
\end{equation}
where $W$ is a Gaussian complex-valued random spectral measure, $H_0$ is given by (\ref{dH2}) and
$$A(q, H):= \bigg( \frac{H(2H-1)}{q![2\Gamma(2-2H_0) \sin (H_0 - \frac{1}{2})\pi]^q} \bigg).$$

Finally, we introduce the \emph{time interval representation}. It turns out to be of particular interest when we want to simulate $Z^{q, H}$ or when we aim to construct a stochastic calculus with respect to it, see e.g., \cite{dChronopoulouTudorViens2011, dTudorViens}. In the case of fractional Brownian motion $(q=1)$, it is well-known that
$$Z^{1, H}_t \overset{(d)}{=} \int_0^t K^H(t, s)dB_s,$$
with $(B_s)_{t  \geq 0}$ a standard Brownian motion,
\begin{equation}\label{deq:KH}
K^H(t, s) = c_H s^{\frac{1}{2} - H} \int_s^t  (u-s)^{H -\frac{3}{2}}u^{H -\frac{1}{2}}du, \qquad t > s,
\end{equation}
and $c_H = \big( \frac{H(H-1)}{\beta(2-2H, H-\frac{1}{2})}\big)^{\frac{1}{2}}$. The time interval representation of the Hermite process makes also use of the kernel $K^H$ given by (\ref{deq:KH}). More precisely, it was shown in Pipiras and Taqqu \cite{dPipirasTaqqu2} that, as a process,
\begin{align}\label{H4}
Z_t^{q,H}&\overset{(d)}{=} b_{q, H} \int_0^t \ldots \int_0^t  \bigg(\int_{u_1 \vee \ldots \vee u_q}^t  \partial_1K^{H_0}(s, u_1) \ldots  \partial_1K^{H_0}(s, u_q) ds\bigg)dB_{u_1} \ldots dB_{u_q} \nonumber\\
& =  b_{q, H} \int_0^t \ldots \int_0^t \bigg( \prod_{i=1}^q u_i^{\frac{1}{2} - H_0} \int_0^t s^{q\big(H_0 -\frac{1}{2}\big)} \prod_{i=1}^q (s-u_i)_+^{H_0 -\frac{3}{2}} ds\bigg) dB_{u_1} \ldots dB_{u_q},
\end{align}
where the positive constant $b_{q, H}$ is chosen so that $E[(Z_1^{q, H})^2] = 1$ and $H_0$ given by (\ref{dH2}). 

\subsection{Wiener integrals with respect to Hermite process}

We now introduce Wiener integrals of a deterministic function with respect to the Hermite process, following the construction done in Maejima and Tudor \cite{dMaejimaTudor}. Due to the equivalence in distribution of the three previous stochastic representations for Hermite processes, we can choose the one we want. In the sequel, we deal with the representation (\ref{dH1}). 

Firstly, let $f$ be an elementary function on $\mathbb{R}$ of the form
$$f(u) = \sum_{j=1}^n a_j \mathbbm{1}_{(t_j, t_{j+1}]}(u).$$
We naturally define the Wiener integral of $f$ with respect to $Z^{q, H}$ as
$$ \int_{\mathbb{R}} f(u) dZ^{q, H}_u = \sum_{j=1}^n a_j (Z^{q, H}_{t_{j+1}} - Z^{q, H}_{t_j}).$$
Observe that the Hermite process given by formula (\ref{dH1}) can equivalently be written this way: 
$$Z^{q, H}_t = \int_{\mathbb{R}^q} I(\mathbbm{1}_{[0, t]})(y_1, \ldots, y_q)dB(y_1) \ldots dB(y_q)$$
where $B$ is a two-sided standard Brownian motion and $I$ is the mapping from the set of functions $f: \mathbb{R} \to \mathbb{R}$ to the set of functions $g: \mathbb{R}^q \to \mathbb{R}$ given by
$$I(f)(y_1, \ldots, y_q): = c(H, q)\int_{\mathbb{R}}f(u) \prod_{i=1}^q (u-y_i)_+^{H_0-\frac{3}{2}}du$$
with $c(H, q)$ and $H_0$ defined as in (\ref{dH2}). It follows that the Wiener integral of $f$ with respect to $Z^{q, H}$ can be expressed as the following $q$-th multiple Wiener integral
\begin{equation}\label{deq:10}
\int_{\mathbb{R}} f(u) dZ^{q, H}_u = \int_{\mathbb{R}^q} I(f)(y_1, \ldots, y_q) dB(y_1) \ldots dB(y_q).
\end{equation}
For every step function $f$, it is easily seen that $E\Big[\int_{\mathbb{R}} f(u) dZ^{q, H}_u \Big] = 0$ and
\begin{align*}
E\bigg[\bigg(\int_{\mathbb{R}} f(u) dZ^{q, H}_u \bigg)^2 \bigg] & = q! \int_{\mathbb{R}^q} (I(f)(y_1, \ldots, y_q))^2 dy_1 \ldots dy_q \\ 
& = H(2H-1) \int_{\mathbb{R}}\int_{\mathbb{R}} f(u)f(v) |u-v|^{2H-2}dudv.
\end{align*}
Let us now introduce the linear space $\mathcal{H}$ of measurable functions $f$ on $\mathbb{R}$ such that
$$||f||_{\mathcal{H}}^2 := q! \int_{\mathbb{R}^q} (I(f)(y_1, \ldots, y_q))^2 dy_1 \ldots dy_q  < \infty.$$
It is immediate to compute that
$$||f||_{\mathcal{H}}^2 = H(2H-1)\int_{\mathbb{R}}\int_{\mathbb{R}} f(u)f(v)|u - v|^{2H-2} < \infty.$$
Observe that the mapping
\begin{equation}\label{deq:isometry}
f \longmapsto \int_{\mathbb{R}}f(u)dZ_u^{q, H}
\end{equation}
is an isometry from the space of elementary functions equipped with the norm $\|.\|_{\mathcal{H}}$ to $L^2(\Omega)$. Furthermore, it was proved in \cite{dPipirasTaqqu} that, for every $f\in \mathcal{H}$, there exists a sequence of step functions $(f_n)_{n \ge 1}$ in $\mathcal{H}$ such that $f_n \to f$ in $\mathcal{H}$. For each $n$, the integral $\int_{\mathbb{R}}f_n(u)dZ^{q, H}_u$ is well-defined and, for all $n , m \geq 0$, one has
\begin{align*}
&E\bigg[\bigg(\int_{\mathbb{R}}f_n(u)dZ^{q, H}_u -\int_{\mathbb{R}}f_m(u)dZ^{q, H}_u \bigg)^2\bigg]\\
&=E\bigg[\bigg(\int_{\mathbb{R}}(f_n-f_m)(u)dZ^{q, H}_u\bigg)^2\bigg]=||f_n - f_m||_{\mathcal{H}}^2 \xrightarrow{m, n \to \infty} 0.
\end{align*}
Hence $\Big\{\int_{\mathbb{R}}f_n(u)dZ^{q, H}_u\Big\}_{n \ge 1}$ is a Cauchy sequence in $L^2(\Omega)$ and thus admits a limit. It allows one to define the Wiener integral of any deterministic functions in the space $\mathcal{H}$ with respect to the Hermite process $Z^{q, H}$ as
$$ \int_{\mathbb{R}}f(u)dZ^{q, H}_u = \lim_{n\to \infty}\int_{\mathbb{R}}f_n(u)dZ^{q, H}_u.$$
By construction, the isometry mapping (\ref{deq:isometry}) as well as the relation (\ref{deq:10}) still hold for any function in $\mathcal{H}$.

\subsection{A particular case: the Rosenblatt process}\label{Rosenblattsection}

The Rosenblatt process, usually denoted by $R^H$ in the litterature, is the other name given to the Hermite process of order $q=2$. For a given $H \in (\frac{1}{2}, 1)$, according to Definition \ref{dDef1.1.1} it is defined as follows:
\begin{equation}\label{dR1}
R^H_t= c(H, 2) \int_{\mathbb{R}^q} \bigg( \int_0^t (s- \xi_1)_+^{H_0 - \frac{3}{2}}(s- \xi_2)_+^{H_0 - \frac{3}{2}}ds\bigg) dB_{\xi_1}dB_{\xi_2},
\end{equation}
where $B$ is a standard Brownian motion on $\mathbb{R}$, and where the positive constants $C(H, 2)$ and $H_0$ are defined by (\ref{dH2}). This stochastic process is $H$-self-similar with stationary increments, exhibits long-range dependence and lives in the second Wiener chaos. As a result, it is not a Gaussian process. In the last few years, the Rosenblatt process has been studied a lot. Among others, we would like to mention several papers related to some topics of interest in this thesis: Tudor \cite{dTudor2008}, Pipiras and Taqqu \cite{dPipirasTaqqu2}, Tudor and Viens \cite{dTudorViens}, Veillette and Taqqu \cite{dVeilletteTaqqu}, Maejima and Tudor \cite{dMaejimaTudor2013}. In the sequel, we discuss more closely about Rosenblatt distribution and the finite time interval representation of a Rosenblatt process. 

The Rosenblatt distribution is the marginal distribution of $R^H_t$ evaluated at time $t = 1$, i.e., the distribution of the Rosenblatt random variable $R^H_1$. Using Monte Carlo simulations, Torres and Tudor \cite{dTorresTudor} have been able to draw empirical histograms for the density of the Rosenblatt distribution, see Figure \ref{fig:2} below.
\begin{figure}[H]
\centering
\includegraphics[height=3.5cm,width=12cm]{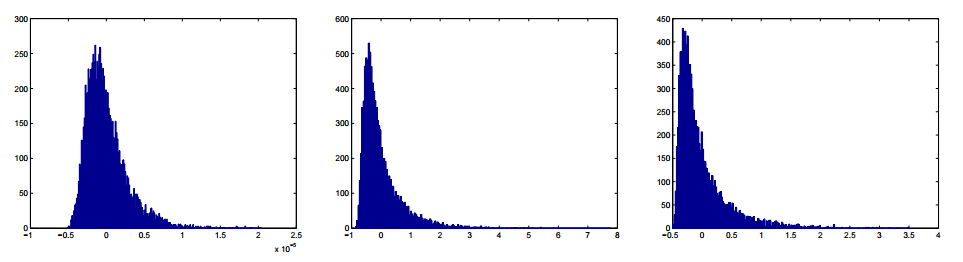}
\caption{Density of the Rosenblatt distribution for $H = 0.5, H = 0.8$ and $H = 0.9$.}\label{fig:2}
\end{figure}

Furthermore, the authors of \cite{dTorresTudor} simulated some sample paths of the Rosenblatt process for different values of the parameter $H$, see Figure \ref{fig:3}.
\begin{figure}[H]
\centering
\includegraphics[height=4.2cm,width=12cm]{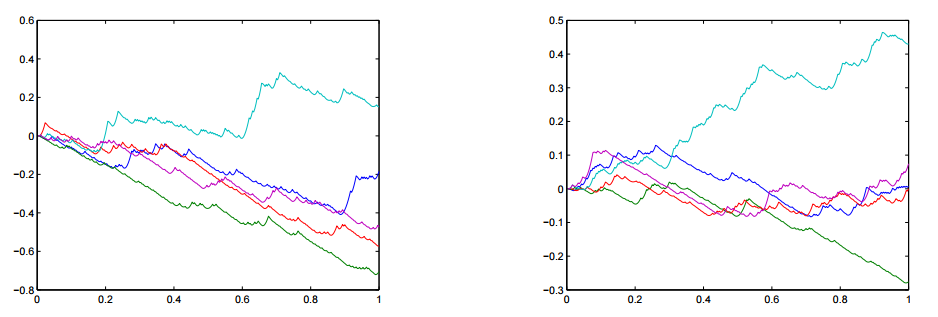}
\caption{Simulations of sample trajectories for the Rosenblatt process with $H = 0.8$ (left) and $H = 0.9$ (right).}\label{fig:3}
\end{figure}

Since computing an explicit expression for the density function of the Rosenblatt random variable is still an open problem, Veillette and Taqqu \cite{dVeilletteTaqqu} developed a technique to evaluate it numerically. The authors plotted the PDF and CDF of the Rosenblatt distribution shown in Figure \ref{fig:1}.
\begin{figure}[H] 
\centering
\includegraphics[height=7cm,width=8cm]{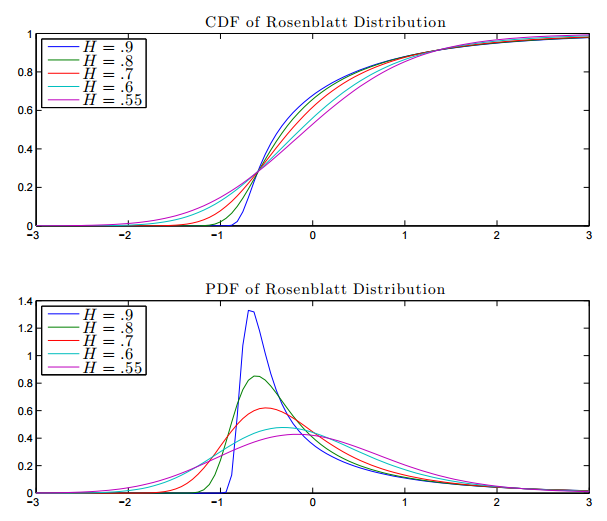}
\caption{CDF and PDF of Rosenblatt distribution.} \label{fig:1}
\end{figure}

A careful look at the CDF of the Rosenblatt distribution $R_1^H$ given in Figure \ref{fig:1} leads to the natural but very mysterious conjecture (borrowed from Taqqu \cite{dTaqqu2014}) that, whatever the value of $H$, 
\begin{align} \label{dconjecture}
P(R_1^H \leq -0.6256)& = 0.2658 \nonumber\\ 
P(R_1^H \leq 1.3552) &= 0.9123.
\end{align}
To understand and prove (\ref{dconjecture}) is still an open problem, the main obstacle being the lack of an explicit expression for the density of $R_1^H$.

Let us now turn to the finite time interval representation of the Rosenblatt process.
\begin{Proposition} (\cite[Prop. 3.7]{dTudorbook}) \label{dProp3.7}
Let $K^H$ be the kernel (\ref{deq:KH}) and let $(R^H_t)_{t \in [0, T]}$ be the Rosenblatt process given by (\ref{dR1}) with parameter $H \in (\frac{1}{2}, 1)$. Then
\begin{equation}\label{deq:R1equiv}
R^H_t \overset{(d)}{=} b_H \int_0^t \int_0^t  \bigg(\int_{u_1\vee u_2}^t  \partial_1K^{H_0}(s, u_1) \partial_1K^{H_0}(s, u_2) ds\bigg)dB_{u_1}dB_{u_2},
\end{equation}
where $B$ is a Brownian motion, $H_0 = \frac{H+1}{2}$ and $b_H = \frac{1}{H+1}\sqrt{\frac{2(2H-1)}{H}}$.
\end{Proposition}

Since the Rosenblatt process belongs to the second Wiener chaos, its distribution is characterized by its mixed \textit{cumulants}, see e.g., \cite[Prop. 2.7.13]{dNourdinPeccatibook}. We recall that, given a random variable $X$ such that $E[|X|^p] < \infty, \forall p \geq 1$, the sequence of the cumulants of $X$, denoted by $\kappa_p(X), p \geq 1$, is defined as follows
$$ \log E[e^{itX}] = \sum_{p=1}^\infty \frac{(it)^p}{p!} \kappa_p(X), \qquad t \in \mathbb{R}.$$
The first cumulant $\kappa_1$ is the mean and the second one $\kappa_2$ is the variance. Let us consider a double Wiener-It\^{o} integral $I_2^B(f)$ with $f \in L^2(\mathbb{R}^2)$ symmetric. Then, for all $p \geq 2$, the $p$-th cumulant of $I_2^B(f)$ can be easily computed as a circular integral of $f$, see \cite[Prop. 2.7.13]{dNourdinPeccatibook}:
\begin{equation}\label{deq:cummulants}
\kappa_p(I_2^B(f)) = 2^{p-1}(p-1)! \int_{\mathbb{R}^p} f(s_1, s_2)f(s_2, s_3) \ldots f(s_{p-1}, s_p) f(s_p, s_1)ds_1 \ldots ds_p.
\end{equation}

Note that the circular shape for the cumulants of double Wiener-It\^{o} integrals becomes wrong for higher order multiple Wiener-It\^{o} integrals.

\textit{Sketch of the proof of Proposition \ref{dProp3.7}:} We follows Tudor's arguments from \cite{dTudor2008} and make use of the cumulants. Let us denote by $R'^H_t$ the right-hand side of (\ref{deq:R1equiv}). Since the law of a double Wiener-It\^{o} integral is completely  determined by its cumulants (\ref{deq:cummulants}), we are left to show that
$$\sum_{i=1}^n b_i R^H_{t_i}, \quad \text{and} \quad  \sum_{i=1}^n b_i R'^H_{t_i}$$
share the same cumulants. We only consider the case $n=2$, because it is representative of the difficulty. More precisely, let us show that, for every $s, t \in [0, T]$ and $\alpha, \beta \in \mathbb{R}$, the random variables $\alpha R^H_t+ \beta R^H_s$ and $\alpha R'^H_t+ \beta R'^H_s$ have the same cumulants. We can write, for all $s, t \in [0, T]$,
$$\alpha R'^H_t+ \beta R'^H_s  = I^B_2(f_{s, t})$$
where
\begin{align*}
f_{t, s}(y_1, y_2)& = \alpha \mathbbm{1}_{[0, t]}(y_1) \mathbbm{1}_{[0, t]}(y_2) \int_{y_1 \vee y_2}^t \partial_1K^{H_0}(u, y_1) \partial_1K^{H_0}(u, y_2) du\\ 
& +  \beta \mathbbm{1}_{[0, s]}(y_1) \mathbbm{1}_{[0, s]}(y_2) \int_{y_1 \vee y_2}^s \partial_1K^{H_0}(u, y_1) \partial_1K^{H_0}(u, y_2) du,
\end{align*}
and 
$$\alpha R^H_t+ \beta R^H_s  = I^B_2(g_{s, t})$$
where
$$g_{s, t}(y_1, y_2) = c(H, 2) \bigg( \alpha \int_0^t (u- y_1)_+^{\frac{H}{2}-1}(u- y_2)_+^{\frac{H}{2} -1}du + \beta \int_0^s (u- y_1)_+^{\frac{H}{2}-1}(u- y_2)_+^{\frac{H}{2} -1}du  \bigg).$$
Following computations in \cite{dTudor2008} or \cite[Prop.3.7]{dTudorbook}, both random variables $I^B_2(f_{s, t})$ and $I^B_2(g_{s, t})$ share the same cumulants given by, for all $p \geq 2$,
\begin{align*}
\kappa_p(I^B_2(f_{s, t})) &= \kappa_p(I^B_2(g_{s, t})) \\ 
& = (p-1)!2^{p-1}b_H^p (H_0(2H_0 -1))^p \sum_{t_j \in \{t, s\}} \int_0^{t_1} \ldots \int_0^{t_p} du_1\ldots du_p \\
&\hspace{1.0cm} \times \alpha^{\sharp \{t_j=t\}} \beta^{\sharp \{t_j =s\}} |u_1-u_2|^{2H_0 -2}|u_2-u_3|^{2H_0 -2} \ldots |u_p-u_1|^{2H_0 -2}.
\end{align*}
This concludes our proof. \hspace{10cm} $\square $
%
%

\section{Multiparameter Hermite random fields}

\subsection{Where our interest for multiparameter Hermite random fields comes from} 

Multiparameter Hermite random fields (aka Hermite sheets) are a generalization of Hermite processes, but instead of a time interval we now deal with a subset of $\mathbb{R}^d$. The family of Hermite sheets share several properties with the family of Hermite processes, including self-similarity, stationary increments and H\"{o}lder continuity. Hermite sheet is parametrized by the order $q \geq 1$ and the self-similarity parameter $\HH = (H_1, \ldots, H_d)$. It includes the well-known fractional Brownian motion (if $q=1, d=1$) as well as the fractional Brownian sheet (if $q=1, d \geq 2$). These latter are the only Gaussian fields in the class of Hermite sheets. When $q =2$, it contains the Rosenblatt process (if $d=1$) and the Rosenblatt sheet (if $d \geq 2$). 

Hermite random fields have been introduced as limits of some Hermite variations of the fractional Brownian sheet. We refer the reader to \cite{dPakkanenReveillac1} or \cite{dReveillacStauchTudor}. Among various aspects of the fractional Brownian sheet, we focus here on the study of its weighted power variations. We start with some historical facts.
\begin{enumerate}
\item In Nourdin, Nualart and Tudor \cite{dNourdinNualartTudor}, see also the references therein, the authors gave a complete description of the convergence of normalized weighted power variations of the fractional Brownian motion for any Hurst parameter $H \in (0, 1)$.
\item R\'{e}veillac \cite{dReveillac1} proved the convergence in the sense of finite-dimensional distributions of the weighted quadratic variation of a two-parameter fractional Brownian sheet. Generalized results for any fractional Brownian sheet are announced in a work in progress by Pakkanen and R\'{e}veillac \cite{dPakkanenReveillac2} (private communication).
\item R\'{e}veillac, Stauch and Tudor \cite{dReveillacStauchTudor} proved central and non-central limit theorems for the Hermite variations of the two-parameter fractional Brownian sheet. Later, generalized variations of $d$-parameter fractional Brownian sheet were studied by Pakkanen and R\'{e}veillac \cite{dPakkanenReveillac1}. The multiparameter Hermite random field appeared in the limit of non-central limit theorems. Furthermore, in the case of non-central asymptotics, Breton \cite{dBreton} gave the rate of convergence for the Hermite variations of fractional Brownian sheet. The study of weighted power variations of fractional Brownian sheet is still an open problem. We have investigated it in our paper in progress \cite{dTran2} (not included in this thesis).
\end{enumerate}   

The study of power variations of multiparameter non-Gaussian Hermite random fields, including Hermite processes, has received less attention: see \cite{dChronopoulouTudorViens2011, dTudorViens} for quadratic variations of Hermite processes. Our main achievement on this aspect is to extend the result of \cite{dChronopoulouTudorViens2011} to the family of Hermite random \textit{fields}.

\subsection{Fractional Brownian sheet}

The fractional Brownian sheet (in short fBs) $B^\HH$ with Hurst parameter $\HH \in (0, 1)^d$ is one particular example of Hermite random fields. It can also be viewed as a generalization of the well-known fractional Brownian motion. In the sequel, we introduce the definition of $B^\HH$ as well as some of its basic properties. From now on, we fix $d \geq 1$ in $\mathbb{N}$.

\begin{Definition}
A $d$-parameter \textit{fractional Brownian sheet} $B^\HH = (B^{H_1, \ldots, H_d}_{t_1, \ldots, t_d})_{(t_1, \ldots, t_d) \in [0, \infty)^d}$ with Hurst indices $\HH = (H_1, \ldots,H_d) \in (0, 1)^d$ is a centered $d$-parameter Gaussian process whose covariance function is given by
$$E[B^\HH_\t B^\HH_\s] = \prod_{i=1}^d \frac{1}{2}(t_i^{2H_i} + s_i^{2H_i} - |t_i-s_i|^{2H_i}).$$
\end{Definition}

There also exists a version of fractional Brownian sheet $B^\HH$ whose covariance is defined by
$$E[B^\HH_\t B^\HH_\s] = \frac12 (\|\t \|^{2H} + \|\s \|^{2H} -\|\t -s\|^{2H} ),$$
where $\| . \|$ denotes the Euclidian norm, (see e.g., \cite{dAdler}).

When $H_1 = \ldots = H_d =\frac{1}{2}$, it is nothing but \textit{the Brownian sheet}, that is, a centered Gaussian process $(B_\t)_{\t \geq 0} = (B_{t_1, \ldots, t_d})_{(t_1, \ldots, t_d) \geq 0}$ with covariance
$$E[B_\t B_\s] = \prod_{i=1}^d (t_i \wedge s_i).$$

Note that the covariance structure of fBs is defined as the tensor product of the covariance of a fBm. Thanks to this fact, fBs shares some properties with fBm such as self-similarity, stationary increments and H\"{o}lder continuity. Precisely, the following proposition states what happens only for \textit{the two-parameter fractional Brownian sheet}, for the sake of simplicity.  

\begin{Proposition}
Let $B^{H_1, H_2}$ be a two-parameter fractional Brownian sheet with Hurst parameter $(H_1, H_2) \in (0, 1)^2$. Then,
\begin{enumerate}[(i)] 
\item  $[$ Self-similarity $]$  For all  $h, k > 0, (B^{H_1, H_2}_{hs, kt})_{s, t \geq 0} \overset{law}{=}  (h^{H_1}k^{H_2} B^{H_1, H_2}_{s, t})_{s, t \geq 0}$. 
\item $[$ Stationarity of increments $]$ For any $h, k >0,$
$$ (B^{H_1, H_2}_{s+h, t+k} - B^{H_1, H_2}_{h, t+k} - B^{H_1, H_2}_{s+h, k} + B^{H_1, H_2}_{h,k})_{s, t \geq 0} \overset{law}{=} (B^{H_1, H_2}_{s, t})_{s, t \geq 0}.$$
\item $[$ H\"{o}lder continuity $]$ The fBs $B^{H_1, H_2}$ admits a version with H\"{o}lder continuous sample paths of order $(\beta_1, \beta_2)$ on any compact set, for any $\beta_1 \in (0, H_1)$ and $ \beta_2 \in (0, H_2)$.
\end{enumerate}
\end{Proposition}


\subsection{Definition of multiparameter Hermite random fields}

We now introduce the definition of multiparameter \textit{Hermite} random fields, following Tudor \cite{dTudorbook}. 

Let $ q \geq 1$ be an integer. Denote by $B = (B_\t)_{\t \in \mathbb{R}^d}$ a Brownian sheet. The $q$-th \textit{multiple Wiener-It\^{o} integral} of kernel $f \in L^2((\mathbb{R}^d)^q)$ with respect to $B$ is written in symbols as
\begin{equation}\label{deq:multipleintegral2}
I_q^B(f) = \int_{(\mathbb{R}^d)^q} f(\uu_1, \ldots, \uu_q)dB_{\uu_1} \ldots dB_{\uu_q}.
\end{equation}
For the construction of (\ref{deq:multipleintegral2}) and its main properties, we refer the reader to \cite{dNualartbook} (chapter 1 therein) or \cite[Section 3]{dPakkanenReveillac1}. It is readily verified that, for any $f \in L^2((\mathbb{R}^d)^q)$ and $g \in L^2((\mathbb{R}^d)^p)$, we have $E[I_q^B(f) ] = 0$ and
\begin{equation}\label{deq:orthogonal}
E[I_q^B(f)I_p^B(g)] =
\begin{cases}
 & q! \big\langle \widetilde{f}, \widetilde{g} \big\rangle_{L^2((\mathbb{R}^d)^p)} \qquad\text{if } q=p\\
& 0 \qquad\qquad\qquad\quad \text{ if } q \ne p,
\end{cases}
\end{equation}
where $\widetilde{f}$ is the symmetrization of $f$ defined by
$$\widetilde{f}(\uu_1, \ldots, \uu_q) = \frac{1}{q!}\sum_{\sigma \in \mathfrak{S}_q} f(\uu_{\sigma(1)}, \ldots, \uu_{\sigma(q)}).$$

\begin{Definition}
Let $(B_\t)_{\t \in \mathbb{R}^d}$ be a standard Brownian sheet. \textit{The $d$-parameter Hermite random field} $(Z_\t^{q, \HH})_{\t \geq 0}$ of order $q \geq 1$ and self-similarity parameter $\HH = (H_1, \ldots, H_d) \in (\frac{1}{2}, 1)^d$ is defined as
\begin{align}\label{dDefHermiterandomfield}
Z^{q, \HH}_\t&= c_{q, \HH} \int_{(\mathbb{R}^{d})^{ q}} dB_{u_{1,1}, \ldots, u_{1,d}} \ldots dB_{u_{q,1}, \ldots, u_{q,d}} \nonumber\\ 
&\qquad\times \bigg(\int_0^{t_1} da_1 \ldots \int_0^{t_d} da_d \prod_{j=1}^q (a_1-u_{j,1})_+^{-(\frac{1}{2} + \frac{1-H_1}{q})}\ldots (a_d-u_{j,d})_+^{-(\frac{1}{2} + \frac{1-H_d}{q})}\bigg)\nonumber\\
& = c_{q, \HH}\int_{(\mathbb{R}^{d})^{q}}dB_{\uu_1} \ldots dB_{\uu_q} \int_0^\t d\aa \prod_{j=1}^q (\aa - \uu_j)_+^{-(\frac{1}{2} + \frac{1-\HH}{q})},
\end{align}
where $x_+ = \max(x, 0)$, and $c(q, \HH)$ is the unique positive constant depending only on $q$ and $\HH$ chosen so that $E[Z^{q, \HH}(\textbf{1})^2] = 1$. The integral (\ref{dDefHermiterandomfield}) is a $q$-th multiple Wiener-It\^{o} integral of order $q$ with respect to the Brownian sheet $B$, as considered in (\ref{deq:multipleintegral2}). 
\end{Definition}

\subsection{Basic properties of multiparameter Hermite random fields}

Similarly as Hermite processes, apart for Gaussianity the multiparameter Hermite random fields of any order $q \geq 2$ share most of the basic properties of the fractional Brownian sheet. Let us make this statement more precise.

\begin{Proposition} Fix an integer $d \geq 1$. Let $Z^{q, \HH}$ be a $d$-parameter Hermite random field of order $q \geq 1$ and self-similarity parameter $\HH \in (\frac{1}{2}, 1)^d$. Then,
\begin{enumerate}[(i)] 
\item  $[$ Self-similarity $]$  For all  $\cc > 0, (Z^{q, \HH}_{\cc\t})_{\t \geq 0} \overset{law}{=}  (\cc^\HH Z^{q, \HH}_\t)_{\t \geq 0}$. 
\item $[$ Stationarity of increments $]$ For any $\hh >0, \hh \in \mathbb{R}^d$, $ \Delta Z^{q, \HH}_{[\hh, \hh + \t]} \overset{law}{=} \Delta Z^{q, \HH}_{[0, \t]}$, where $\Delta Z^{q, \HH}_{[\s, \t]}$ denotes the increment of $Z^{q, \HH}$ given by
\begin{equation}\label{deq:generalincrements}
\Delta Z^{q, \HH}_{[\s, \t]} = \sum_{\rr \in \{0, 1\}^d} (-1)^{d-\sum_{i=1}^d r_i} Z^{q, \HH}_{\s + \rr \times (\t -\s)}.
\end{equation}
\item $[$ Covariance function $]$ For all $\s, \t \geq 0$,
$$E[Z_\t^{q, \HH}Z_\s^{q,\HH}]= \prod_{i=1}^d \frac{1}{2}(t_i^{2H_i} + s_i^{2H_i} - |t_i-s_i|^{2H_i}).$$
\item $[$ H\"{o}lder continuity $]$ Hermite process $Z^{q, \HH}$ admits a version with H\"{o}lder continuous sample paths of order $\beta = (\beta_1, \ldots, \beta_d)$ for any $\beta \in (0, \HH)$.
\end{enumerate}
\end{Proposition}
Observe that (\ref{deq:generalincrements}) reduces to $\Delta Z^{q, H}_{[s, t]} = Z^{q, H}_t - Z^{q, H}_s$ when $d=1$, and to $\Delta Z^{q, H}_{[\s, \t]} = Z^{q, H}_{t_1, t_2} - Z^{q, H}_{t_1, s_2} - Z^{q, H}_{s_1, t_2} + Z^{q, H}_{s_1, s_2}$ when $d=2$.
\begin{proof}
Point $(i)$ follows from the self-similarity of the Brownian sheet $B$ with index $1/2$, that is, $dB_{\cc \t}$ has the same law as $\cc^{1/2}dB_\t$ for all $\cc = (c_1, \ldots, c_d) > 0$. Indeed,
{\allowdisplaybreaks
\begin{align*}
Z^{q, \HH}_{\cc \t}&= c_{q, \HH} \int_{(\mathbb{R}^d)^q} \bigg( \int_0^{\cc \t} \prod_{j=1}^q (\s- \yy_j)_+^{-\big(\frac{1}{2} + \frac{1-\HH}{q}\big)}d\s \bigg) dB_{\yy_1}\ldots dB_{\yy_q}\\ 
& = c_{q, \HH} \int_{\mathbb{R}^q} \bigg( \cc \int_0^{\t} \prod_{j=1}^q(\cc \s- \yy_j)_+^{-\big(\frac{1}{2} + \frac{1-\HH}{q}\big)}d\s \bigg) dB_{\yy_1}\ldots dB_{\yy_q}\\
&= c_{q, \HH}\int_{\mathbb{R}^q} \bigg( \cc \int_0^{\t} \prod_{j=1}^q(\cc \s- \cc\yy_j)_+^{-\big(\frac{1}{2} + \frac{1-\HH}{q}\big)}d\s \bigg) dB_{\cc \yy_1}\ldots dB_{\cc \yy_q}\\
&\overset{(d)}{=} \cc \cc^{-q(\frac{1}{2} + \frac{1-\HH}{q})}\cc^{q/2} c_{q, \HH} \int_{\mathbb{R}^q} \bigg( \cc \int_0^{t} \prod_{j=1}^q(\s- \yy_j)_+^{-\big(\frac{1}{2} + \frac{1-\HH}{q}\big)} d\s\bigg) dB_{\yy_1}\ldots dB_{\yy_q}\\
&= \cc^\HH Z^{q, \HH}_\t.
\end{align*}
}To prove point $(ii)$, we deal with the increments of $Z^{q, \HH}$ and then use the change of variable $\s' = \s - \hh$. For the sake of simplicity, we will only check the case $d=2$. In fact, for any $h_1, h_2  >0$ we have
\begin{align*}
\Delta Z^{q, \HH}_{[\hh, \hh + \t]}& = Z^{q, H_1, H_2}_{h_1+t_1, h_2 + t_2} -   Z^{q, H_1, H_2}_{h_1+t_1, h_2} -  Z^{q, H_1, H_2}_{h_1, h_2 + t_2} +  Z^{q, H_1, H_2}_{h_1, h_2}\\ 
& = c_{q, \HH} \int_{(\mathbb{R}^2)^q} dB_{y_{1, 1}, y_{1, 2}} \ldots dB_{y_{q, 1}, y_{q, 2}} \\
& \qquad \times \bigg( \int_{h_1}^{h_1+t_1}ds_1 \int_{h_2}^{h_2+t_2}ds_2 \prod_{j=1}^q (s_1 - y_{j,1})_+^{-(\frac{1}{2} + \frac{1-H_1}{q})} (s_2 - y_{j,2})_+^{-(\frac{1}{2} + \frac{1-H_2}{q})} \bigg).
\end{align*}
{\allowdisplaybreaks
The change of variables $s'_1 = s_1 -h_1, s'_2 = s_2-h$ gives
\begin{align*}
\Delta Z^{q, \HH}_{[\hh, \hh + \t]}& = c_{q, \HH} \int_{(\mathbb{R}^2)^q} dB_{y_{1, 1}, y_{1, 2}} \ldots dB_{y_{q, 1}, y_{q, 2}} \\
&\times \bigg( \int_0^{t_1}ds_1 \int_0^{t_2}ds_2 \prod_{j=1}^q (s_1 +h_1 - y_{j,1})_+^{-(\frac{1}{2} + \frac{1-H_1}{q})} (s_2 +h_2 - y_{j,2})_+^{-(\frac{1}{2} + \frac{1-H_2}{q})} \bigg)\\
& = c_{q, \HH} \int_{(\mathbb{R}^2)^q} dB_{y_{1, 1} + h_1, y_{1, 2}+h_2} \ldots dB_{y_{q, 1}+h_1, y_{q, 2}+h_2} \\
&\qquad\times \bigg( \int_0^{t_1}ds_1 \int_0^{t_2}ds_2 \prod_{j=1}^q (s_1  - y_{j,1})_+^{-(\frac{1}{2} + \frac{1-H_1}{q})} (s_2  - y_{j,2})_+^{-(\frac{1}{2} + \frac{1-H_2}{q})} \bigg)\\
& \overset{(d)}{=} c_{q, \HH} \int_{(\mathbb{R}^2)^q} dB_{y_{1, 1}, y_{1, 2}} \ldots dB_{y_{q, 1}, y_{q, 2}} \\
&\qquad\times \bigg( \int_0^{t_1}ds_1 \int_0^{t_2}ds_2 \prod_{j=1}^q (s_1  - y_{j,1})_+^{-(\frac{1}{2} + \frac{1-H_1}{q})} (s_2  - y_{j,2})_+^{-(\frac{1}{2} + \frac{1-H_2}{q})} \bigg)\\
& = \Delta Z^{q, \HH}_{[0, \t]}.
\end{align*}
}For point $(iii)$, we refer the reader to \cite[Chapter 4]{dTudorbook} for the details of calculation. From $(i)$ and $(ii)$, we have for all $p \geq 2$,
$$E[|\Delta Z^{q, \HH}_{[\s, \t]}|^p] = E[|Z^{q, \HH}_{\textbf{1}}|^p] |t_1-s_1|^{pH_1} \ldots |t_d-s_d|^{pH_d}.$$
Applying Kolmogorov's criterion for Wiener random fields depending on several parameters, $Z^{q, \HH}$ admits a version with H\"{o}lder continuous sample paths of any order $\beta = (\beta_1, \ldots, \beta_d)$  with $0 < \beta_j < H_j$ for $j=1, \ldots, d$, which proves the point $(iv)$.
\end{proof}

\subsection{A further stochastic representation of Hermite random fields}

We now introduce the \emph{finite-time representation} for the Hermite sheet $Z^{q, \HH}$. The equivalence in the sense of finite-dimensional distributions between $Z^{q, H}$ in (\ref{dDefHermiterandomfield}) and the following representation is shown in \cite[Chapter 4]{dTudorbook}:
\begin{align}\label{dDefHermiterandomfield2}
&Z^{q, \HH}(\t) \overset{(d)}{=} b_{q, \HH} \int_0^{t_1} \ldots \int_0^{t_d} dW_{u_{1,1}, \ldots, u_{1,d}} \ldots \int_0^{t_1}\ldots\int_0^{t_d} dW_{u_{q,1}, \ldots, u_{q,d}} \nonumber\\ 
&\hspace{2.5cm}\times \bigg(\int_{u_{1,1} \vee \ldots \vee u_{q,1}}^{t_1} da_1 \partial_1K^{H_1'}(a_1, u_{1,1}) \ldots \partial_1K^{H_1'}(a_1, u_{q,1})\bigg) \nonumber \\
& \hspace{6cm} \vdots \nonumber\\
& \hspace{2.5cm}\times \bigg(\int_{u_{1,d} \vee \ldots \vee u_{q,d}}^{t_d} da_d \partial_1K^{H_d'}(a_d, u_{1,d}) \ldots \partial_1K^{H_d'}(a_d, u_{q,d})\bigg) \nonumber\\
& = b_{q, \HH}\int_{[0, \t]^{q}}dW_{\uu_1}\ldots dW_{\uu_q} \prod_{j=1}^d\int_{u_{1,j} \vee \ldots \vee u_{q,j}}^{t_j} da \partial_1K^{H'_j}(a, u_{1,j}) \ldots \partial_1K^{H'_j}(a, u_{q,j}),
\end{align}
where $K^{H}$ stands for the usual kernel appearing in the classical expression of  the fractional Brownian motion $B^H$ as a Volterra integral with respect to Brownian motion given by (\ref{deq:KH}), the positive constant $b_{q, \HH}$ is chosen to ensure that $E[Z^{q, \HH}(\textbf{1})^2] = 1$ and $\HH' = 1+ \frac{\HH-1}{q}$.

\subsection{Wiener integrals with respect to Hermite random  fields}

We now introduce Wiener integrals of a deterministic function with respect to the $d$-parametric Hermite random field $(Z^{q, \HH}_\t)_{\t \in \mathbb{R}^d}$, following the construction done in Clarke De la Cerda and Tudor \cite{dClarkeDelaCerdaTudor}. When $d=1$, notice that we recover the construction of Wiener integrals with respect to Hermite processes.

Firstly, let $f$ be an elementary function on $\mathbb{R}^d$ of the form
$$f(\uu)  = \sum_{j=1}^n a_j \mathbbm{1}_{(\t_j, \t_{j+1}]}(\uu)= \sum_{j=1}^n a_j \mathbbm{1}_{(t_{1, j}, t_{1, j+1}] \times \ldots \times (t_{d, j}, t_{d, j+1}] }(u_1, \ldots, u_d).$$
We define naturally the Wiener integral of $f$ with respect to $Z^{q, \HH}$ as
$$ \int_{\mathbb{R}^d} f(\uu) dZ^{q, \HH}_\uu = \sum_{j=1}^n a_j \Delta Z^{q, \HH}_{[\t_j, \t_{j+1}]},$$
where $\Delta Z^{q, \HH}_{[\t_j, \t_{j+1}]}$ is the generalized increment of $Z^{q, \HH}$ given by (\ref{deq:generalincrements}).
Observe that the Hermite sheet given by formula (\ref{dDefHermiterandomfield}) can equivalently be written as follows
$$Z^{q, \HH}_\t = \int_{(\mathbb{R}^d)^q} J(\mathbbm{1}_{[0, t_1] \times \ldots \times [0, t_d]})(\yy_1, \ldots, \yy_q)dB(\yy_1) \ldots dB(\yy_q),$$
where $B$ is a Brownian sheet and $J$ is the mapping from the set of functions $f: \mathbb{R}^d \to \mathbb{R}$ to the set of functions $g: (\mathbb{R}^d)^q \to \mathbb{R}$ given by
\begin{align*}
J(f)(\yy_1, \ldots, \yy_q): &= c(\HH, q)\int_{\mathbb{R}^d}f(\uu) \prod_{i=1}^q (\uu-\yy_i)_+^{-\big(\frac{1}{2} + \frac{1-\HH}{q}\big)}d\uu\\
& = c(\HH, q)\int_{\mathbb{R}^d}f(u_1, \ldots, u_d) \prod_{i=1}^q \prod_{j=1}^d (u_j-y_{i, j})_+^{-\big(\frac{1}{2} + \frac{1-H_i}{q}\big)}du_1 \ldots u_d.
\end{align*}
It follows that the Wiener integral for step functions $f$ with respect to $Z^{q, \HH}$ can be expressed as the following $q$-th multiple Wiener integral
\begin{equation}\label{deq:101}
\int_{\mathbb{R}^d} f(\uu) dZ^{q, \HH}_\uu = \int_{(\mathbb{R}^d)^q} J(f)(\yy_1, \ldots, \yy_q) dB(\yy_1) \ldots dB(\yy_q).
\end{equation}
For every step function $f$, it is readily verified that $E\Big[\int_{\mathbb{R}^d} f(\uu) dZ^{q, \HH}_\uu \Big] = 0$ and
\begin{align*}
&E\bigg[\bigg(\int_{\mathbb{R}^d} f(\uu) dZ^{q, \HH}_\uu \bigg)^2 \bigg]  = q! \int_{(\mathbb{R}^d)^q} (J(f)(\yy_1, \ldots, \yy_q))^2 d\yy_1 \ldots d\yy_q \\ 
& = \HH(2\HH-1) \int_{\mathbb{R}^d}\int_{\mathbb{R}^d} f(\uu)f(\vv) |\uu-\vv|^{2\HH-2}d\uu d\vv\\
&= \int_{\mathbb{R}^d}\int_{\mathbb{R}^d} f(u_1, \ldots, u_d)f(v_1, \ldots, v_d) \prod_{i=1}^d H_i(2H_i-1) |u_i-v_i|^{2H_i-2} du_1 \ldots du_d dv_1 \ldots dv_d.
\end{align*}
Let us now introduce the linear space $\mathcal{H}$ of measurable functions $f$ on $\mathbb{R}^d$ such that
$$\|f\|_{\mathcal{H}}^2 := q! \int_{(\mathbb{R}^d)^q} (I(f)(\yy_1, \ldots, \yy_q))^2 d\yy_1 \ldots d\yy_q  < \infty.$$
Playing with the expression of the norm yields
$$\|f \|_{\mathcal{H}}^2 = \HH(2\HH-1)\int_{\mathbb{R}^d}\int_{\mathbb{R}^d} f(\uu)f(\vv)|\uu - \vv|^{2\HH-2}d\uu d\vv < \infty.$$
Observe that the mapping
\begin{equation}\label{deq:isometry2}
f \longmapsto \int_{\mathbb{R}^d}f(\uu)dZ_\uu^{q, \HH}
\end{equation}
is an isometry from the space of elementary functions equipped with the norm $\|.\|_{\mathcal{H}}$ to $L^2(\Omega)$. Furthermore, it was shown in \cite{dPipirasTaqqu} that the set of elementary functions is dense in $\mathcal{H}$. Then the isometry mapping (\ref{deq:isometry2}) and the relation (\ref{deq:101}) still hold for any function in $\mathcal{H}$.

\section{Introduction to Fisher information}

We now leave the world of Hermite processes and fields, to introduce the definitions of entropy and Fisher information for continuous random variables or vectors, two notions at the heart of the work in progress \cite{dTran3} (see Section \ref{Section1.4}). We then describe the relationships between the different induced forms of convergences. For the sake of simplicity, we first start with the one-dimensional case.
  
 \subsection{Entropy and Fisher information for real-valued random variables}

\begin{Definition}(\cite[Def. 1.4, 1.5]{dJohnson2004}) 
The \textit{differential entropy} (or simply, the \textit{entropy}) of a continuous random variable $F$ with density $f$ is defined by:
\begin{equation}\label{deq:1.15}
H(F) = H(f) := - \int f(x) \log{ f(x)}dx = - E[\log{ f(F)}].
\end{equation}
We use the convention that $0\log{0}\equiv 0$. For two continuous random variables $F$ and $Z$ with densities $f$ and $\phi$ respectively, the measure of the discrepancy between the distributions of $F$ and $Z$ is the \textit{relative entropy} (or the so-called \textit{Kullback-Leibler distance}) defined by
\begin{equation}\label{deq:1.27}
D(F\|Z) = D(f \| \phi):= \int f(x) \log{\bigg(\frac{f(x)}{\phi(x)}\bigg)}dx.
\end{equation}
Note that if $\text{supp}(f) \nsubseteq \text{supp}(\phi)$, then $D(f \| \phi) = \infty$.
\end{Definition}

\begin{remark} 

The relative entropy is non-negative: $D(F \| Z) \geq 0$ for any random variables $F, Z$ with densities $f$ and $\phi$ respectively. Indeed, using Jensen's inequality for the convex function $-\log$, we have 
\begin{align*}
D(F \| Z)& =  - \int f(x) \log{\bigg(\frac{\phi(x)}{f(x)}\bigg)}dx = E\bigg[- \log \bigg( \frac{\phi(F)}{f(F)}\bigg)\bigg] \\ 
& \geq -\log \bigg(E\bigg[ \frac{\phi(F)}{f(F)}\bigg] \bigg) = -\log \bigg( \int \frac{\phi(x)}{f(x)}f(x)dx \bigg)\\
&=  -\log \bigg( \int \phi(x)dx \bigg) = 0.
\end{align*}
\end{remark}

\begin{Definition}(\cite[Def. 1.12]{dJohnson2004})
For a random variable $F$ with continuously differentiable density $f$, we define the \textit{score function} $\rho_F$ of $F$ as the $\mathbb{R}$-valued function given by
$$\rho_F(x) = \frac{f'(x)}{f(x)} = \frac{d}{dx} (\log{f(x)}).$$
Additionally, if we assume that $F$ has variance $\sigma^2$, we define the \textit{Fisher information} $J(F)$ and the standardised Fisher information $J_{st}(F)$ as follows:
\begin{align}
J(F) &= E[\rho_F(F)^2]  \label{deq:1.90}\\
J_{st}(F)&=\sigma^2 E[(\rho_F(F))^2] - 1 = \sigma^2 J(F) -1. \label{deq:1.109}
\end{align}
\end{Definition}
It is easily seen that the score function $\rho_F$ is uniquely determined by the so-called Stein identity (see e.g. \cite[C1]{dJohnson2004}). That is, $\rho_F$ is the only function satisfying
\begin{equation}\label{deq:1.94}
E[\rho_F(F)g(F)] = -E[g'(F)]
\end{equation}
for any test function $g: \mathbb{R} \to \mathbb{R}$. Moreover, if $Z$ is $\mathcal{N}(0, \sigma^2)$-distributed then $J(Z)= \sigma^{-2}$. Hence, $ \sigma^{-2}J_{st}(F) = J(F) -  \sigma^{-2} = J(F) - J(Z)$ is the difference between the Fisher information of $F$ and $Z$.
 
\vspace{0.5cm}
We now turn to the study of relationships between convergences in the sense of entropy, Fisher information, total variation and $L^p$-distances. 

Throughout the sequel, we denote by $F$ a centered real-valued random variable with unit variance and smooth density $f$, and we let $Z \sim \mathcal{N}(0, 1)$ be a standard Gaussian with density $\phi(x) = \frac{1}{\sqrt{2\pi}}e^{-x^2/2}, x \in \mathbb{R}$.

The $L^p$-distance (resp. supremum norm) between densities of $F$ and $Z$ is given by
\begin{align*}
&\|f - \phi\|_{L^p} = \bigg(\int_{\mathbb{R}}|f(x) - \phi(x)|^p dx  \bigg)^{1/p},\\ 
& \Big(\text{resp. } \|f - \phi\|_\infty = \sup_{x \in \mathbb{R}}\|f(x) - \phi(x) \|\Big).
\end{align*}

The \emph{total variation distance} $d_{TV}(F, Z)$ between $F$ and $Z$ is defined as 
$$d_{TV}(F, Z) = \sup_{A \in \mathcal{B}(\mathbb{R})}|\mathbb{P}(F \in A) - \mathbb{P}(Z \in A)|.$$
It is known that the convergence in total variation distance is stronger than the convergence in distribution, see e.g., \cite[Proposition C.3.1]{dNourdinPeccatibook}. Moreover, for any $p \in [1, \infty)$, 
$$\|f - \phi\|_{L^p} \leq \|f - \phi\|_{L^1}^{1/p}\|f - \phi\|_{\infty}^{1-1/p}.$$
Thus, a bound for $L^p$-distance may be always deduced from a bound for $L^1$ and supremum distances. Furthermore, we have the following useful identity for total variation distance, see e.g., \cite{dNourdinPeccatibook},
\begin{align*}
d_{TV}(F, Z)&= \frac{1}{2} \sup_{\|h\|_\infty \leq 1}|E[h(F)] - E[h(Z)]| \\ 
& = \frac{1}{2} \int_{\mathbb{R}}|f(x) - \phi(x)|dx = \frac{1}{2}\|f -\phi\|_{L^1}.
\end{align*}
As a result, controlling both the total variation distance and the supremum distance implies the $L^p$-convergence for any $p \in (1, \infty)$.

It is worth pointing out that a link between relative entropy and total variation distance is provided by the celebrated \textit{Csiszár-Kullback-Pinsker inequality}, implying that for any probability densities, the convergence in the sense of relative entropy is stronger than convergence in total variation distance. More precisely: 

\begin{Proposition}( \cite[Lemma 1.8]{dJohnson2004}) For any random variables $F$ and $ Z$, we have
$$2 (d_{TV}(F, Z))^2  \leq D(F\| Z).$$
\end{Proposition} 

%

In 1975, Shimizu \cite{dShimizu} (see also \cite[Lemma E.1]{dJohnson2004}) proved that the convergence in the sense of Fisher information distance to a standard Gaussian random variable is stronger than convergence in total variation distance and supremum norm. The constants obtained by Shimizu in his original paper \cite{dShimizu} have been then improved by Johnson and Barron \cite{dBarronJohnson} and Ley and Swan \cite{dLeySwan}.

\begin{Proposition}(Shimizu's inequality) Let $F$ be a centered real-valued random variable with unit variance and continuously differentiable density $f$. Let $Z$ be a standard Gaussian random variable. Then the following two inequalities hold:
\begin{align}
&\sup_{x}|f(x) - \phi(x)| \leq \sqrt{J(F)} \label{dE.24}\\ 
&d_{TV}(F, Z) \leq \frac{1}{\sqrt{2}} \sqrt{J(F)}. \label{dE.26} 
\end{align}
\end{Proposition}
%

The relative entropy and the Fisher information are also strongly related to each other via the so-called \textit{de Bruijn's identity}, see e.g., \cite[Lemma 1]{dBarron} or \cite[C1]{dJohnson2004}.
\begin{Lemma}(de Bruijn's identity) \label{dLemma2.3}
Let $F$ be a centered real-valued random variable with unit variance and let $Z$ be a standard Gaussian. Assume, without loss of generality, that $F$ and $Z$ are independent. Then,
\begin{equation}\label{deq:2.38}
D(F\|Z) = \int_0^1 \frac{J (\sqrt{t}F + \sqrt{1-t}Z) - 1}{2t}dt.
\end{equation}
\end{Lemma}
Furthermore, taking into account (see e.g., \cite[Lemma 1.21]{dJohnson2004}) that 
 $$J(\sqrt{t}F + \sqrt{1-t}Z) \leq tJ(F) + (1-t)J(Z) = 1 + t(J(F) -1),$$
we obtain
 \begin{equation}\label{deq:1.7}
 D(F\| Z) \leq \frac{1}{2} (J(F) - 1).
 \end{equation}
As a consequence, convergence in the sense of Fisher information distance to a standard Gaussian random variable is stronger than convergence in the sense of relative entropy distance.

\subsection{Entropy and Fisher information for random vectors}

\begin{Definition}(\cite[Def 3.1]{dJohnson2004}) 
The \textit{differential entropy} (or simply, the \textit{entropy}) of a continuous random vector $F$ with density $f$ is defined by:
\begin{equation}\label{deq3.42}
H(F) = H(f) := - \int f(\textbf{x}) \log{ f(\textbf{x})}d\textbf{x} = - E[\log{ f(F)}].
\end{equation}
We use the convention $0\log{0}\equiv 0$. The measure of the discrepancy between the distributions of $F$ and $Z$ is the \textit{relative entropy} (aka the \textit{Kullback-Leibler distance})
\begin{equation}\label{deq:3.43}
D(F\|Z) := \int f(\textbf{x}) \log{\bigg(\frac{f(\textbf{x})}{\phi(\textbf{x})}\bigg)}d\textbf{x}.
\end{equation}
\end{Definition}

Now we can define the Fisher information matrix as follows. Given a function $p$, write $\nabla p$ for the gradient vector $(\partial p / \partial x_1, \ldots, \partial p / \partial x_n)^T$ and $\text{Hess } p$ for the Hessian matrix $(\text{Hess }p)_{ij} = \partial^2p/\partial x_i \partial x_j$.

\begin{Definition}(\cite[Def 3.2]{dJohnson2004})\label{Def3.2}
 For a random vector $F$ with differentiable density $f$ and covariance matrix $C > 0$, we define the \textit{score} $\rho_F$ of $F$ as the $\mathbb{R}^d$-valued function given by
\begin{equation}\label{deq:2.27}
\rho_F: \mathbb{R}^d \to \mathbb{R}^d: \textbf{x} \mapsto \rho_F(\textbf{x}) = (\rho_{F,1}(\textbf{x}), \ldots, \rho_{F,d}(\textbf{x}))^T: = \nabla \log{f(\textbf{x})}.
\end{equation}
We also define the \textit{Fisher information matrix} $J(F)$ and its \textit{standardised version} $J_{st}(F)$ of $F$ by
\begin{align}
J(F) &:= E[\rho_F(F) \rho_F(F)^T] \label{deq:3.44}\\
J_{st}(F)&:=C E[(\rho_F(F) + C^{-1}F)(\rho_F(F) + C^{-1}F)^T] = C(J(F) - C^{-1}) \label{deq:3.46}
\end{align}
(with components of $J(F)$ are $J(F)_{ij} = E[\rho_{F, i}(F)\rho_{F, j}(F)]$ for $1 \leq i, j \leq d$).
\end{Definition}

It is known that the score vector-function $\rho_F(F)$ is uniquely determined by the following integration by parts (see e.g., \cite[Lemma 3.3]{dJohnson2004}). That is, $\rho_F$ is the only function satisfying:
\begin{equation}\label{deq:3.50}
E[\rho_F(F)g(F)] = -E[\nabla g(F)] \qquad \text{ for all test function } g: \mathbb{R}^d \to \mathbb{R}.
\end{equation}
In particular, 
\begin{equation}\label{deq:3.49}
E[\rho_{F, i}(F)g(F)] = -E\Big[\frac{\partial g}{\partial x_i}(F)\Big] \qquad \forall 1\leq i \leq d.  
\end{equation}
Note that if $Z$ is a centered Gaussian vector with covariance $C$, then $J(Z)= C^{-1}$ and the positive semidefinite matrix $C^{-1}J_{st}(F) = J(F) - C^{-1}$ is the difference between the Fisher information matrices of $F$ and $Z$.

\vspace{0.5cm}
As in dimension one, let us now review the relationships between convergence to Gaussian vectors in the sense of entropy, Fisher information, total variation and $L^p$-distances. 
 
Similarly as in dimension one, the \emph{total variation distance} $d_{TV}(F, Z)$ between $d$-dimensional random vectors $F$ and $Z$ is defined as 
$$d_{TV}(F, Z) = \sup_{A \in \mathcal{B}(\mathbb{R}^d)}|\mathbb{P}(F \in A) - \mathbb{P}(Z \in A)|.$$
We also have a strong connection between total variation distance and $L^1$-norm of densities as follows:
\begin{align*}
d_{TV}(F, Z)&= \frac{1}{2} \sup_{\|h\|_\infty \leq 1}|E[h(F)] - E[h(Z)]| \\ 
& = \frac{1}{2} \int_{\mathbb{R}^d}|f(\xx) - \phi(\xx)|d\xx= \frac{1}{2}\|f -\phi\|_{L^1}
\end{align*}

In the multi-dimensional case, the relative entropy and the total variation distance are also linked together. Precisely: 

\begin{Proposition}(Csisz\'{a}r-Kullback-Pinsker inequality) For any random vectors $F$ and $ Z$, we have
\begin{equation}\label{deq:114}
2 (d_{TV}(F, Z))^2  \leq D(F \| Z).
\end{equation}
\end{Proposition} 

Therefore, the convergence in the sense of relative entropy is stronger than convergence in total variation distance. In particular, note that $D(F \| Z) \geq 0$. See, e.g., \cite{dBolleyVillani} for a proof of (\ref{deq:114}) and original references or see \cite{dHuLuNualart2014}.

The analogue in dimension one of the relationship between relative entropy and Fisher information is provided by the multidimensional counterpart of \textit{de Bruijn's identity}, see e.g., \cite[Lemma 2.2]{dJohnsonSuhov}.

 \begin{Lemma}(Multivariate de Bruijn's identity) \label{dLemma2.31}
Let $F$ be a $d$-dimensional random vector with invertible covariance matrix $C$ and let $Z$ be Gaussian with covariance $C$ as well. Then,
\begin{equation}\label{deq:2.381}
D(F\|Z) = \int_0^1 \frac{\text{tr} (J_{st}(F_t))}{2t}dt,
\end{equation}
where $F_t: = \sqrt{t}F + \sqrt{1-t}Z$ is the centered random vector with covariance matrix $C$ and 'tr' is the usual trace operator.
\end{Lemma}

If $F$ and $Z$ are random vectors that have both \textit{identity covariance matrix}, a straightforward extension of \cite[Lemma 1.21]{dJohnson2004} to the multivariate setting yields that the standardised Fisher information decreases along convolutions. Precisely, for all $ 0 \leq t \leq 1$,
$$\text{tr}(J_{st}(F_t)) \leq t \text{ tr}(J_{st}(F)) + (1-t)\text{tr}(J_{st}(Z)) = t \text{ tr}(J_{st}(F)).$$
It follows that
\begin{equation}\label{deq:8}
D(F\| Z) \leq \frac{1}{2} \text{tr}(J_{st}(F)).
\end{equation}
As a consequence, convergence in the sense of Fisher information to a standard Gaussian random vector is stronger than convergence in the sense of relative entropy.

\section{Summary of the four articles that constitute this thesis}\label{Section1.4}

\subsection{[a] Non-central limit theorems for quadratic functionals of Hermite-driven long memory moving average processes}

Fractional Ornstein-Uhlenbeck process, fOU in short, is the unique strong solution of the Langevin equation, driven by the fractional Brownian motion $B^H$ as a noise. Namely,
\begin{equation}\label{deqfOU}
dX_t =  - \alpha X_tdt + \sigma dB^H_t, \qquad X_0=0.
\end{equation}
Here $\sigma > 0$ is a constant, and $\alpha > 0$ is the drift of the model. 

The fOU process has received a lot of attention recently, especially because one can use the powerful toolbox of Gaussian analysis to deal with it, see e.g., \cite{dHuNualart, dHuNualartZhou, dKleptsynaLeBreton, dTudorViens2007}. But in some practical models, the Gaussian assumption may be implausible (cf. Taqqu \cite{dTaqqu1978}). This is why we propose to add a new parameter, namely $q \geq 1$, in (\ref{deqfOU}):
\begin{equation}\label{deqHOU}
dX_t = - \alpha X_tdt + \sigma dZ^{q, H}_t, \qquad x_0= 0.
\end{equation}
In (\ref{deqHOU}), $Z^{q, H}$ is a Hermite process of order $q \geq 1$ and Hurst parameter $H \in (1/2, 1)$. Note that $q=1$ in (\ref{deqHOU}) corresponds to (\ref{deqfOU}). The stochastic differential equation (\ref{deqHOU}) has a unique strong solution that is the almost surely continuous process given by
\begin{equation}\label{dHOUprocess}
X_t =\sigma \int_0^t e^{-\alpha (t-u)}dZ^{q, H}_u, \qquad t \geq 0.
\end{equation} 
Here, the integral $\int_0^t e^{\alpha u}dZ^{q, H}_u$ must be understood in the Riemann-Stieljes sense (see \cite[Prop. 1]{dMaejimaTudor}). Following \cite{dMaejimaTudor}, the stochastic process (\ref{dHOUprocess}) is called (non-stationary) \textit{Hermite Ornstein-Uhlenbeck process} of order $q$.
 
More generally, we can consider a class of long memory moving average processes driven by Hermite process of the form
\begin{equation}\label{deq:Xt}
X_t^{(q, H)} := \int_0^t x(t-u) dZ^{q, H}(u), \qquad t \geq 0,
\end{equation}
where $x$ is a regular deterministic function. For $x(u) = \sigma e^{-\alpha u}$ with $\sigma$ and $ \alpha > 0$, one recovers the Hermite Ornstein-Uhlenbeck process (\ref{dHOUprocess}). The purpose of the article \cite{dTran} is to study the asymptotic behavior, as $T \to \infty$, of the normalized quadratic functional
\begin{equation}\label{deqGt}
G_T^{(q, H)}(t):=\frac{1}{T^{2H_0 - 1}}\int_0^{Tt}\Big(\big(X_s^{(q, H)}\big)^2 - E\Big[\big(X_s^{(q, H)}\big)^2\Big]\Big) ds,
\end{equation}
where $H_0$ given by (\ref{dH2}), because of its potential to be then used for dealing with statistical inference related to (\ref{deqHOU}).

Theorem 1.1 in \cite{dTran} proves a \textit{non}-central limit theorem for $G_T^{(q, H)}$ as $T \to \infty$. Roughly speaking, it shows the convergence, in the sense of finite-dimensional distributions, to the Rosenblatt process (up to a multiplicative constant), irrespective of the value of $q \geq 2$ and $H \in (\frac{1}{2}, 1)$.

\begin{Theorem}(\cite[Theorem 1.1]{dTran}) \label{dthm4.1}
Let $H \in (\frac{1}{2}, 1)$ and let $Z^{(q, H)}$ be a Hermite process of order $q \geq 2$ and self-similarity parameter $H$. Consider the Hermite-driven moving average process $X^{(q, H)}$ defined by (\ref{deq:Xt}), and assume that the kernel $x$ is a real-valued integrable function on $[0, \infty)$ satisfying, in addition,
\begin{equation}\label{deq:1}
\int_{\mathbb{R}_+^2} |x(u)||x(v)||u-v|^{2H-2}dudv < \infty.
\end{equation}
Then, as $T \to \infty$, the family of stochastic processes $G_T^{(q, H)}$ converges in the sense of finite-dimensional distributions to $b(H, q)R^{H'}$, where $R^{H'}$ is the Rosenblatt process of parameter $H' =  1 + (2H-2)/q$, and $b(H, q)$ is an explicit positive constant.
\end{Theorem}

For $q=1$ (Gaussian case), $X^{(1, H)}$ is nothing but the fractional Volterra process, which includes fOU process as a particular case. Theorem 1.2 in \cite{dTran} shows that, for all $H \in (\frac{3}{4}, 1)$, the family of stochastic processes $G_T^{(1, H)}$ converges to the Rosenblatt process in the sense of finite-dimensional distributions, up to a multiplicative constant. The result complements a study initiated by Nourdin \textit{et al} in \cite{dNourdinNualartZintout} where a central limit theorem was established for $H \in (\frac{1}{2}, \frac{3}{4})$. 

\begin{Theorem} (\cite[Theorem 1.2]{dTran}) \label{dthm4.2}
Let $H \in (\frac{3}{4}, 1)$. Consider the fractional Volterra process $X^{(1, H)}$ given by (\ref{deq:Xt}) with $q=1$. If the function $x$ defining $X^{(1, H)}$ is an integrable function on $[0, \infty)$ and satisfies (\ref{deq:1}), then the family of stochastic processes $G_T^{(1, H)}$ converges in the sense of finite-dimensional distributions, as $T \to \infty$, to the Rosenblatt process $R^{H''}$ of parameter $H'' = 2H-1$ multiplied by an explicit positive constant $b(1, H)$.
\end{Theorem}

As a consequence of Theorem \ref{dthm4.1}, it is worth pointing out that, irrespective of the value of the self-similarity parameter $H \in (\frac{1}{2}, 1)$, the normalized quadratic functionals of any non-Gaussian Hermite-driven long memory moving average processes $(q \geq 2)$ always exhibits a convergence to a random variable belonging to the \textit{second} Wiener chaos. It is in contrast with what happens in the Gaussian case $(q=1)$, where either central or non-central limit theorems may arise depending on the value of the self-similarity parameter. This phenomenon is analogous to the one studied in the works \cite{dChronopoulouTudorViens2011, dClauselRoueffTaqquTudor2013, dClauselRoueffTaqquTudor2014, dTudorViens}.

Proof of Theorem \ref{dthm4.1} and \ref{dthm4.2} are done via the expansion of $G^{(q,H)}_T$ into a sum of components belonging to different Wiener chaoses. The asymptotic behavior of each chaos component is then analyzed and it follows that the dominant term is the term in the second Wiener chaos (i.e. other terms are negligible). The convergence of the second chaos term is studied by means of the isometry property of multiple integrals, and eventually leads to the convergence of $G^{(q,H)}_T$ to the Rosenblatt process. 
 
\subsection{[b] Non-central limit theorem for quadratic variations of non-Gaussian multiparameter Hermite random fields}

Let $Z^{q, \HH} = (Z^{q, \HH}_\t)_{\t \in [0, 1]^d}$ be a $d$-parameter Hermite random field of order $q \geq 1$ and self-similarity parameter $\HH = (H_1, \ldots, H_d) \in (\frac{1}{2}, 1)^d$. The quadratic variation of $Z^{q, \HH}$ is defined as
\begin{equation}\label{deq:quadraticHsheet}
V_{\NN}: = \frac{1}{\NN}\sum_{\ii=0}^{\NN-1}\bigg[\NN^{2\HH}\Big(\Delta Z^{q, \HH}_{[\frac{\ii}{\NN}, \frac{\ii +1}{\NN}]}\Big)^2 - 1 \bigg],
\end{equation}
where $\Delta Z^{q, \HH}_{[\s, \t]}$ is the increments of $Z^{q, \HH}$ given by (\ref{deq:generalincrements}). The bold notation is here systematically used in presence of multi-indices (we refer to \cite[Section 3.2]{dTran1} for precise definitions). Quadratic variation is often the quantity of interest when we deal with the estimation problem for the self-similarity parameter, see \cite{dChronopoulouTudorViens2011, dTudorViens}.

When $q=1$, $Z^{1, \HH}$ is either a fractional Brownian motion if $d=1$ or a fractional Brownian sheet if $d \geq 2$. The behavior of the quadratic variation of fBm is well-known since the eighties, and was analyzed in a series of seminal works by Breuer and Major \cite{dBreuerMajor}, Dobrushin and Major \cite{dDobrushinMajor}, Giraitis and Surgailis \cite{dGiraitisSurgailis} or Taqqu \cite{dTaqqu1979}. In the case $d\geq 2$, the asymptotic behavior for the quadratic variation of fBs has been actually known recently and we refer the readers to \cite{dPakkanenReveillac1, dPakkanenReveillac2} (see also \cite{dReveillac1}). In all these references, central and non-central limit theorems may arise, depending on the value of the Hurst parameter. 

Note that in the case $q \geq 2$ and $d=1$, we deal with the quadratic variation of a \textit{non-Gaussian} Hermite process. Chronopoulou, Tudor and Viens \cite{dChronopoulouTudorViens2011} (see also \cite{dTudorViens}) showed the following behavior for the sequence $V_N$:
$$ N^{(2-2H)/q}V_N \xrightarrow{L^2(\Omega)} c_{q, H}R^{H'}_1.$$
Here, $R^{H'}_1$ is a Rosenblatt random variable with Hurst parameter $H' = 1+ (2H-2)/q$ and $c_{q, H}$ is an explicit constant.

When $q \geq 2$ and $d \geq 2$, we have extended the result of \cite{dChronopoulouTudorViens2011} by studying quadratic variations for the class of non-Gaussian Hermite \textit{sheets}. Precisely, Theorem 1.1 in \cite{dTran1} proves the following non-central limit theorem.

\eject
\begin{Theorem} (\cite[Theorem 1.1]{dTran1}) \label{dthm4.3}
Fix $q \geq 2$, $d\geq 1$ and $\HH \in (\frac{1}{2}, 1)^d$. Let $Z^{q, \HH}$ be a $d$-parameter Hermite random field of order $q$ with self-similarity parameter $\HH$. Then
$$  c_{q, \HH}^{-1}\NN^{(2-2\HH)/q} V_{\NN} \xrightarrow{L^2(\Omega)}R^{\HH'}_{\textbf{1}},$$
where $R^{\HH'}_{\textbf{1}}$ is a $d$-parameter Rosenblatt sheet with Hurst parameter $\HH' = 1 + (2\HH -2)/q$ evaluated at time $\bf{1}$, and $c_{q, \HH}$ is an explicit constant.
\end{Theorem}

In this multiparameter setting, we observe the same phenomenon than in \cite{dTran}. Whatever the value of the self-similarity parameter, the normalized quadratic variation of a \textit{non-Gaussian} multiparameter Hermite random fields always converges to a random variable belonging to the second Wiener chaos.

Our proof of Theorem \ref{dthm4.3} is based on the use of chaotic expansion of the quadratic variation $V_{\NN}$ into multiple Wiener-It\^{o} integrals, that is, we use a similar strategy than in \cite{dTran}. Among all these chaos terms, the dominant one is the term in the second Wiener chaos. The convergence to Rosenblatt sheet evaluated at time $\bf{1}$ is then shown by applying the isometry property of multiple Wiener-It\^{o} integrals.

\subsection{[c] Statistical inference for Vasicek-type model driven by Hermite processes}

Let us now review the recent contribution \cite{dNourdinTran} about parameter estimation for Vasicek-type model driven by Hermite processes. Fractional Vasicek process is the unique almost surely continuous solution to the following SDE:
\begin{equation}\label{dfractionlVasicek}
dX_t = a(b - X_t)dt +dB_t^H,
\end{equation}
where $B^H$ is a fractional Brownian motion of index $H \in (\frac12, 1)$, and $a >0, b \in \mathbb{R}$ are real parameters.
The statistical inference for fractional Vasicek model has been analyzed recently in \cite{dXiaoYu}. This stochastic model, displaying self-similarity and long-range dependence, has been used to describe phenomenons appearing in hydrology, geophysics, telecommunication, economics or finance. 

In \cite{dNourdinTran}, we propose a new extended model of (\ref{dfractionlVasicek}), where fractional Brownian motion is replaced by a Hermite process:
\begin{equation}\label{dfracV}
dX_t = a(b - X_t)dt +dZ_t^{q, H}, \qquad t \geq 0,
\end{equation}
with initial condition $X_0=0$. Here $a > 0$ and $b \in \mathbb{R}$ are unknown drift parameters, and $Z^{q, H}$ is a Hermite process of order $q \geq 1$ with known Hurst parameter $H \in (\frac{1}{2}, 1)$. 

When $q=1$ in (\ref{dfracV}), one recovers the fractional Vasicek model. When $b=0$, the solution to (\ref{dfracV}) is nothing but a Hermite Ornstein-Uhlenbeck process. These various models have the potential to successfully model non-Gaussian data with long range dependence and self-similarity.

Our main purpose in \cite{dNourdinTran} is to construct an estimator for $(a, b)$ in (\ref{dfracV}) based on continuous-time observations of the sample paths of $X$. We prove the strong consistency and we derive rates of convergence. 
\vspace{0.1cm}

Our estimators for the drift parameters $a$ and $b$ in (\ref{dfracV}) are defined as follows:
\begin{eqnarray}\label{deq:X3estimator}
\widehat{a}_T&=&\left(\frac{\alpha_T}{H\Gamma(2H)}\right)^{-\frac1{2H}}, \quad\mbox{where $\alpha_T=\frac{1}{T}\int_0^T X_t^2dt - \left(\frac{1}{T}\int_0^T X_tdt\right)^2$},\label{dalpha_T}\\
\widehat{b}_{T}&=&\frac{1}{T}\int_0^T X_tdt.\notag
\end{eqnarray}
Before describing our result, we state the following proposition which will be needed to study the joint convergence of the estimators.  

\begin{Proposition}(\cite[Proposition 1.2]{dNourdinTran})\label{dfinfini}
Assume either ($q=1$ and $H>\frac34$) or $q\geq 2$. Fix $T>0$, and let
$U_T=(U_T(t))_{t\geq 0}$ be the process defined as $U_T(t)=\int_0^t e^{-T(t-u)}dZ^{q, H}_u$.
Finally, let $G_T$ be the random variable defined as
\[
G_T = T^{\frac{2}{q}(1-H)+2H}\int_0^1 \big(U_T(t)^2 - \E[U_T(t)^2])dt.
\]
Then $G_T$ converges in $L^2(\Omega)$ to a limit written $G_\infty$.
Moreover, $G_\infty/B_{H,q}$ is distributed according to the Rosenblatt distribution of parameter $1-\frac{2}{q}(1-H)$, where $B_{H, q}$ is an explicit constant depending only on $H$ and $q$.
\end{Proposition}

We can now describe the asymptotic behavior of $(\widehat{a}_T,\widehat{b}_T)$ as
$T\to\infty$.

\begin{Theorem}(\cite[Theorem 1.3]{dNourdinTran})\label{dmain3}
Let $X=(X_t)_{t\geq 0}$ be given by (\ref{dfracV}), where $Z^{q, H}=(Z^{q, H}_t)_{t\geq 0}$ is a Hermite process of order $q\geq 1$ and parameter $H\in(\frac12,1)$, and where $a>0$ and $b\in\R$
are (unknown) real parameters.
The following convergences take place as $T\to\infty$.
\begin{enumerate}
\item{} {\rm [Consistency]} 
$
(\widehat{a}_T,\widehat{b}_{T})\overset{\rm a.s.}{\to} (a,b).
$
\item{} {\rm [Fluctuations]}
They depend on the values of $q$ and $H$.
\begin{itemize}
\item {\rm (Case $q=1$ and $H<\frac34$)} 
\begin{eqnarray}
\left(\sqrt{T}\{\widehat{a}_T-a\},
T^{1-H}\{\widehat{b}_{T}-b\}
\right)&\overset{\rm law}{\to}&\left(-\frac{a^{1+4H}\sigma_H}{2H^2\Gamma(2H)}\,N,
\frac{1}{a}N'\right), \nonumber\\ \label{dG1}
\end{eqnarray}
where $N,N'\sim \mathcal{N}(0,1)$ are independent and $\sigma_H$ is given by 
\begin{equation}\label{dsigma}
\sigma_H=\frac{2H-1}{H\Gamma(2H)^2}\,\sqrt{\int_\R\left(
\int_{\R_+^2}e^{-(u+v)}|u-v-x|^{2H-2}dudv
\right)^2dx}.
\end{equation}
\item {\rm (Case $q=1$ and $H=\frac34$)} 
\begin{equation}\label{d34}
\left(\sqrt{\frac{T}{\log T}}\{\widehat{a}_T-a\},T^{\frac14}\big\{\widehat{b}_T - b\} \right)
\to
\left(\frac{3}4\sqrt{\frac{a}\pi}\, N,\frac{N'}{a}\right),
\end{equation}
where $N,N'\sim \mathcal{N}(0,1)$ are independent.
\item {\rm (Case $q=1$ and $H>\frac34$)} 
\begin{equation}
\left(T^{2(1-H)}\{\widehat{a}_T-a\},T^{1-H}\big\{\widehat{b}_T - b\} \right)\overset{\rm law}{\to}
\left(-\frac{a^{2H-1}}{2H^2\Gamma(2H)}\Big(G_\infty-(B^H_1)^2\Big),\frac{B^H_1}{a}\right),
\label{dR4}
\end{equation}
where $B^H=Z^{1,H}$ is the fractional Brownian motion and $G_\infty$ is defined in Proposition \ref{dfinfini}.

\item {\rm (Case $q\geq 2$ and any $H$)} 
\begin{equation}
\left(T^{\frac2q(1-H)}\{\widehat{a}_T-a\},T^{1-H}\big\{\widehat{b}_T - b\} \right)
\overset{\rm law}{\to}
\left(-\frac{a^{1-\frac2q(1-H)}}{2H^2\Gamma(2H)}\,G_\infty,\frac{Z^{q, H}_1}{a}\right),
\label{dR}
\end{equation}
where $G_\infty$ is defined in Proposition \ref{dfinfini}.
\end{itemize}
\end{enumerate}
\end{Theorem}

We see from Theorem \ref{dmain3} that the strong consistency of $(\widehat{a}_T, \widehat{b}_T)$ is universal for any Vasicek type model driven by Hermite process as a noise, no matter that it is Gaussian or not. Very differently, the fluctuations of our estimators around the true value of the drift parameters depend heavily on the order $q$ and Hurst parameter $H$ of the underlying Hermite process. This gives us some hints to understand how much the fractional model (\ref{dfractionlVasicek}) relies on the Gaussian feature.

\subsection{[d] Fisher information and multivariate Fourth Moment Theorem}

Fix an integer $d \geq 1$. Let $F = (F_1, \ldots, F_d )$ be a $d$-dimensional centered random vector with invertible covariance matrix $C$. We assume that the law of $F$ admits a density $f = f_F$ with respect to the Lebesgue measure. Let $Z = (Z_1, \ldots, Z_d)$ be a $d$-dimensional centered Gaussian vector which has the same covariance matrix $C$ as $F$ and admits the density $\phi = \phi_d( .; C)$. Without loss of generality, we may and will assume that the vectors $F$ and $Z$ are stochastically independent. 

In the first part of \cite{dTran3}, we extend the relationship (\ref{deq:8}) for any covariance matrix $C$ of $F$. Precisely, the convergence in the sense of standardised Fisher information is always stronger than the convergence in the sense of relative entropy. 

\begin{Proposition}(\cite[Ch. 5, Prop. 5.2.1]{dTran3}) Let the above notation. Then,
\begin{equation}\label{deq:new1}
D(F\| Z) \leq \|C\|_{op} \times \frac{1}{2} \text{tr}(C^{-1}J_{st}(F)) = \|C\|_{op} \times \frac{1}{2}\Big(\text{tr} (J(F)) - \text{tr}(J(Z))\Big)
\end{equation}
As a consequence,
\begin{equation}\label{deq:new2}
\|f - \phi \|^{2}_{L^1(\mathbb{R}^d)} = 4(d_{TV}(F, Z))^2 \leq 2D(F\|Z) \leq \| C\|_{op}\text{tr}(C^{-1}J_{st}(F)).
\end{equation}
\end{Proposition}

The study of normal approximations for sequences of multiple stochastic integrals has received a lot of attention recently. In the main part of \cite{dTran3}, we are interested in estimating the discrepancy between the distributions of $F$ and the Gaussian vector $Z$ by working with $L^p$- norms, total variation distance, relative entropy or Fisher information, when $F$ is a $d$-dimensional centered random vector whose components are multiple stochastic integrals. 

In \cite{dNourdinPeccatiSwan}, Nourdin, Peccati and Swan have obtained an upper bound for the total variation distance between the distributions of the sequences of $d$-dimensional random vector $F_n$ and the standard Gaussian vector $Z_n$, via an evaluation of the relative entropy involving Malliavin calculus. Precisely, suppose that $F_n =(I_{q_1}(f_{1, n}), \ldots, I_{q_d}(f_{d, n}))$ is a random vector with unit covariance matrix and $ q_i \geq 1, \forall i = 1, \ldots, d$. Then, 
$$  2(d_{TV}(F_n, Z_n))^2 \leq D(F_n \| Z_n) \leq O(1) \Delta_n|\log{\Delta_n}|, $$
where $\Delta = E[\|F_n\|^4] - E[\|Z_n\|^4] $. Here the constant depends on $d, q_1, \ldots, q_d$ and on the sequence $(F_n)$, but not $n$. The notation $\| . \|$ denotes the Euclidian norm on $\mathbb{R}^d$.

Furthermore, in the one-dimensional case, Nourdin and Nualart \cite{dNourdinNualart2013} exhibited a sufficient condition, in terms of the negative moments of the norm of the Malliavin derivative, under which convergence in Fisher information to the standard Gaussian of sequences belonging to a given Wiener chaos is actually equivalent to convergence of only the fourth moment. That is, if $F = I_q(f), q \geq 2$ has unit variance, then under assumption $E[\| DF\|^{-4-\epsilon}] \leq \eta$ for some $\epsilon > 0$ and $\eta \geq 1$ we have
$$J(F) - 1 \leq \text{cst} (E[F^4] - 3).$$
Here the constant depends on $q, \epsilon$ and $\eta$ but not on $F$. As a direct consequence of this upper bound, together with the Fourth Moment Theorem (see \cite[Theorem 5.2.7]{dNourdinPeccatibook}), we obtain the equivalence of various forms of convergence. More precisely, given a sequence of random variables $(F_n)$ of multiple stochastic integrals with unit variance, one has, under the assumption that $\limsup_{n \to \infty}E[\| DF_n\|^{-4-\epsilon}] < \infty$ and with $N \sim \mathcal{N}(0, 1)$: 
\begin{description}
\item[\hspace{0.8cm}]Convergence of the fourth moments: $E[F_n^4] \to 3$;
\item[$\Longleftrightarrow $]  Convergence in distribution: $F_n \xrightarrow{(d)} N$;
\item[$\Longleftrightarrow $]  Convergence in total variation distance: $d_{TV}(F_n, N) \to 0$;
\item[$\Longleftrightarrow $] Convergence in the sense of relative entropy: $D(F_n \| N) \to 0$;
\item[$\Longleftrightarrow $] Convergence in the sense of Fisher information: $J(F_n) \to 1$;
\item[$\Longleftrightarrow $]  Uniform convergence of densities: $\|f_{F_n} - \phi\|_\infty \to 0$, where $f_{F_n}$ and $\phi$ are densities of $F_n$ and $N$ respectively.
\end{description}

In the multi-dimensional case, that is when $F = (I_{q_1}(f_1), \ldots, I_{q_d}(f_d))$, one can naturally wonder whether under suitable sufficient conditions we could obtain an upper bound on Fisher information and deduce from them a list of equivalences between different forms of convergence.
 
Before stating our results, we recall that a random vector $F = (F_1, \ldots, F_d)$ in $\mathbb{D}^\infty$ is called \textit{non-degenerate} if its \textit{Malliavin matrix} $\gamma_F = (\left\langle DF_i, DF_j \right\rangle_{\mathfrak{H}})_{1 \leq i,j \leq d}$ is invertible a.s. and $(\text{det}\gamma_F)^{-1} \in \cap_{p \geq 1}L^p(\Omega)$.

\begin{Theorem}(\cite[Ch. 5, Thm 5.2.3]{dTran3})\label{dThm222} Let $F = (F_1, \ldots, F_d) = (I_{q_1}(f_1), \ldots, I_{q_d}(f_d))$ be a non-degenerate random vector with $1 \leq q_1 \leq \ldots \leq q_d$ and $f_i \in \mathfrak{H}^{\odot q_i}$. Let $\gamma_F$ be the Malliavin matrix of $F$. Denote by $C: = (E[F_iF_j])_{1 \leq i, j \leq d}$ the covariance matrix of $F$ and set $Q: = \text{diag}(q_1, \ldots, q_d)$. Then, for any real number $p > 12$,
\begin{equation}\label{deq:result1}
\text{tr}(C^{-1}J_{st}(F)) \leq \text{cst}(C, Q, d) \|(\text{det}\gamma_F)^{-1}\|_{p}^4 \sum_{j=1}^d \Big\| \|DF_j\|_{\mathfrak{H}}^2 - q_jc_{jj}\Big\|_{L^2(\Omega)},
\end{equation}
where $\text{cst}(C, Q, d)$ means a positive constant depending only on $d, C$ and $Q$.
\end{Theorem}

As a direct result of Theorem \ref{dThm222} and the Fourth Moment Theorem, we obtain the following equivalence between different ways of converging to the normal distribution for random vectors whose components are multiple stochastic integrals.

\begin{corollary} (\cite[Ch.5, Corrollary 5.2.4]{dTran3})
Let $d \geq 2$ and let $q_1, \ldots, q_d \geq 1$ be some fixed integers. Consider vectors
$$F_n = (F_{1, n}, \ldots, F_{d, n}) = (I_{q_1}(f_{1, n}), \ldots, I_{q_d}(f_{d, n}) ), \quad n \geq 1,$$
with $f_{i, n} \in \mathfrak{H}^{\odot q_i}$. Let $C = (c_{ij})_{1 \leq i, j \leq d}$ be a symmetric non-negative definite matrix, and let $Z \sim \mathcal{N}_d(0, C)$. Assume that $F_n$ is \emph{uniformly non-degenerate} (in the sense that $\gamma_{F_n}$ is invertible a.s. for all $n$ and $\limsup_{n \to \infty} \|(\text{det }\gamma_{F_n})^{-1}\|_{L^p} < \infty$ for all $p > 12$) and that
$$\lim_{n\to \infty} E[F_{i,n}F_{j,n}] =c_{ij}, \qquad 1 \leq i, j \leq d.$$
Then, as $n \to \infty$, the following assertions are equivalent:
\begin{enumerate}[(a)]
\item $F_n$ converges in law to $Z$;
\item For every $1 \leq i \leq d$, $F_{i,n}$ converges in law to $\mathcal{N}(0, c_{ii})$;
\item $tr (J(F_n)) \to tr (J(Z))$, that is $F_n$ converges to $Z$ in the sense of Fisher information distance;
\item $D (F_n \| Z) \to 0$;
\item $d_{TV}(F_n, Z) \to 0$;
\item  $\|f_{F_n} - \phi\|_\infty \to 0$, where $f_{F_n}$ and $\phi$ are densities of $F_n$ and $Z$ respectively, that is the uniform convergence of densities.
\end{enumerate}
\end{corollary}

\chapter[Non-central limit theorems for quadratic functionals of Hermite-driven long memory moving average processes]{Non-central limit theorems for quadratic functionals of Hermite-driven long memory moving average processes}

\begin{center}
T. T. Diu Tran \\
Universit\'{e} du Luxembourg
\end{center}

\begin{tcolorbox}
This article is published on \textit{Stochastics and Dynamics}, \textbf{18}, no. 4, 2017. arXiv:1607.08278.
\end{tcolorbox}

\begin{center}
\textbf{Abstract}
\end{center}

Let $(Z_t^{q, H})_{t \geq 0}$ denote a Hermite process of order $q \geq 1$ and self-similarity parameter $H  \in (\frac{1}{2}, 1)$. Consider the Hermite-driven moving average process
$$X_t^{q, H} = \int_0^t x(t-u) dZ^{q, H}_u, \qquad t \geq 0.$$ 
In the special case of $x(u) = e^{-\theta u}, \theta > 0$, $X$ is the non-stationary Hermite Ornstein-Uhlenbeck process of order $q$. Under suitable integrability conditions on the kernel $x$, we prove that as $T \to \infty$, the normalized quadratic functional
$$G_T^{q, H}(t)=\frac{1}{T^{2H_0 - 1}}\int_0^{Tt}\Big(\big(X_s^{q, H}\big)^2 - E\Big[\big(X_s^{q, H}\big)^2\Big]\Big) ds , \qquad t \geq 0,$$
where $H_0 = 1 + (H-1)/q$, converges in the sense of finite-dimensional distribution to the Rosenblatt process of parameter $H' =  1 + (2H-2)/q$, up to a multiplicative constant, irrespective of self-similarity parameter whenever $q \geq 2$. In the Gaussian case $(q=1)$, our result complements the study started by Nourdin \textit{et al} in \cite{mIvan2}, where either central or non-central limit theorems may arise depending on the value of self-similarity parameter. A crucial key in our analysis is an extension of the connection between the classical multiple Wiener-It\^{o} integral and the one with respect to a random spectral measure (initiated by Taqqu (1979)), which may be independent of interest.

%

\section{Motivation and main results}

Let $(Z_t^{q, H})_{t \geq 0}$ be a Hermite process of order $q \geq 1$ and self-similarity parameter $H  \in (\frac{1}{2}, 1)$. It is a $H$-self-similar process with stationary increments, exhibits long-range dependence and can be expressed as a multiple Wiener-It\^{o} integral of order $q$ with respect to a two-sided standard Brownian motion $(B(t))_{t \in \mathbb{R}}$ as follows:
\begin{equation}\label{meq:H1}
Z^{q, H}_t= c(H, q) \int_{\mathbb{R}^q} \bigg( \int_0^t \prod_{j=1}^q(s- \xi_j)_+^{H_0 - \frac{3}{2}}ds\bigg) dB(\xi_1)\ldots dB(\xi_q),
\end{equation}
where
\begin{equation}\label{meq:H2}
c(H, q) = \sqrt{\frac{H(2H - 1)}{q! \beta^q(H_0 - \frac{1}{2}, 2-2H_0)}} \qquad \text{and} \qquad H_0 = 1+\frac{H-1}{q}.
\end{equation}
Particular examples include the fractional Brownian motion $(q=1)$ and the Rosenblatt process $(q=2)$. For $q \geq 2$, it is no longer Gaussian. All Hermite processes share the same basic properties with fractional Brownian motion such as self-similarity, stationary increments, long-range dependence and even covariance structure. The Hermite process has been pretty much studied in the last decade, due to its potential to be good model for various phenomena.

A theory of stochastic integration with respect to $Z^{q, H}$, as well as stochastic differential equation driven by this process, have been considered recently. We refer to \cite{mIvan, mNualart} for a recent account of the fractional Brownian motion and its large amount of applications. We refer to \cite{mTaqqu2, mTudor2,mTudor} for different aspects of the Rosenblatt process. Furthermore, in the direction of stochastic calculus, the construction of Wiener integrals with respect to $Z^{q, H}$ is studied in \cite{mTudor1}. According to this latter reference, stochastic integrals of the form
\begin{equation}\label{meq:12}
\int_{\mathbb{R}}f(u)dZ^{q, H}_u
\end{equation}
are well-defined for elements of $\mathcal{H} = \{f: \mathbb{R} \to \mathbb{R}: \int_{\mathbb{R}}\int_{\mathbb{R}} f(u)f(v)|u-v|^{2H-2}dudv < \infty\}$, endowed with the norm
\begin{equation}\label{meq:14}
||f||_{\mathcal{H}}^2 = H(2H-1)\int_{\mathbb{R}}\int_{\mathbb{R}} f(u)f(v)|u-v|^{2H-2}dudv.
\end{equation}
Moreover, when $f \in \mathcal{H}$, the stochastic integral (\ref{meq:12}) can be written as
\begin{equation}\label{meq:13}
\int_{\mathbb{R}}f(u)dZ^{q, H}_u = c(H, q) \int_{\mathbb{R}^q}\bigg(\int_{\mathbb{R}}f(u) \prod_{j=1}^q(u- \xi_j)_+^{H_0 - \frac{3}{2}}du\bigg)dB(\xi_1)\ldots dB(\xi_q)
\end{equation}
where $c(H, q)$ and $H_0$ are as in (\ref{meq:H2}). Since the elements of $\mathcal{H}$ may be not functions but distributions (see \cite{mNualart}), it is more practical to work with the following subspace of $\mathcal{H}$, 
which is a set of functions:
$$ |\mathcal{H}| = \bigg\{f: \mathbb{R} \to \mathbb{R}: \int_{\mathbb{R}}\int_{\mathbb{R}} |f(u)||f(v)||u-v|^{2H-2}dudv < \infty \bigg\}.$$

Consider the stochastic integral equation
\begin{equation}\label{meq:SDE}
X(t) = \xi  - \lambda \int_0^tX(s)ds + \sigma Z^{q, H}_t, \qquad t \geq 0,
\end{equation}
where $ \lambda, \sigma > 0$ and where the initial condition $\xi$ can be any random variable. By \cite[Prop. 1]{mTudor1}, the unique continuous solution of (\ref{meq:SDE}) is given by
$$X(t) = e^{-\lambda t} \bigg( \xi + \sigma \int_0^t e^{\lambda u} dZ^{q, H}_u \bigg), \qquad t \geq 0.$$
In particular, if  for $\xi$ we choose $\xi = \sigma \int_{-\infty}^0 e^{\lambda u} dZ^{q, H}(u)$, then
\begin{equation}\label{meq:17}
X(t) = \sigma \int_{-\infty}^t e^{-\lambda(t-u)}dZ^{q, H}_u, \qquad t\geq 0.
\end{equation}
According to \cite{mTudor1}, the process $X$ defined by (\ref{meq:17}) is referred to as the Hermite Ornstein-Uhlenbeck process of order $q$. On the other hand, if the initial condition $\xi$ is set to be zero, then the unique continuous solution of (\ref{meq:SDE}) is this time given by
\begin{equation}\label{meq:18}
X(t) = \sigma \int_0^t e^{-\lambda(t-u)}dZ^{q, H}_u, \qquad t\geq 0.
\end{equation}
In this paper, we call the stochastic process (\ref{meq:18}) the \textit{non-stationary} Hermite Ornstein-Uhlenbeck process of order $q$. It is a particular example of a wider class of moving average processes driven by Hermite process, of the form
\begin{equation}\label{meq:Xt}
X_t^{q, H} := \int_0^t x(t-u) dZ^{q, H}_u, \qquad t \geq 0.
\end{equation}

In many situations of interests (see, e.g., \cite{mViens1, mViens2}), we may have to analyze the asymptotic behavior of the \textit{quadratic functionals of $X_t^{(q, H)}$} for statistical purposes. More precisely, let us consider
\begin{equation}\label{meq:Gt}
G_T^{q, H}(t):=\frac{1}{T^{2H_0 - 1}}\int_0^{Tt}\Big(\big(X_s^{q, H}\big)^2 - E\Big[\big(X_s^{q, H}\big)^2\Big]\Big) ds.
\end{equation}
In this paper, we will show that $G_T^{q, H}$  converges in the sense of finite-dimensional distribution to the Rosenblatt process (up to a multiplicative constant), irrespective of the value of $q \geq 2$ and $H \in (\frac{1}{2}, 1)$. The case $q=1$ is apart, see Theorem \ref{mTheorem2} below.

\begin{Theorem}\label{mTheorem1}
Let $H \in (\frac{1}{2}, 1)$ and let $Z^{q, H}$ be a Hermite process of order $q \geq 2$ and self-similarity parameter $H$. Consider the Hermite-driven moving average process $X^{q, H}$ defined by (\ref{meq:Xt}), and assume that the kernel $x$ is a real-valued integrable function on $[0, \infty)$ satisfying, in addition,
\begin{equation}\label{meq:1}
\int_{\mathbb{R}_+^2} |x(u)||x(v)||u-v|^{2H-2}dudv < \infty.
\end{equation}
Then, as $T \to \infty$, the family of stochastic processes $G_T^{q, H}$ converges in the sense of finite-dimensional distribution to $b(H, q)R^{H'}$, where $R^{H'}$ is the Rosenblatt process of parameter $H' =  1 + (2H-2)/q$ (which is the \emph{second-order} Hermite process of parameter $H'$), and the multiplicative constant $b(H, q)$ is given by
\begin{equation}\label{meq:19}
b(H, q) = \frac{H(2H-1)}{\sqrt{(H_0 -\frac{1}{2})(4H_0 - 3)}}\int_{\mathbb{R}_+^2}x(u)x(v)|u-v|^{(q-1)(2H_0 -2)}dudv.
\end{equation}
(The fact that (\ref{meq:19}) is well-defined is part of the conclusion of the theorem.)
\end{Theorem}
Theorem \ref{mTheorem1} only deals with $q \geq 2$, because $q=1$ is different. In this case, $Z^{1, H}$ is nothing but the fractional Brownian motion of index $H$ and $X^{1, H}$ is the fractional Volterra process, as considered by Nourdin, Nualart and Zintout in \cite{mIvan2}. In this latter reference, a Central Limit Theorem for $G_T^{1, H}$ has been established for $H \in (\frac{1}{2}, \frac{3}{4})$. Here, we rather study the situation where $H \in (\frac{3}{4}, 1)$ and, in contrast to \cite{mIvan2}, we show a Non-Central Limit Theorem. More precisely, we have the following theorem. 

\begin{Theorem}\label{mTheorem2}
Let $H \in (\frac{3}{4}, 1)$. Consider the fractional Volterra process $X^{1, H}$ given by (\ref{meq:Xt}) with $q=1$. If the function $x$ defining $X^{1, H}$ is an integrable function on $[0, \infty)$ and satisfies (\ref{meq:1}), then the family of stochastic processes $G_T^{1, H}$ converges in the sense of finite-dimensional distribution, as $T \to \infty$, to the Rosenblatt process $R^{H''}$ of parameter $H'' = 2H-1$ multiplied by $b(1, H)$ as above.
\end{Theorem}

It is worth pointing out that, irrespective of the self-similarity parameter $H \in (\frac{1}{2}, 1)$, the normalized quadratic functionals of any non-Gaussian Hermite-driven long memory moving average processes $(q \geq 2)$ exhibits a convergence to a random variable belonging to the second Wiener chaos. It is in strong contrast with what happens in the Gaussian case $(q=1)$, where either central or non-central limit theorems may arise depending on the value of the self-similarity parameter.

We note that our Theorem \ref{mTheorem2} is pretty close to Taqqu's seminal result \cite{mTaqqu1975}, but cannot be obtained as a consequence of it. In contrast, the statement of Theorem \ref{mTheorem1} is completely new, and provides new hints on the importance and relevance of the Rosenblatt process in statistics.

Our proofs of Theorems \ref{mTheorem1} and \ref{mTheorem2} are based on the use of chaotic expansions into multiple Wiener-It\^{o} integrals and the key transformation lemma from the classical multiple Wiener-It\^{o} integrals into the one with respect to a random spectral measure (following a strategy initiated by Taqqu in \cite{mTaqqu}). Let us sketch them. Since the random variable $X_t^{q, H}$ is an element of the $q$-th Wiener chaos, we can firstly rely on the product formula for multiple integrals to obtain that the quadratic functional $G_T^{q, H}(t)$ can be decomposed into a sum of multiple integrals of even orders from $2$ to $2q$. Secondly, we prove that the projection onto the second Wiener chaos converges in $L^2(\Omega)$ to the Rosenblatt process: we do this by using its spectral representation of multiple Wiener-It\^{o} integrals and by checking the $L^2(\mathbb{R}^2)$ convergence of its kernel. Finally, we prove that all the remaining terms in the chaos expansion are asymptotically negligible. 

Our findings and the strategy we have followed to obtain them owe a lot and were influenced by several seminal papers on Non-Central Limit Theorems for functionals of Gaussian (or related) processes, including Dobrushin and Major \cite{mDobrushin Major}, Taqqu \cite{mTaqqu} and most recently, Clausel \textit{et al} \cite{mCATudor1, mCATudor2} and Neufcourt and Viens \cite{mViens}.
 
Our paper is organised as follows. Section $2$ contains preliminary key lemmas. The proofs of our two main results, namely Theorems \ref{mTheorem1} and \ref{mTheorem2}, are then provided in Section $3$ and Section $4$. 

\section{Preliminaries}
Here, we mainly follow Taqqu \cite{mTaqqu}. We describe a useful connection between multiple Wiener-It\^{o} integrals with respect to random spectral measure and the classical stochastic It\^{o} integrals. Stochastic representations of the Rosenblatt process are then provided at the end of the section.

\subsection{Multiple Wiener-It\^{o} integrals with respect to Brownian motion}

Let $f \in L^2(\mathbb{R}^q)$ and let us denote by $I_q^B(f)$ the $q$th multiple Wiener-It\^{o} integral of $f$ with respect to the standard two-sided Brownian motion $(B_t)_{t \in \mathbb{R}}$, in symbols
$$I_q^B(f) = \int_{\mathbb{R}^q} f(\xi_1, \ldots, \xi_q) dB(\xi_1)\ldots dB(\xi_q).$$
When $f$ is symmetric, we can see $I_q^B(f)$ as the following iterated adapted It\^{o} stochastic integral:
$$I_q^B(f) = q! \int_{-\infty}^{\infty}dB(\xi_1) \int_{-\infty}^{\xi_1}dB(\xi_2) \ldots \int_{-\infty}^{\xi_{q-1}}dB(\xi_q) f(\xi_1, \ldots, \xi_q).$$
Moreover, when $f$ is not necessarily symmetric one has $I_q^B(f) = I_q^B(\widetilde{f})$, where $\widetilde{f}$ is the symmetrization of $f$ defined by
\begin{equation}\label{meq:dola}
\widetilde{f}(\xi_1, \ldots, \xi_q) = \frac{1}{q!}\sum_{\sigma \in \mathfrak{S}_q} f(\xi_{\sigma(1)}, \ldots, \xi_{\sigma(q)}).
\end{equation}
The set of random variables of the form $I_q^B(f), f \in L^2(\mathbb{R}^q)$, is called the $q$th Wiener chaos of $B$. We refer to Nualart's book \cite{mNualart} (chapter 1 therein) or Nourdin and Peccati's books \cite{mIvan, mPeccatiIvan} for a detailed exposition of the construction and properties of multiple Wiener-It\^{o} integrals. Here, let us only recall the product formula between two multiple integrals: if $f \in L^2(\mathbb{R}^p)$ and $g \in L^2(\mathbb{R}^q)$ are two symmetric functions then
\begin{equation}\label{meq:P1}
I_p^B(f)I_q^B(g) = \sum_{r=0}^{p \wedge  q} r!\binom{p}{r}\binom{q}{r}I_{p+q-2r}^B(f \widetilde{\otimes}_r g),
\end{equation}
where the contraction $f \otimes_r g$, which belongs to $L^2(\mathbb{R}^{p+q-2r})$ for every $r = 0, 1, \ldots, p \wedge q$, is given by
\begin{align}\label{meq:P2}
f \otimes_r g& (y_1, \ldots, y_{p-r}, z_1, \ldots, z_{q-r}) \nonumber\\ 
& = \int_{\mathbb{R}^r} f(y_1, \ldots, y_{p-r}, \xi_1, \ldots, \xi_r) g(z_1, \ldots, z_{q-r}, \xi_1, \ldots, \xi_r) d\xi_1 \ldots d\xi_r
\end{align}
and where a tilde denotes the symmetrization, see (\ref{meq:dola}). Observe that
\begin{equation}\label{meq:3}
\| f \widetilde{\otimes}_r g\|_{L^2(\mathbb{R}^{p+q-2r})} \leq \| f \otimes_r g\|_{L^2(\mathbb{R}^{p+q-2r})} \leq \|f\|_{L^2(\mathbb{R}^p)}\|g\|_{L^2(\mathbb{R}^q)}, \quad r= 0, \ldots, p\wedge q
\end{equation}
by Cauchy-Schwarz inequality, and that $f \otimes_p  g = \left\langle f, g \right\rangle_{L^2(\mathbb{R}^p)}$ when $p=q$.
Furthermore, we have the orthogonality property
$$E[I_p^B(f)I_q^B(g)] =
\begin{cases}
 & p! \big\langle \widetilde{f}, \widetilde{g} \big\rangle_{L^2(\mathbb{R}^p)} \qquad\text{if } p=q\\
& 0 \qquad\qquad\qquad\quad \text{ if } p \ne q.
\end{cases}$$
\subsection{Multiple Wiener-It\^{o} integrals with respect to a random spectral measure}

Let $W$ be a Gaussian complex-valued random spectral measure that satisfies $E[W(A)] = 0, E[W(A)\overline{W(B)}] = \mu(A \cap B), W(A) = \overline{W(-A)} $ and $W(\bigcup_{j=1}^n A_j) = \sum_{j=1}^n W(A_j)$ for all disjoint Borel sets that have finite Lebesgue measure (denoted here by $\mu$). The Gaussian random variables $\text{Re}W(A)$ and $\text{Im}W(A)$ are then independent with expectation zero and variance $\mu(A)/2$. We now recall briefly the construction of multiple Wiener-It\^{o} integrals with respect to $W$, as defined in Major \cite{mMajor} or Section $4$ of Dobrushin \cite{mDobrushin}. To define such stochastic integrals let us introduce the real Hilbert space $\mathscr{H}_m$ of complex-valued symmetric functions $f(\lambda_1, \ldots, \lambda_m), \lambda_j \in \mathbb{R}, j=1, 2, \ldots, m$, which are even, i.e. $f(\lambda_1, \ldots, \lambda_m) = \overline{f(-\lambda_1, \ldots, -\lambda_m)}$, and square integrable, that is, 
$$\|f \|^2 = \int_{\mathbb{R}^m}|f(\lambda_1, \ldots, \lambda_m)|^2 d\lambda_1\ldots d\lambda_m < \infty.$$
The scalar product is similarly defined: namely, if $f, g \in \mathscr{H}_m$, then
$$\left\langle f, g \right\rangle_{\mathscr{H}_m} = \int f(\lambda_1, \ldots, \lambda_m)\overline{g(\lambda_1, \ldots, \lambda_m)}d\lambda_1 \ldots d\lambda_m.$$
The integrals $I_m^W$ are then defined through an isometric mapping from $\mathscr{H}_m$ to $L^2(\Omega)$:
$$ f \longmapsto I_m^W(f)  = \int_{\mathbb{R}}f(\lambda_1, \ldots, \lambda_m) W(d\lambda_1)\ldots W(d\lambda_m),$$
Following \emph{e.g.} the lecture notes of Major \cite{Major}, if $f \in \mathscr{H}_m$ and $g \in \mathscr{H}_n$, then $E[I_m^W(f)] = 0$ and
\begin{equation}\label{meq:21}
E[I_m^W(f)I_n^W(g)] =
\begin{cases}
 & m!\left\langle f, g \right\rangle_{\mathscr{H}_m} \text{ if } m=n\\
& 0 \qquad\qquad\quad\text{if } m \ne n.
\end{cases}
\end{equation}

\subsection{Preliminary lemmas}
We recall a connection between the classical Wiener-It\^{o} integral $I^B$ and the one with respect to a random spectral measure $I^W$ that will play an important role in our analysis.

\begin{lemma}\cite[Lemma 6.1]{mTaqqu}\label{mLemma6.1} Let $A(\xi_1, \ldots, \xi_m)$ be a real-valued symmetric function in $L^2(\mathbb{R}^m)$ and let
\begin{equation}\label{meq:Fourier}
\mathcal{F}A(\lambda_1, \ldots, \lambda_m) = \frac{1}{(2\pi)^{m/2}}\int_{\mathbb{R}^m}e^{i\sum_{j=1}^m \xi_j\lambda_j}A(\xi_1,\ldots, \xi_m)d\xi_1\ldots d\xi_m
\end{equation}
be its Fourier transform. Then
$$\int_{\mathbb{R}^m}A(\xi_1,\ldots, \xi_m)dB(\xi_1) \ldots dB(\xi_m) \overset{(d)}{=} \int_{\mathbb{R}^m}\mathcal{F}A(\lambda_1,\ldots, \lambda_m)W(d\lambda_1)\ldots W(d\lambda_m).$$
\end{lemma}

Applying Lemma \ref{mLemma6.1}, we deduce the following lemma which is an extended result of Lemma 6.2 in \cite{mTaqqu}.

\begin{lemma}\label{mLemma2}
Let 
$$ A(\xi_1, \ldots, \xi_{m+n}) = \int_{\mathbb{R}^2} \phi(z_1, z_2)\prod_{j=1}^{m}(z_1 - \xi_j)_+^{H_0 -\frac{3}{2}}\prod_{k=m+1}^{m+n}(z_2 - \xi_k)_+^{H_0 -\frac{3}{2}}dz_1dz_2$$
where $\frac{1}{2} < H_0 < 1$ and where $\phi$ is an integrable function on $\mathbb{R}^2$  whose Fourier transform is given by (\ref{meq:Fourier}).
Let 
$$\widetilde{A}(\xi_1, \ldots, \xi_{m+n}) = \frac{1}{(m+n)!}\sum_{\sigma \in \mathfrak{S}_{m+n}}A(\xi_{\sigma(1)}, \ldots, \xi_{\sigma(m+n)})$$
be the symmetrization of $A$. Assume that 
$$\int_{\mathbb{R}^{m+n}}|\widetilde{A}(\xi_1, \ldots, \xi_{m+n})|^2d\xi_1\ldots d\xi_{m+n} < \infty.$$
 Then,
\begin{align*}
&\int_{\mathbb{R}^{m+n}}\widetilde{A}(\xi_1,\ldots, \xi_{m+n})dB(\xi_1)\ldots dB(\xi_{m+n})\\ 
&\overset{(d)}{=} \bigg(\frac{\Gamma(H_0 - \frac{1}{2})}{\sqrt{2\pi}}\bigg)^{m+n} \int_{\mathbb{R}^{m+n}}W(d\lambda_1) \ldots W(d\lambda_{m+n}) \prod_{j=1}^{m+n} |\lambda_j|^{\frac{1}{2} - H_0}\\
&\qquad\qquad\quad\qquad\times\frac{1}{(m+n)!}\sum_{\sigma \in \mathfrak{S}_{m+n}} 2\pi\mathcal{F}\phi(\lambda_{\sigma(1)}+\ldots+ \lambda_{\sigma(m)}, \lambda_{\sigma(m+1)} + \ldots + \lambda_{\sigma(m+n)}).
\end{align*}
\end{lemma}

\begin{proof}
Thanks to Lemma \ref{mLemma6.1}, we first estimate the Fourier transform of $A(\xi_1, \ldots, \xi_{m+n})$. Because the function $u_+^{H_0 - \frac{3}{2}}$ belongs neither to $L^1(\mathbb{R})$ nor to $L^2(\mathbb{R})$, by similar arguments as in the proof of \cite[Lemma 6.2]{mTaqqu} let us introduce
$$A_T(\xi_1, \ldots, \xi_{m+n})= \begin{cases}
&A(\xi_1, \ldots \xi_{m+n}) \text{ if } |\xi_j| < T \text{ }\forall  j =1, \ldots, m+n.\\
&0 \qquad\qquad\qquad\text{ otherwise.}
\end{cases}
$$
Set
$$B_\lambda(a, b) = \frac{1}{\sqrt{2\pi}}\int_a^b e^{-iu\lambda}u^{H_0 - \frac{3}{2}}du$$
for $0 \leq a \leq b < \infty$, and $B_\lambda(a, \infty)  = \lim_{b \to \infty}B_\lambda(a, b)$. By \cite[page 80]{mTaqqu}, we get 
$$\sup_{0 \leq a \leq b}|B_\lambda(a, b)| \leq \frac{1}{\sqrt{2\pi}}\bigg(\frac{1}{H_0 - \frac{1}{2}} + \frac{2}{|\lambda|}\bigg).$$
Now, 
\begin{align*}
&\mathcal{F}A_T(\lambda_1,\ldots, \lambda_{m+n}) = \frac{1}{(\sqrt{2\pi})^{m+n}}\int_{\mathbb{R}^{m+n}}d\xi_1\ldots d\xi_{m+n} e^{i \sum_{j=1}^{m+n}\lambda_j\xi_j} \int_{\mathbb{R}^2}dz_1dz_2 \phi(z_1, z_2)\\
&\qquad\quad\qquad\qquad\qquad\qquad\times\prod_{j=1}^m(z_1 - \xi_j)_+^{H_0 - \frac{3}{2}}\prod_{j=m+1}^{m+n}(z_2 - \xi_j)_+^{H_0 - \frac{3}{2}}\mathbf{1}_{\{|\xi_j|<T, \forall j =1,\ldots, m+n\}}.
\end{align*}
The change of variables $\xi_j = z_1 - u_j$ for $j=1, \ldots, m$ and $\xi_j = z_2 - u_j$ for $j=m+1, \ldots, m+n$ yields
\begin{align*}
&\mathcal{F}A_T(\lambda_1,\ldots, \lambda_{m+n})\\
&= \frac{1}{(\sqrt{2\pi})^{m+n}}\int_{\mathbb{R}^{m+n}}du_1 \ldots du_{m+n} e^{-i \sum_{j=1}^{m+n}\lambda_j u_j} \int_{\mathbb{R}^2}dz_1dz_2 \phi(z_1, z_2)e^{i \sum_{j=1}^m\lambda_jz_1 }e^{i \sum_{j=m+1}^{m+n}\lambda_jz_2}\\
&\qquad\qquad\qquad\times \prod_{j=1}^{m} u_j^{H_0 - \frac{3}{2}}\mathbf{1}_{\{u_j > 0\}}\mathbf{1}_{\{ z_1 - T < u_j < z_1 + T\}}\prod_{j=m+1}^{m+n} u_j^{H_0 - \frac{3}{2}}\mathbf{1}_{\{u_j > 0\}}\mathbf{1}_{\{ z_2 - T < u_j < z_2 + T\}}.
\end{align*}
Suppose that $\lambda_1,\ldots, \lambda_{m+n}$ are different from zero. Since $\phi$ is integrable on $\mathbb{R}^2$ then
\begin{align*}
|\mathcal{F}A_T&(\lambda_1,\ldots, \lambda_{m+n})|\\
&\leq \int_{\mathbb{R}^2}dz_1dz_2 |\phi(z_1, z_2)| \prod_{j=1}^m B_{\lambda_j}(\max(0, z_1-T), \max(0, z_1+T))\\ 
&\qquad\qquad\qquad\quad\quad\times \prod_{j=m+1}^{m+n} B_{\lambda_j}(\max(0, z_2-T), \max(0, z_2+T))\\
& \leq \int_{\mathbb{R}^2}dz_1dz_2 |\phi(z_1, z_2)| \prod_{j=1}^{m+n}\frac{1}{\sqrt{2\pi}}\bigg(\frac{1}{H_0 -\frac{1}{2}} + \frac{2}{|\lambda_j|}\bigg),
\end{align*}
which is finite and uniformly bounded with respect to $T$. Thus, 
\begin{align*}
&\mathcal{F}A(\lambda_1,\ldots, \lambda_{m+n})= \lim_{T\to\infty}\mathcal{F}A_T(\lambda_1,\ldots, \lambda_{m+n})\\ 
& = 2\pi \mathcal{F}\phi(\lambda_1+ \ldots+ \lambda_m, \lambda_{m+1} + \ldots + \lambda_{m+n}) \prod_{j=1}^{m+n} \bigg(\frac{1}{\sqrt{2\pi}}\int_0^\infty e^{-iu\lambda_j}u^{H_0 -\frac{3}{2}}du \bigg).
\end{align*}
The integral inside the product is an improper Riemann integral. After the change of variables $v=u|\lambda_j|$, we get
\begin{align*}
\mathcal{F}A(&\lambda_1,\ldots, \lambda_{m+n})\\ 
& = 2\pi \mathcal{F}\phi(\lambda_1+ \ldots+ \lambda_m, \lambda_{m+1} + \ldots + \lambda_{m+n})\\
&\qquad\qquad\qquad\times \prod_{j=1}^{m+n} \bigg(|\lambda_j|^{\frac{1}{2}-H_0}\frac{1}{\sqrt{2\pi}}\int_0^\infty e^{-iu\text{sign} \lambda_j}u^{H_0 -\frac{3}{2}}du \bigg)\\
&= 2\pi \mathcal{F}\phi(\lambda_1+ \ldots+ \lambda_m, \lambda_{m+1} + \ldots + \lambda_{m+n}) \\
&\qquad\qquad\qquad\times\prod_{j=1}^{m+n} \bigg(|\lambda_j|^{\frac{1}{2}-H_0}\frac{1}{\sqrt{2\pi}}\Gamma(H_0 -\frac{1}{2})C(\lambda_j)\bigg),
\end{align*}
where $C(\lambda) = e^{-i\frac{\pi}{2}(H_0 - \frac{1}{2})}$ for $\lambda > 0, C(-\lambda) =\overline{C(\lambda)}$ and thus $|C(\lambda)|=1$ for all $\lambda \ne 0$, see appendix for the detailed computations. Applying Lemma \ref{mLemma6.1} by noticing that $C(\lambda_j)W(d\lambda_j) \overset{(d)}{=} W(d\lambda_j)$ (see \cite[Proposition 4.2]{mDobrushin}) and symmetrizing the Fourier transform of $A(\lambda_1, \ldots, \lambda_{m+n})$ lead to the desired conclusion.
\end{proof}

\subsection{Stochastic representations of the Rosenblatt process}
Let $(R^H(t))_{t \geq 0}$ be the Rosenblatt process of parameter $H \in (\frac{1}{2}, 1)$. The time representation of $R^H$ is 
\begin{align*}
R^H(t) &= a_1(D)\int_{\mathbb{R}^2}\bigg(\int_0^t (s-\xi_1)_+^{D-\frac{3}{2}}(s-\xi_2)_+^{D-\frac{3}{2}}ds\bigg)dB(\xi_1)dB(\xi_2)\\
 &=A_1(H)\int_{\mathbb{R}^2}\bigg(\int_0^t (s-\xi_1)_+^{\frac{H}{2}-1}(s-\xi_2)_+^{\frac{H}{2}-1}ds\bigg)dB(\xi_1)dB(\xi_2), 
\end{align*}
where $D = \frac{H+1}{2}$ and
$$a_1(D):= \frac{\sqrt{(D-1/2)(4D-3)}}{\beta(D-1/2, 2-2D)} = \frac{\sqrt{(H/2)(2H-1)}}{\beta(H/2, 1-H)}=:A_1(H).$$
Observe also that $1/2 < H<1 \Longleftrightarrow 3/4 < D < 1$. The corresponding spectral representation of this process, see for instance \cite{mTaqqu, mTaqqu2} or apply Lemma \ref{mLemma2}, is given by
\begin{align*}
R^H(t) &= a_2(D)\int_{\mathbb{R}^2}|\lambda_1|^{\frac{1}{2} - D}|\lambda_2|^{\frac{1}{2} - D}\frac{e^{i(\lambda_1+\lambda_2)t}-1}{i(\lambda_1+\lambda_2)}W(d\lambda_1)W(d\lambda_2)\\
 &=A_2(H)\int_{\mathbb{R}^2}|\lambda_1|^{-\frac{H}{2}}|\lambda_2|^{-\frac{H}{2}}\frac{e^{i(\lambda_1+\lambda_2)t}-1}{i(\lambda_1+\lambda_2)}W(d\lambda_1)W(d\lambda_2),
\end{align*}
where
$$a_2(D):= \sqrt{\frac{(2D-1)(4D-3)}{2[2\Gamma(2-2D)\sin (\pi (D-1/2))]^2}} = \sqrt{\frac{H(2H-1)}{2[2\Gamma(1-H)\sin (H\pi/2)]^2}}=:A_2(H).$$

\section{Proof of Theorem \ref{mTheorem1}}
We are now in a position to give the proof of our Theorem \ref{mTheorem1}. It is devided into four steps.

\subsection{Chaotic decomposition}

Using (\ref{meq:13}), we can write $X^{(q, H)}$ as a $q$-th Wiener-It\^{o} integral with respect to the standard two-sided Brownian motion $(B_t)_{t \in \mathbb{R}}$ as follows:
\begin{equation}\label{meq:5}
X^{(q, H)}_t = \int_{\mathbb{R}^q} L(x, t)(\xi_1,\ldots, \xi_q)dB(\xi_1)\ldots dB(\xi_q) = I_q^B(L(x, t)),
\end{equation}
where
\begin{equation}\label{meq:6}
L(x, t)(\xi_1,\ldots, \xi_q): = c(H, q) \int_{\mathbb{R}}\mathbf{1}_{[0, t]}(z) x(t -z) \prod_{j=1}^q (z - \xi_j)_+^{H_0 - \frac{3}{2}}dz,
\end{equation}
 with $c(H, q)$ and $H_0$ given by (\ref{meq:H2}). Applying the product formula (\ref{meq:P1}) for multiple Wiener-It\^{o} integrals, we easily obtain that
\begin{equation}\label{meq:dola2}
(X_t^{(q, H)})^2 - E[(X_t^{(q, H)})^2] = \sum_{r=0}^{q-1}r!\binom{q}{r}^2I_{2q-2r}^B(L(x, t) \widetilde{\otimes}_r L(x, t)).
\end{equation}
{\allowdisplaybreaks
Let us compute the contractions appearing in the right-hand side of (\ref{meq:dola2}). For every $ 0 \leq r \leq q-1$, by using Fubini's theorem we first have
\begin{align*}
L(x,& s) \otimes_r L(x, s) (\xi_1, \ldots, \xi_{2q-2r})\\ 
& = \int_{\mathbb{R}^r}dy_1 \ldots dy_r L(x, s)(\xi_1, \ldots, \xi_{q-r}, y_1, \ldots, y_r)L(x, s)(\xi_{q-r+1}, \ldots, \xi_{2q-2r}, y_1, \ldots, y_r)\\
&= c(H, q)^2 \int_{\mathbb{R}^r}dy_1 \ldots dy_r \int_0^s dz_1x(s-z_1) \prod_{j=1}^{q-r}(z_1 - \xi_j)_+^{H_0 - \frac{3}{2}}\prod_{i=1}^r(z_1 - y_i)_+^{H_0 - \frac{3}{2}}\\
& \qquad\qquad\quad\qquad\qquad\times  \int_0^s dz_2x(s-z_2) \prod_{j=q-r+1}^{2q-2r}(z_2 - \xi_j)_+^{H_0 - \frac{3}{2}}\prod_{i=1}^r(z_2 - y_i)_+^{H_0 - \frac{3}{2}}\\
& = c(H, q)^2 \int_{[0, s]^2}dz_1dz_2 x(s-z_1) x(s-z_2) \prod_{j=1}^{q-r}(z_1 - \xi_j)_+^{H_0 - \frac{3}{2}}\prod_{j=q-r+1}^{2q-2r}(z_2 - \xi_j)_+^{H_0 - \frac{3}{2}}\\
&\qquad\qquad\qquad\qquad\quad\times \bigg(\int_{\mathbb{R}}dy (z_1 - y)_+^{H_0 - \frac{3}{2}} (z_2 - y)_+^{H_0 - \frac{3}{2}}\bigg)^r,
\end{align*}
}and, since for any $z_1, z_2 \geq 0$
\begin{equation}\label{meq:2}
 \int_{\mathbb{R}}(z_1-y)_+^{H_0 - \frac{3}{2}}(z_2-y)_+^{H_0 -\frac{3}{2}}dy = \beta\Big(H_0 - \frac{1}{2}, 2-2H_0\Big)|z_1-z_2|^{2H_0-2},
\end{equation}
we end up with the following expression
 \begin{align}\label{m1}
&L(x, s) \otimes_r L(x, s) (\xi_1, \ldots, \xi_{2q-2r}) \nonumber \\ 
& = c(H, q)^2\beta\Big(H_0 - \frac{1}{2}, 2-2H_0\Big)^r\int_{[0, s]^2}dz_1dz_2 x(s-z_1) x(s-z_2)|z_1 -z_2|^{(2H_0-2)r} \nonumber \\
&\qquad\qquad\qquad\qquad\qquad\qquad\qquad\times \prod_{j=1}^{q-r}(z_1 - \xi_j)_+^{H_0 - \frac{3}{2}}\prod_{j=q-r+1}^{2q-2r}(z_2 - \xi_j)_+^{H_0 - \frac{3}{2}}.
 \end{align}
Recall $G_T^{(q, H)}$ from (\ref{meq:Gt}). As a consequence, we can write 
\begin{equation}\label{meq:8}
G_T^{(q, H)} (t)= F_{2q, T}(t) + c_{2q-2}F_{2q-2, T}(t) + \ldots + c_4F_{4, T}(t) + c_2F_{2, T}(t)
\end{equation}
where $c_{2q-2r}:= r!\binom{q}{r}^2$ and for $ 0 \leq r \leq q-1$,
\begin{equation}\label{meq:7}
F_{2q-2r, T}(t): = \frac{1}{T^{2H_0 -1}}\int_0^{Tt} I_{2q-2r}^B(L(x,s) \widetilde{\otimes}_r L(x, s))ds,
\end{equation}
where the kernels in each Wiener integral above are given explicitly in (\ref{m1}).

\subsection{Spectral representations}

Recall the expression of the contractions $L(x, s) \otimes_r L(x, s), 0\leq r \leq q-1$ given in (\ref{m1}). Set
\begin{align*}
\phi_r(s, z_1, z_2) :=&c(H, q)^2\beta\Big(H_0 - \frac{1}{2}, 2-2H_0\Big)^r\\
&\times  \mathbf{1}_{[0, s]}(z_1)\mathbf{1}_{[0, s]}(z_2)x(s-z_1)x(s-z_2) |z_1-z_2|^{(2H_0-2)r}.
\end{align*}
It is a symmetric function with respect to $z_1$ and $z_2$. Furthermore, by H\"{o}lder's inequality, we have
\begin{align*}
&\int_{\mathbb{R}^2}\Big|\mathbf{1}_{[0, s]}(z_1)\mathbf{1}_{[0, s]}(z_2)x(s-z_1)x(s-z_2) |z_1-z_2|^{(2H_0-2)r}\Big|dz_1dz_2\\
&\leq \int_{[0, s]^2} |x(s-z_1)| |x(s-z_2)| |z_1-z_2|^{(2H_0-2)r}dz_1dz_2 \\
&= \int_{[0, s]^2} |x(z_1)| |x(z_2)| |z_1-z_2|^{r\frac{(2H-2)}{q}}dz_1dz_2\\
&\leq \bigg(\int_{[0, \infty)^2}|x(z_1)||x(z_2)||z_1 -z_2|^{2H-2}dz_1dz_2 \bigg)^{\frac{r}{q}}\bigg(\int_0^\infty |x(z)|dz\bigg)^{2(1-\frac{r}{q})}.
\end{align*}
Using the integrability of $x$ together with the assumption (\ref{meq:1}), it turns out that $\phi_r(. , z_1, z_2) $ is integrable on $\mathbb{R}^2_+$. Applying Lemma \ref{mLemma2} with $m=n=q-r$, we get
\begin{align*}
F_{2q-2r, T}(t) &= \frac{1}{T^{2H_0 - 1}}\int_0^{Tt} I_{2q-2r}^B(L(x, s) \widetilde{\otimes}_r L(x, s))ds \nonumber\\
&  \overset{(d)}{=} A_r(H, q) \frac{1}{T^{2H_0 -1}} \int_{\mathbb{R}^{2q-2r}}W(d\lambda_1)\ldots W(d\lambda_{2q-2r}) \prod_{j=1}^{2q-2r}|\lambda_j|^{\frac{1}{2}-H_0} \nonumber\\
&\times \frac{1}{(2q-2r)!}\sum_{\sigma \in \mathfrak{S}_{2q-2r}}\int_0^{Tt}ds \int_{[0, s]^2}d\xi_1d\xi_2 x(s-\xi_1) x(s-\xi_2)|\xi_1 -\xi_2|^{(2H_0-2)r} \nonumber\\
&\qquad\qquad\qquad\qquad\qquad\qquad\times e^{i(\lambda_{\sigma(1)} + \ldots+ \lambda_{\sigma(q-r)})\xi_1}e^{i(\lambda_{\sigma(q-r+1)} + \ldots + \lambda_{\sigma(2q-2r)})\xi_2},
 \end{align*}
where
\begin{equation}\label{meq:Ar}
A_r(H,q) := c(H, q)^2 \beta(H_0 - \frac{1}{2}, 2- 2H_0)^r \bigg( \frac{\Gamma(H_0 - \frac{1}{2})}{\sqrt{2\pi}}\bigg)^{2q-2r}.
\end{equation}
The change of variable $s = Ts'$ yields
\begin{align*}
F_{2q-2r, T}(t) & \overset{(d)}{=} A_r(H, q)T^{2-2H_0} \int_{\mathbb{R}^{2q-2r}}W(d\lambda_1)\ldots W(d\lambda_{2q-2r}) \prod_{j=1}^{2q-2r}|\lambda_j|^{\frac{1}{2}-H_0}\nonumber\\
&\times \frac{1}{(2q-2r)!}\sum_{\sigma \in \mathfrak{S}_{2q-2r}}\int_0^{t}ds \int_{[0, Ts]^2}d\xi_1d\xi_2 x(Ts-\xi_1) x(Ts-\xi_2)|\xi_1 -\xi_2|^{(2H_0-2)r} \nonumber\\
&\qquad\qquad\qquad\qquad\qquad\qquad\times e^{i(\lambda_{\sigma(1)} + \ldots+ \lambda_{\sigma(q-r)})\xi_1}e^{i(\lambda_{\sigma(q-r+1)} + \ldots + \lambda_{\sigma(2q-2r)})\xi_2}.
 \end{align*}
Let us do a further change of variables: $\lambda'_{\sigma(j)} = T\lambda_{\sigma(j)}, j = 1, \ldots, 2q-2r $ and $\xi'_k = Ts -  \xi_k, k=1,2$. Thanks to the self-similarity of $W$ with index $1/2$ (that is, $W(T^{-1}d\lambda)$ has the same law as $T^{-1/2}W(d\lambda)$) we finally obtain that
\begin{align}\label{m4}
F_{2q-2r, T}(t) & \overset{(d)}{=} A_r(H, q)T^{-(2-2H_0)(q-1-r)} \nonumber \\
& \times\int_{\mathbb{R}^{2q-2r}}W(d\lambda_1)\ldots W(d\lambda_{2q-2r}) \prod_{j=1}^{2q-2r}|\lambda_j|^{\frac{1}{2}-H_0}\int_0^{t}ds e^{i(\lambda_1 + \ldots + \lambda_{2q-2r})s}\nonumber\\
&\times \frac{1}{(2q-2r)!}\sum_{\sigma \in \mathfrak{S}_{2q-2r}} \int_{[0, Ts]^2}d\xi_1d\xi_2 x(\xi_1) x(\xi_2)|\xi_1 -\xi_2|^{(2H_0-2)r}  \nonumber\\
&\qquad\qquad\qquad\qquad\quad\times e^{-i (\lambda_{\sigma(1)} + \ldots+ \lambda_{\sigma(q-r)})\frac{\xi_1}{T}}e^{-i (\lambda_{\sigma(q-r+1)} + \ldots + \lambda_{\sigma(2q-2r)})\frac{\xi_2}{T}}.
 \end{align}

\subsection{Reduction lemma}

\begin{lemma}\label{mReduction}
Fix $t$, fix $H \in (\frac{1}{2}, 1)$ and fix $q \geq 2$. Assume (\ref{meq:1}) and  the integrability of the kernel $x$. Then for any $r \in \{0, \ldots, q-2 \}$, one has
$$\lim_{T \to \infty} E[F_{2q-2r, T}(t)^2]= 0.$$ 
\end{lemma}

\begin{proof}
Without loss of generality, we may and will assume that $t=1$. From the spectral representation of multiple Wiener-It\^{o} integrals (\ref{m4}), one has 
\eject
\begin{align*}
&E[F_{2q-2r, T}(1)^2]   \\
&=  T^{-2(2-2H_0)(q-1-r)}  A_r^2(H, q)(2q-2r)!\int_{\mathbb{R}^{2q-2r}} d\lambda_1 \ldots d\lambda_{2q-2r} \prod_{j=1}^{2q-2r}|\lambda_j|^{1-2H_0}\\
&\times \bigg( \frac{1}{(2q-2r)!}\sum_{\sigma \in \mathfrak{S}_{2q-2r}}\int_0^{1}ds e^{i(\lambda_1 + \ldots + \lambda_{2q-2r})s} \int_{[0, Ts]^2}d\xi_1d\xi_2 x(\xi_1) x(\xi_2)|\xi_1 -\xi_2|^{(2H_0-2)r} \\
 &\hspace{6cm}\times e^{-i (\lambda_{\sigma(1)} + \ldots+ \lambda_{\sigma(q-r)})\frac{\xi_1}{T}}e^{-i (\lambda_{\sigma(q-r+1)} + \ldots + \lambda_{\sigma(2q-2r)})\frac{\xi_2}{T}} \bigg)^2.
\end{align*}
Since $x$ is a real-valued integrable function on $[0, \infty)$ satisfying assumption (\ref{meq:1}), we deduce from Lebesgue dominated convergence that, as $T \to \infty$,
\begin{align*}
& \frac{1}{(2q-2r)!}\sum_{\sigma \in \mathfrak{S}_{2q-2r}}\int_0^{1}ds e^{i(\lambda_1 + \ldots + \lambda_{2q-2r})s} \int_{[0, Ts]^2}d\xi_1d\xi_2 x(\xi_1) x(\xi_2)|\xi_1 -\xi_2|^{(2H_0-2)r} \\
 &\hspace{6cm} \times e^{-i (\lambda_{\sigma(1)} + \ldots+ \lambda_{\sigma(q-r)})\frac{\xi_1}{T}}e^{-i (\lambda_{\sigma(q-r+1)} + \ldots + \lambda_{\sigma(2q-2r)})\frac{\xi_2}{T}}\\
&\longrightarrow \int_{[0, \infty)^2}x(u)x(v)|u-v|^{(2H_0-2)r}dudv \int_0^1 e^{i(\lambda_1+\ldots + \lambda_{2q-2r})s}ds.
\end{align*}
Since $1- \frac{1}{2q} < H_0 < 1$ and $0 \leq r \leq q-2$, we have $ T^{-2(2-2H_0)(q-1-r)} \to 0 $ as $T \to \infty$. Moreover, since $\int_0^1 e^{i(\lambda_1+\ldots + \lambda_{2q-2r})\xi}d\xi = \frac{e^{i(\lambda_1+\ldots + \lambda_{2q-2r})} -1}{i(\lambda_1 + \ldots + \lambda_{2q-2r})}$,
\begin{align*}
&\int_{\mathbb{R}^{2q-2r}} d\lambda_1 \ldots d\lambda_{2q-2r} \prod_{j=1}^{2q-2r}|\lambda_j|^{1-2H_0} \bigg|\frac{e^{i(\lambda_1+\ldots + \lambda_{2q-2r})} -1}{i(\lambda_1 + \ldots + \lambda_{2q-2r})}\bigg|^2 \leq \bigg(\int_{\mathbb{R}} |\lambda|^{1-2H_0}d\lambda\bigg)^{2q-2r}
\end{align*}
which is integrable at zero, and
\begin{align*}
&\int_{\mathbb{R}^{2q-2r}} d\lambda_1 \ldots d\lambda_{2q-2r} \prod_{j=1}^{2q-2r}|\lambda_j|^{1-2H_0} \bigg|\frac{e^{i(\lambda_1+\ldots + \lambda_{2q-2r})} -1}{i(\lambda_1 + \ldots + \lambda_{2q-2r})}\bigg|^2\\ 
& \leq \int_{\mathbb{R}^{2q-2r}} d\lambda_1 \ldots d\lambda_{2q-2r} \prod_{j=1}^{2q-2r}|\lambda_j|^{1-2H_0} \frac{4}{(\lambda_1 + \ldots + \lambda_{2q-2r})^2}
\end{align*}
which is integrable at infinity, we have
$$\int_{\mathbb{R}^{2q-2r}} d\lambda_1 \ldots d\lambda_{2q-2r} \prod_{j=1}^{2q-2r}|\lambda_j|^{1-2H_0} \bigg|\frac{e^{i(\lambda_1+\ldots + \lambda_{2q-2r})} -1}{i(\lambda_1 + \ldots + \lambda_{2q-2r})}\bigg|^2 < \infty.$$
All these facts taken together imply 
\begin{equation}\label{meq:10}
E[F_{2q-2r, T}(1)^2] \longrightarrow 0, \text{ as } T \to \infty, \text{ for all } 0 \leq r \leq q-2,
\end{equation}
which proves the lemma.
\end{proof}

\subsection{Concluding the proof of Theorem \ref{mTheorem1}}
Thanks to Lemma \ref{mReduction}, we are left to concentrate on the convergence of the term $F_{2, T}$ (belonging to the second Wiener chaos) corresponding to $r=q-1$. Recall from (\ref{m4}) that $F_{2, T}(t)$ has the same law as the double Wiener integral with symmetric kernel given by
\begin{align}\label{meq:15}
f_T(t, \lambda_1, &\lambda_2):= A_{q-1}(H,q) |\lambda_1|^{\frac{1}{2} - H_0}|\lambda_2|^{\frac{1}{2} - H_0}\int_0^t ds e^{i(\lambda_1 +\lambda_2)s} \nonumber\\ 
&\qquad\quad  \times  \int_{[0, Ts]^2}d\xi_1d\xi_2 e^{-i(\lambda_1\frac{\xi_1}{T} + \lambda_2\frac{\xi_2}{T})}x(\xi_1)x(\xi_2)|\xi_1 - \xi_2|^{(q-1)(2H_0 - 2)}.
\end{align}
Observe that $f_T(t, .)$ is symmetric, so there is no need to care about symmetrization. By the isometry property of multiple Wiener-It\^{o} integrals with respect to the random spectral measure, in order to prove the $L^2(\Omega)$-convergence of $c_2F_{2, T}$ to $bR^{H'}$, we can equivalently prove that $c_2f_T(t, .)$ converges in $L^2(\mathbb{R}^2)$ to the kernel of $bR^{H'}(t)$ . First, by Lebesgue dominated convergence, as $T \to \infty$, we have
\begin{align*}
f_T(t, \lambda_1, \lambda_2)& \longrightarrow A_{q-1}(H, q) \int_{\mathbb{R}^2}x(u)x(v)|u-v|^{(q-1)(2H_0 - 2)}dudv\\ 
&\qquad\qquad\qquad\qquad\qquad\times  |\lambda_1|^{\frac{1}{2} - H_0}|\lambda_2|^{\frac{1}{2} - H_0} \frac{e^{i(\lambda_1 + \lambda_2)t} -1}{i(\lambda_1 + \lambda_2)}.
\end{align*}
This shows that $f_T(t, .)$ converges pointwise to the kernel of $R^{H'}(t)$, up to some constant. 
Moreover, for all $0 < S <T $,
\begin{align*}
&\|f_T(t, .) - f_S(t, .)\|_{L^2(\mathbb{R}^2)}^2\\
& = A_{q-1}^2(H, q) \int_{\mathbb{R}^2}d\lambda_1 d\lambda_2 |\lambda_1|^{1 - 2H_0}|\lambda_2|^{1 - 2H_0}\\ 
&\quad \times \bigg( \int_0^t ds e^{i(\lambda_1 +\lambda_2)s}\int_{[0, Ts]^2 \setminus  [0, Ss]^2}d\xi_1d\xi_2 e^{-i(\lambda_1\frac{\xi_1}{T} + \lambda_2\frac{\xi_2}{T})}x(\xi_1)x(\xi_2)|\xi_1 - \xi_2|^{(q-1)(2H_0 - 2)}\bigg)^2.
\end{align*}
By Lebesgue dominated convergence, it comes that $\|f_T(t, .) - f_S(t, .)\|_{L^2(\mathbb{R}^2)}^2 \longrightarrow 0$ as $T, S \to \infty$. It follows that $(f_T(t, .))_{T \geq 0}$ is a Cauchy sequence in $L^2(\mathbb{R}^2)$. Hence, the multiple Wiener integral $c_2F_{2, T}$ (with kernel (\ref{meq:15})) converges in $L^2(\Omega)$ to $b(H, q) \times R^{H'}$ with the explicit constant $b(H, q)$ as in (\ref{meq:19}). (Note that $c_2 =q!$). The finite-dimensional convergence then follows from (\ref{m4}). The proof of Theorem \ref{mTheorem1} is achieved.
\qed

\section{Proof of Theorem \ref{mTheorem2}}

We follow the same route as for the proof of Theorem \ref{mTheorem1}, with some slight modifications. Here, the chaos decomposition of $G_T^{1, H}$ contains uniquely the term $F_{2, T}$ obtained for $q=1$ and $ r = 0$. Its spectral representation is as follows:
\begin{align*}
F_{2, T}(t)&= \frac{H(2H-1)}{\beta(H-\frac{1}{2}, 2-2H)}\frac{\Gamma^2(H-\frac{1}{2})}{2\pi} \int_{\mathbb{R}^2}W(d\lambda_1)W(d\lambda_2) |\lambda_1|^{\frac{1}{2} - H}|\lambda_2|^{\frac{1}{2} - H}\nonumber\\ 
&\qquad\qquad\qquad  \times \int_0^t ds e^{i(\lambda_1 +\lambda_2)s}  \int_{[0, Ts]^2}d\xi_1d\xi_2 e^{-i(\lambda_1\frac{\xi_1}{T} + \lambda_2\frac{\xi_2}{T})}x(\xi_1)x(\xi_2).
\end{align*}
It is easily seen that that $F_{2, T}$ is well-defined if and only if $3/4 < H <1$. The same arguments as in the proof of Theorem \ref{mTheorem1} yield
\begin{equation}\label{meq:d}
G_T^{1, H}(t) = F_{2, T}(t) \longrightarrow \frac{H(2H-1)}{\sqrt{(H-1/2)(4H-3)}}\bigg(\int_0^\infty x(u)du\bigg)^2 \times R^{H''}(t)
\end{equation}
in $L^2(\Omega)$ as $T \to \infty$, thus completing the proof of the theorem. \qed

\section*{Acknowledgements} 
I would like to sincerely thank my supervisor Ivan Nourdin, who has led the way on this work. I did appreciate his advised tips and encouragement for my first research work. I also warmly thank Frederi Viens for interesting discussions and several helpful comments about this work. Also, I would like to  thank my friend, Nguyen Van Hoang, for his help in proving the identity about $I$ in the appendix. Finally, I deeply thank an anonymous referee for a very careful and thorough reading of this work, and for her/his constructive remarks.

\section*{Appendix}
 
The following identity has been used at the end of the proof of Lemma \ref{mLemma2} and also appeared in the proof of \cite[Lemma 6.2]{mTaqqu}. 

For all $H_0 \in (1/2, 1)$, we have
$$I: = \int_0^\infty e^{-iu}u^{H_0 - \frac{3}{2}}du =  e^{-i\frac{\pi}{2}(H_0 - \frac{1}{2})} \Gamma(H_0 - \frac{1}{2}).$$
\begin{proof}
First, observe that
$$u^{H_0 - \frac{3}{2}} = \frac{1}{\Gamma(\frac{3}{2} - H_0)} \int_0^\infty e^{-tu}t^{\frac{1}{2} - H_0}dt.$$
Then, Fubini's theorem yields
\begin{align*}
I& =\frac{1}{\Gamma(\frac{3}{2} - H_0)} \int_0^\infty du e^{-iu}  \int_0^\infty dt e^{-tu}t^{\frac{1}{2} - H_0}\\ 
& = \frac{1}{\Gamma(\frac{3}{2} - H_0)} \int_0^\infty dt \text{ } t^{\frac{1}{2} - H_0} \int_0^\infty du e^{-u(t+i)}\\
&=  \frac{1}{\Gamma(\frac{3}{2} - H_0)} \int_0^\infty  t^{\frac{1}{2} - H_0} \frac{1}{t+i} dt=  \frac{1}{\Gamma(\frac{3}{2} - H_0)} \int_0^\infty  \frac{t^{\frac{1}{2} - H_0}(t-i)}{t^2+1}dt\\
& = \frac{1}{\Gamma(\frac{3}{2} - H_0)} \bigg( \int_0^\infty  \frac{t^{\frac{3}{2} - H_0}}{t^2+1}dt - i \int_0^\infty  \frac{t^{\frac{1}{2} - H_0}}{t^2+1}dt \bigg).
\end{align*}
A change of variables $t = \sqrt{u}$ and $v= \frac{u}{u+1}$ leads to 
\begin{align*}
\int_0^\infty  \frac{t^{\frac{3}{2} - H_0}}{t^2+1}dt& = \frac{1}{2}\int_0^\infty \frac{u^{\frac{1-2H_0}{4}}}{u+1}du =  \frac{1}{2} \int_0^1 v^{\frac{1-2H_0}{4}}(1-v)^{\frac{2H_0 -  5}{4}} \\ 
& =  \frac{1}{2}  \beta \Big(\frac{5-2H_0 }{4}, \frac{2H_0 -  1}{4}\Big) = \frac{1}{2} \frac{\Gamma(\frac{5-2H_0 }{4})\Gamma(\frac{2H_0 -1}{4})}{\Gamma(1)}.
\end{align*}
Similarly, one also has,
$$\int_0^\infty  \frac{t^{\frac{1}{2} - H_0}}{t^2+1}dt = \frac{1}{2}  \beta \Big(\frac{3-2H_0 }{4}, \frac{2H_0 +  1}{4}\Big) = \frac{1}{2} \frac{\Gamma(\frac{3-2H_0 }{4})\Gamma(\frac{2H_0 +1}{4})}{\Gamma(1)}.$$
Furthermore, by using the identity $\Gamma(1-z)\Gamma(z) = \frac{\pi}{\sin(\pi z)}, 0<z<1$, we obtain
\begin{align*}
I&=\frac{1}{2 \Gamma(\frac{3}{2} - H_0)}   \bigg(\frac{\pi}{\sin (\frac{2H_0 -1}{4} \pi)} - i \frac{\pi}{\sin (\frac{3-2H_0}{4} \pi)}  \bigg)\\ 
& = \frac{1}{2 \Gamma(\frac{3}{2} - H_0)}   \bigg(\frac{\pi}{\sin (\frac{2H_0 -1}{4} \pi)} - i \frac{\pi}{\cos (\frac{2H_0-1}{4} \pi)}  \bigg)\\
&= \frac{\pi}{ \Gamma(\frac{3}{2} - H_0)} \frac{e^{-i\frac{\pi}{2}(H_0 - \frac{1}{2})}}{2\sin (\frac{2H_0 -1}{4} \pi)\cos(\frac{2H_0 -1}{4} \pi)}\\
& = \frac{e^{-i\frac{\pi}{2}(H_0 - \frac{1}{2})} \pi}{ \Gamma(\frac{3}{2} - H_0) \sin (\frac{2H_0 -1}{2} \pi)} = e^{-i\frac{\pi}{2}(H_0 - \frac{1}{2})} \Gamma(H_0 - \frac{1}{2}).
\end{align*}

\end{proof}
\chapter[Asymptotic behavior for quadratic variations of non-Gaussian multiparameter Hermite random fields] {Asymptotic behavior for quadratic variations of non-Gaussian multiparameter Hermite random fields}

\begin{center}
T.T. Diu Tran\\
University of Luxembourg
\end{center}

\begin{tcolorbox}
This article is under revision for \textit{Probability and Mathematical Statistics}, arXiv:1611.03674.
\end{tcolorbox}

\begin{center}
\textbf{Abstract }
\end{center}
Let $(Z^{q, \HH}_\t)_{\t \in [0, 1]^d}$ denote a $d$-parameter Hermite random field of order $q \geq 1$ and self-similarity parameter $\HH = (H_1, \ldots, H_d) \in (\frac{1}{2}, 1)^d$. This process is $\HH$-self-similar, has stationary increments and exhibits long-range dependence. Particular examples include fractional Brownian motion ($q=1$, $d=1$), fractional Brownian sheet $(q=1, d \geq 2)$, Rosenblatt process ($q=2$, $d=1$) as well as Rosenblatt sheet $(q=2, d \geq 2)$. For any $q \geq 2, d\geq 1$ and $\HH \in (\frac{1}{2}, 1)^d$  we show in this paper that a proper normalization of the quadratic variation of $Z^{q, \HH}$ converges in $L^2(\Omega)$ to a standard $d$-parameter Rosenblatt random variable with self-similarity index $\HH'' = 1+ (2\HH-2)/q$.

%

\section{Motivation and main results}

In recent years, analysing the asymptotic behaviour of power variations of self-similar stochastic processes has attracted a lot of attention. This is because they play an important role in various aspects, both in probability and statistics. As far as quadratic variations are concerned, a classical application  is to use them for the construction of efficient estimators for the self-similarity parameter (see e.g. \cite{hTudorViens1, hTudorViens2}). For a less conventional application, let us also mention the recent reference \cite{hIvanG}, in which the authors have used weighted power variations of fractional Brownian motion to compute exact rates of convergence of some approximating schemes associated to one-dimensional fractional stochastic differential equations.

In this paper, we deal with the quadratic variation in the context of  {\it multiparameter Hermite random fields}. To be more specific, let $Z^{q, \HH} = (Z^{q, \HH}_\t)_{\t \in [0, 1]^d}$ stand for the $d$-parameter Hermite random field of order $q \geq 1 $ and self-similarity parameter $\HH = (H_1,\ldots, H_d) \in (\frac{1}{2}, 1)^d$ (see Definition \ref{hDeforiginal} for the precise meaning), and consider a renormalized version of its quadratic variation, namely
\begin{equation}\label{heq:1}
V_{\NN}: = \frac{1}{\NN}\sum_{\ii=0}^{\NN-1}\bigg[\NN^{2\HH}\Big(\Delta Z^{q, \HH}_{[\frac{\ii}{\NN}, \frac{\ii +1}{\NN}]}\Big)^2 - 1 \bigg],
\end{equation} 
where $\Delta Z^{q, \HH}_{[\s, \t]}$ is the increments of $Z^{q, \HH}$ defined as
\begin{equation}\label{heq:incrementZ}
\Delta Z^{q, \HH}_{[\s, \t]} = \sum_{\rr \in \{0, 1\}^d} (-1)^{d- \sum_i r_i} Z^{q, \HH}_{\s + \rr.(\t - \s)},
\end{equation}
and where the bold notation is systematically used in presence of multi-indices (we refer to Section 2 for precise definitions). As illustrating examples, observe that (\ref{heq:incrementZ}) reduces to $\Delta Z^{q, H}_{[s, t]} = Z^{q, H}_t -Z^{q, H}_s$ when $d=1$, and to  $\Delta Z^{q, H_1, H_2}_{[\s, \t]} = Z^{q, H_1, H_2}_{t_1, t_2} - Z^{q, H_1, H_2}_{t_1, s_2} - Z^{q, H_1, H_2}_{s_1, t_2} + Z^{q, H_1, H_2}_{s_1, s_2}$ when $d=2$. 

It is well-known that each Hermite random field $Z^{q, \HH}$ is $\HH$-self-similar (that is, $ (Z^{q, \HH}_{\aa \t})_{\t \in \mathbb{R}^d} \overset{(d)}{=} (\aa^\HH Z^{q, \HH}_\t)_{\t \in \mathbb{R}^d}$ for any $\aa > 0$), has stationary increments (that is $(\Delta Z^{q, \HH}_{[0, \t]})_{\t \in \mathbb{R}^d} \overset{(d)}{=} (\Delta Z^{q, \HH}_{[\hh, \hh+ \t]})_{\t \in \mathbb{R}^d}$ for all $\hh \in\mathbb{R}^d$) and exhibits long-range dependence. Also, when $q=1$, observe that $Z^{1, \HH}$ is either the fractional Brownian motion (if $d=1$) or the fractional Brownian sheet (if $d\geq 2$); in particular, among all the Hermite random fields $Z^{q,\HH}$, it is the only one to be Gaussian. When $q=2$, we use the usual terminologies Rosenblatt process (if $d=1$) or Rosenblatt sheet (if $d \geq 2$). 

Before describing our results, let us give a brief overview of the current state of the art. Firstly, let us consider the case  $q=d=1$, that is, the case where $Z^{1, H} = B^H$ is a fractional Brownian motion with Hurst parameter $H$. The behavior of the quadratic variation of $B^H$ is well-known since the eighties, and dates back to the seminal works of Breuer and Major \cite{hBreuerMajor}, Dobrushin and Major \cite{hDobrushinMajor}, Giraitis and Surgailis \cite{hGiraitis} or Taqqu \cite{hTaqqu}. We have, as $N\to\infty$:
\begin{itemize}
\item If $H<3/4$, then 
$$N^{-1/2}\sum_{j=1}^N \bigg(N^{2H}\Big(B^H_{j/N} - B^H_{(j-1)/N}\Big)^2 - 1\bigg) \xrightarrow[]{(d)} \mathcal{N}(0, \sigma_H^2).$$
\item If $H =3/4$, then
$$(N\log{N})^{-1/2}\sum_{j=1}^N \bigg(N^{3/2}\Big(B^H_{j/N} - B^H_{(j-1)/N}\Big)^2 - 1\bigg) \xrightarrow[]{(d)} \mathcal{N}(0, \sigma_{3/4}^2).$$
\item If $H >3/4$, then
$$N^{1-2H}\sum_{j=1}^N \bigg(N^{2H}\Big(B^H_{j/N} - B^H_{(j-1)/N}\Big)^2 - 1\bigg) \xrightarrow[]{L^2(\Omega)} \text{``Rosenblatt r.v''},$$
where $\text{``Rosenblatt r.v''}$ denotes the random variable which is the value at time 1 of the Rosenblatt process.
\end{itemize}
Secondly, assume now that $q=1$ and $d =2$, that is, consider the case where $Z^{1, \HH}$ is this time a two-parameter fractional Brownian sheet with Hurst parameter $\HH = (H_1, H_2)$. According to R\'eveillac, Stauch and Tudor \cite{hReveillacTudor} and with $\varphi(N, \HH)$ a suitable scaling factor, the quadratic variation of $Z^{1, \HH}$ behaves as follows, as $N\to\infty$:
\begin{itemize}
\item If $\HH\notin (3/4, 1)^2$, then
$$\varphi (N, \HH)\sum_{i=1}^N\sum_{j=1}^N \bigg(N^{2H_1 + 2H_2}\Big(\Delta Z^{1, \HH}_{[\frac{\ii-1}{\NN}, \frac{\ii }{\NN}]}\Big)^2 - 1\bigg) \xrightarrow[]{(d)} \mathcal{N}(0, \sigma_{\HH}^2).$$
\item If $\HH \in (3/4, 1)^2$, then
\begin{align*}
\varphi (N, \HH)\sum_{i=1}^N\sum_{j=1}^N \bigg(N^{2H_1 + 2H_2}&\Big(\Delta Z^{1, \HH}_{[\frac{\ii-1}{\NN}, \frac{\ii }{\NN}]}\Big)^2 - 1\bigg)\\ 
&  \xrightarrow[]{L^2(\Omega)} \text{``two-parameter Rosenblatt r.v"},
\end{align*}
where $\text{``two-parameter Rosenblatt r.v''}$ means the value at point ${\bf 1}=(1, 1)$ of the two-parameter Rosenblatt sheet. 
\end{itemize}
Here, we observe the following interesting phenomenon: the limit law in the mixture case (that is, when $H_1 \leq 3/4$ and $H_2 > 3/4$) is Gaussian. For the simplicity of exposition, above we have only described what happens when $d=2$. But the asymptotic behaviour for the quadratic variation of $Z^{1, \HH}$ is actually known for any value of the dimension $d\geq 2$, and we refer to Pakkanen and R\'{e}veillac \cite{hReveillac2, hReveillac3, hReveillac1} for precise statements.

Let us finally review the existing literature about the quadratic variation of $Z^{q,\mathbf{H}}$ in the {\it non}-Gaussian case, that is, when $q\geq 2$. It is certainly because it is a more difficult case to deal with that only the case where $d=1$ has been studied so far. Chronopoulou, Tudor and Viens have shown in \cite{hTudorViens1} (see also \cite{hTudorViens2, hTudor}) that, properly renormalized, the quadratic variation of $Z^{q, H}$ converges in $L^2(\Omega)$, for \textit{any} $q \geq 2$ and \textit{any} value of $H \in (1/2, 1)$, to the Rosenblatt random variable. A consequence of this finding is that fractional Brownian motion is the only Hermite process ($d=1$) for which there exists a range of parameters such that its quadratic variation exhibits {\it normal} convergence; indeed, for all the other Hermite processes, \cite{hTudorViens1} shows that we have the convergence towards a non-Gaussian random variable belonging to the second Wiener chaos. 

In the present paper, we study what happens in the remaining cases, that is, when $q$ and $d$ are both bigger or equal than 2.
Thanks to our main result, Theorem \ref{hmain}, we now have a complete picture for the asymptotic behaviour of the quadratic variation of {\it any} Hermite random field.
\eject
\begin{Theorem} \label{hmain}
Fix $q \geq 2$, $d\geq 1$ and $\HH \in (\frac{1}{2}, 1)^d$. Let $Z^{q, \HH}$ be a $d$-parameter Hermite random field of order $q$ and self-similarity parameter $\HH$ (see Definition \ref{hDeforiginal}). Then $c_{1,\HH}^{-1/2}\NN^{(2-2\HH)/q}(q!q)^{-1} V_{\NN}$ converges, in $L^2(\Omega)$, to the standard $d$-parameter Rosenblatt sheet with self-similarity parameter $1+ (2\HH-2)/q$ evaluated at time $\bf{1}$, where $c_{1, \HH}$ given by (\ref{heq:3.5}).
\end{Theorem}

 Our proof of Theorem \ref{hmain} follows a strategy introduced by Tudor and Viens in \cite{hTudorViens2}, based on the use of chaotic expansion into multiple Wiener-It\^{o} integrals. Let us sketch it. Since the Hermite random field $Z^{q, \HH}$ is an element of the $q$-th Wiener chaos, we can firstly rely on the product formula for multiple integrals to obtain that the quadratic variation $V_{\NN}$ can be decomposed into a sum of multiple integrals of even orders from $2$ to $2q$, see Section 3.3.1. We are thus left to analyse the behavior of each chaos component. As we will prove in Section 3.3.2,  the dominant term of $V_\NN$ (after proper normalization) is the term in the second Wiener chaos, that is, all the other terms in the chaotic expansion are asymptotically negligible. Finally, by using the isometric property of multiple Wiener-It\^{o} integrals and checking the $L^2(([0, 1]^d)^2)$ convergence of its kernel, we will prove in Section 3.3.3 that the projection onto the second Wiener chaos converges in $L^2(\Omega)$ to the $d$-parameter Rosenblatt random variable, which will lead to the convergence of the normalization of $V_\NN$ to the same random variable.

In conclusion, it is worth pointing out that, irrespective of the self-similarity parameter, the (properly normalized) quadratic variation of any \textit{non-Gaussian} multiparameter Hermite random fields exhibits a convergence to a random variable belonging to the second Wiener chaos. It is in strong contrast with what happens in the Gaussian case ($q=1$), where either central or non-central limit theorems may arise, depending on the value of the self-similarity parameter.

The remainder of the paper is structured as follows. Section $3.2$ contains some preliminaries and useful notation. The proof of our main result, namely Theorem \ref{hmain}, is then provided in Section $3.3$.

\section{Preliminaries}

This section describes the notation and the mathematical objects (together with their main properties) that are used throughout this paper.

\subsection{Notation}

Fix an integer $d\geq 1$. In what follows,  we shall systematically use bold notation when dealing with multi-indexed quantities. We thus write $\aa = (a_1, a_2, \ldots, a_d)$, $\aa\bb = (a_1b_1, a_2b_2, \ldots, a_db_d)$ or $\aa/\bb = (a_1/b_1, a_2/b_2, \ldots, a_d/b_d)$. Similarly, $[\aa, \bb] = \prod_{i=1}^d[a_i, b_i],  (\aa, \bb) = \prod_{i=1}^d(a_i, b_i)$. Summation is as follows: $\sum_{\ii=1}^\NN a_\ii = \sum_{i_1 =1}^{N_1}\sum_{i_2 =1}^{N_2}\ldots \sum_{i_d =1}^{N_d}a_{i_1,i_2,\ldots,i_d}$ whereas, for products, we shall write $\aa^\bb = \prod_{i=1}^d a_i^{b_i}$. Finally, we shall write $ \aa < \bb$ (resp. $ \aa \leq \bb$) whenever $a_1 < b_1$, $a_2 < b_2$, $\ldots$, $a_d < b_d$
(resp. $a_1 \leq b_1$, $a_2 \leq b_2$, $\ldots$, $a_d \leq b_d$).

\subsection{Multiple Wiener-It\^{o} integrals}
We will now briefly review the theory of multiple Wiener-It\^{o} integrals with respect to the Brownian sheet, as described e.g. in Nualart's book \cite{hNualart} (chapter 1 therein) or in \cite[Section 3]{hReveillac2}. Let $f \in L^2((\mathbb{R}^{d})^q)$ and let us denote by $I_q^W(f)$ the $q$-fold multiple Wiener-It\^{o} integral of $f$ with respect to the standard two-sided Brownian sheet $(W_\t)_{\t \in \mathbb{R}^d}$. In symbols, such an integral is written
\begin{equation}\label{heq:abc}
I_q^W(f)  = \int_{(\mathbb{R}^d)^q}  dW_{\uu_1}\ldots dW_{\uu_q} f(\uu_1, \ldots, \uu_q).
\end{equation}
Moreover, one has $I_q^W(f) = I_q^W(\widetilde{f})$, where $\widetilde{f}$ is the symmetrization of $f$ defined by
\begin{equation}\label{heq:dola}
\widetilde{f}(\uu_1, \ldots, \uu_q) = \frac{1}{q!}\sum_{\sigma \in \mathfrak{S}_q} f(\uu_{\sigma(1)}, \ldots, \uu_{\sigma(q)}).
\end{equation}
The set of random variables of the form $I_q^W(f)$, when $f$ runs over $L^2((\mathbb{R}^{d})^q)$, is called the $q$th Wiener chaos of $W$. Furthermore, if $f \in L^2((\mathbb{R}^{d})^p)$ and $g \in L^2((\mathbb{R}^{d})^q)$ are two symmetric functions, then
\begin{equation}\label{heq:P1}
I_p^W(f)I_q^W(g) = \sum_{r=0}^{p \wedge  q} r!\binom{p}{r}\binom{q}{r}I_{p+q-2r}^W(f \widetilde{\otimes}_r g),
\end{equation}
where the contraction $f \otimes_r g$, which belongs to $L^2((\mathbb{R}^{d})^{p+q-2r})$ for every $r = 0, 1, \ldots, p \wedge q$, is given by
\begin{align}\label{heq:P2}
f \otimes_r g& (\uu_1, \ldots, \uu_{p-r}, \vv_1, \ldots, \vv_{q-r}) \nonumber\\ 
& = \int_{(\mathbb{R}^{d})^r} d\aa_1 \ldots d\aa_r f(\uu_1, \ldots, \uu_{p-r}, \aa_1, \ldots, \aa_r) g(\vv_1, \ldots, \vv_{q-r}, \aa_1, \ldots, \aa_r) 
\end{align}
and $f \widetilde{\otimes}_r g$ stands for the symmetrization of $f \otimes_r g$ (according to the notation introduced in (\ref{heq:dola})). 
For any $r= 0, \ldots, p\wedge q$, Cauchy-Schwarz inequality yields
\begin{equation}\label{heq:3}
\| f \widetilde{\otimes}_r g\|_{L^2((\mathbb{R}^{d})^{p+q-2r})} \leq \| f \otimes_r g\|_{L^2((\mathbb{R}^{d})^{p+q-2r})} \leq \|f\|_{L^2((\mathbb{R}^{d})^p)}\|g\|_{L^2((\mathbb{R}^{d})^{q})}.
\end{equation}
Also, $f \otimes_p  g = \left\langle f, g \right\rangle_{L^2((\mathbb{R}^{d})^p)}$ when $q=p$.
Furthermore, multiple Wiener-It\^{o} integrals satisfy the following isometry and orthogonality properties
$$E[I_p^W(f)I_q^W(g)] =
\begin{cases}
 & p! \big\langle \widetilde{f}, \widetilde{g} \big\rangle_{L^2((\mathbb{R}^{d})^{p})} \qquad\text{if } p=q\\
& 0 \qquad\qquad\qquad\quad\quad \text{if } p \ne q.
\end{cases}$$

\subsection{Multiparameter Hermite Random Fields}

Let us now introduce our main object of interest in this paper, the so-called multiparameter Hermite random field. We follow the definition given by Tudor in \cite[Chapter 4]{hTudor}.

\begin{Definition}\label{hDeforiginal}
{\it
Let $q,d \geq 1$ be two integers and let $\HH = (H_1, \ldots, H_d)$ be a vector belonging to $(\frac{1}{2}, 1)^d$.
The $d$-parameter Hermite random field of order $q$ and self-similarity parameter $\HH$ means any random field of the form
\begin{align}\label{hDef1}
Z^{q, \HH}(\t)&= c_{q, \HH} \int_{(\mathbb{R}^{d})^{ q}} dW_{u_{1,1}, \ldots, u_{1,d}} \ldots dW_{u_{q,1}, \ldots, u_{q,d}} \nonumber\\ 
&\qquad\times \bigg(\int_0^{t_1} da_1 \ldots \int_0^{t_d} da_d \prod_{j=1}^q (a_1-u_{j,1})_+^{-(\frac{1}{2} + \frac{1-H_1}{q})}\ldots (a_d-u_{j,d})_+^{-(\frac{1}{2} + \frac{1-H_d}{q})}\bigg)\nonumber\\
& = c_{q, \HH}\int_{(\mathbb{R}^{d})^{q}}dW_{\uu_1} \ldots dW_{\uu_q} \int_0^\t d\aa \prod_{j=1}^q (\aa - \uu_j)_+^{-(\frac{1}{2} + \frac{1-\HH}{q})},
\end{align}
where $x_+ = \max(x, 0)$, $W$ is a standard two-sided Brownian sheet, and $c(q, \HH)$ is the unique positive constant depending only on $q$ and $\HH$
chosen so that $E[Z^{q, \HH}(\textbf{1})^2] = 1$. }
\end{Definition}
The above integral  (\ref{hDef1}) represents a multiple Wiener-It\^{o} integral of the form (\ref{heq:abc}).

In many occasions (for instance when one wants to simulate $Z^{q,\HH}$, or when one looks for constructing a stochastic calculus with respect to it), the following finite-time representation for $Z^{q,\HH}$ may also be of interest:
\begin{align}\label{hDef2}
&Z^{q, \HH}(\t) \overset{(d)}{=} b_{q, \HH} \int_0^{t_1} \ldots \int_0^{t_d} dW_{u_{1,1}, \ldots, u_{1,d}} \ldots \int_0^{t_1}\ldots\int_0^{t_d} dW_{u_{q,1}, \ldots, u_{q,d}} \nonumber\\ 
&\hspace{2.5cm}\times \bigg(\int_{u_{1,1} \vee \ldots \vee u_{q,1}}^{t_1} da_1 \partial_1K^{H_1'}(a_1, u_{1,1}) \ldots \partial_1K^{H_1'}(a_1, u_{q,1})\bigg) \nonumber \\
& \hspace{6cm} \vdots \nonumber\\
& \hspace{2.5cm}\times \bigg(\int_{u_{1,d} \vee \ldots \vee u_{q,d}}^{t_d} da_d \partial_1K^{H_d'}(a_d, u_{1,d}) \ldots \partial_1K^{H_d'}(a_d, u_{q,d})\bigg) \nonumber\\
& = b_{q, \HH}\int_{[0, \t]^{q}}dW_{\uu_1}\ldots dW_{\uu_q} \prod_{j=1}^d\int_{u_{1,j} \vee \ldots \vee u_{q,j}}^{t_j} da \partial_1K^{H'_j}(a, u_{1,j}) \ldots \partial_1K^{H'_j}(a, u_{q,j}).
\end{align}
In (\ref{hDef2}), $K^{H}$ stands for the usual kernel appearing in the
classical expression of  the fractional Brownian motion $B^H$ as a Volterra integral with respect to Brownian motion (see e.g. \cite{hIvan, hIvan1}), that is,
$B^H_t = \int_0^t K^H(t, s)dB_s,$
whereas
\begin{equation}\label{heq:bqH}
b_{q, \HH} :=\frac{(\HH(2\HH-1))^{1/2}}{(q!(\HH'(2\HH' -1))^q)^{1/2}} =(\sqrt{q!})^{d-1} \prod_{j=1}^d \frac{(H_j(2H_j-1))^{1/2}}{(q!(H'_j(2H'_j -1))^q)^{1/2}} 
\end{equation}
is the unique positive constant ensuring that $E[Z^{q, \HH}(\textbf{1})^2] = 1$, where
\begin{equation}\label{heq:HH}
\HH' := 1+ \frac{\HH-1}{q}\quad\big(\Longleftrightarrow (2\HH' -2)q = 2\HH-2\big).
\end{equation}
For a proof of (\ref{hDef2}) when $d=2$, we refer to Tudor \cite[Chapter 4]{hTudor}. Extension to any value of $d$ as presented here is straightforward.

\section{Proof of Theorem \ref{hmain}}

We are now in a position to give the proof of our Theorem \ref{hmain}.
It is divided into three steps.

\subsection{Expanding into Wiener chaos}

In preparation of analysing the quadratic variation (\ref{heq:1}), let us find an explicit expression for the chaos decomposition of $V_\NN$. Using (\ref{hDef2}) and proceeding by induction on the dimension $d$, we can write $\Delta Z^{q, \HH}_{[\frac{\ii}{\NN}, \frac{\ii +1}{\NN}]}$ as a $q$-th Wiener It\^{o} integral with respect to the standard two-sided Brownian sheet $(W_\t)_{\t \in \mathbb{R}^d}$ as follows: for every $0 \leq \ii \leq \NN-1$, one has
\begin{equation}\label{heq:Wchaos}
\Delta Z^{q, \HH}_{[\frac{\ii}{\NN}, \frac{\ii +1}{\NN}]} = I_q(f_{\ii,\NN}),
\end{equation} 
where
\begin{equation}\label{heq:Wchaos1}
f_{\ii,\NN}(\xx_1, \ldots, \xx_q) = b_{q, \HH}\prod_{j=1}^d f_{i_j,N_j}(x_{1,j}, \ldots, x_{q,j}),
\end{equation}
with $f_{i,N}(x_{1}, \ldots, x_{q})$ denoting the expression
\begin{align}\label{hexpression}
&\mathbbm{1}_{[0, \frac{i+1}{N}] }(x_{1} \vee \ldots \vee x_{q}) \int_{x_{1} \vee \ldots \vee x_{q}}^{\frac{i+1}{N}}du \partial_1K^{H'}(u, x_{1}) \ldots \partial_1K^{H'}(u, x_{q}) \nonumber\\
& - \mathbbm{1}_{[0, \frac{i}{N}] }(x_{1} \vee \ldots \vee x_{q}) \int_{x_{1} \vee \ldots \vee x_{q}}^{\frac{i}{N}}du \partial_1K^{H'}(u, x_{1}) \ldots \partial_1K^{H'}(u, x_{q}),
\end{align}
and with $b_{q, \HH}$ and $\HH'$ given by (\ref{heq:bqH}) and (\ref{heq:HH}) respectively. Indeed, for $d=1$, see \cite[Section 3, p.8]{hTudorViens1}, it reduces to
$$\Delta Z^{q, H}_{[\frac{i}{N}, \frac{i +1}{N}]} = Z^{q, H}_{\frac{i+1}{N}} - Z^{q, H}_{\frac{i}{N}} =b_{q, H} I_q(f_{i,N}),$$
while for $d=2$, it is easy to verify that
$$\Delta Z^{q, \HH}_{[\frac{\ii}{\NN}, \frac{\ii +1}{\NN}]} = Z^{q, H_1, H_2}_{\frac{i+1}{N}, \frac{j+1}{M}} - Z^{q, H_1, H_2}_{\frac{i}{N}, \frac{j+1}{M}} - Z^{q, H_1, H_2}_{\frac{i+1}{N}, \frac{j}{M}} + Z^{q, H_1, H_2}_{\frac{i}{N}, \frac{j}{M}} = I_q(f_{i,j,N,M})$$
where
{\allowdisplaybreaks
\begin{align*}
&f_{i,j,N,M} (x_1, y_1, \ldots, x_q, y_q)\\ 
& = b_{q, H_1, H_2} \mathbbm{1}_{[0, \frac{i+1}{N}] }(x_1 \vee \ldots \vee x_q) \int_{x_1 \vee \ldots \vee x_q}^{\frac{i+1}{N}}du \partial_1K^{H_1'}(u, x_1) \ldots \partial_1K^{H_1'}(u, x_q) \\
& \qquad\qquad\times \mathbbm{1}_{[0, \frac{j+1}{M}] }( y_1 \vee \ldots \vee y_q) \int_{y_1 \vee \ldots \vee y_q}^{\frac{j+1}{M}} dv \partial_1K^{H_2'}(v, y_1) \ldots \partial_1K^{H_2'}(v, y_q)\\
& - b_{q, H_1, H_2} \mathbbm{1}_{[0, \frac{i+1}{N}] }(x_1 \vee \ldots \vee x_q) \int_{x_1 \vee \ldots \vee x_q}^{\frac{i+1}{N}}du \partial_1K^{H_1'}(u, x_1) \ldots \partial_1K^{H_1'}(u, x_q) \\
& \qquad\qquad\times \mathbbm{1}_{[0, \frac{j}{M}] }( y_1 \vee \ldots \vee y_q) \int_{y_1 \vee \ldots \vee y_q}^{\frac{j}{M}} dv \partial_1K^{H_2'}(v, y_1) \ldots \partial_1K^{H_2'}(v, y_q)\\
& - b_{q, H_1, H_2} \mathbbm{1}_{[0, \frac{i}{N}] }(x_1 \vee \ldots \vee x_q) \int_{x_1 \vee \ldots \vee x_q}^{\frac{i}{N}}du \partial_1K^{H_1'}(u, x_1) \ldots \partial_1K^{H_1'}(u, x_q) \\
& \qquad\qquad\times \mathbbm{1}_{[0, \frac{j+1}{M}] }( y_1 \vee \ldots \vee y_q) \int_{y_1 \vee \ldots \vee y_q}^{\frac{j+1}{M}} dv \partial_1K^{H_2'}(v, y_1) \ldots \partial_1K^{H_2'}(v, y_q)\\
& + b_{q, H_1, H_2} \mathbbm{1}_{[0, \frac{i}{N}] }(x_1 \vee \ldots \vee x_q) \int_{x_1 \vee \ldots \vee x_q}^{\frac{i}{N}}du \partial_1K^{H_1'}(u, x_1) \ldots \partial_1K^{H_1'}(u, x_q) \\
& \qquad\qquad\times \mathbbm{1}_{[0, \frac{j}{M}] }( y_1 \vee \ldots \vee y_q) \int_{y_1 \vee \ldots \vee y_q}^{\frac{j}{M}} dv \partial_1K^{H_2'}(v, y_1) \ldots \partial_1K^{H_2'}(v, y_q)\\
& =  b_{q, H_1, H_2}f_{i,N}(x_1, \ldots, x_q)f_{j, M}(y_1, \ldots, y_q).
\end{align*}
}The last equality above is obtained by grouping each term of $f_{i, j, N, M}$ together. Suppose that the expression (\ref{heq:Wchaos}), (\ref{heq:Wchaos1}) is true for $d$, that is, the kernel of $\Delta Z^{q, \HH}_{[\frac{\ii}{\NN}, \frac{\ii +1}{\NN}]}$ is equal to 
\begin{align*}
&b_{q, \HH} \sum_{(r_1, \ldots, r_d) \in \{0, 1\}^d} (-1)^{d - \sum_{i=1}^{d}r_i} \prod_{j=1}^d \mathbbm{1}_{[0, \frac{i_j + r_j}{N_j}] }(x_{1, j} \vee \ldots \vee x_{q, j}) \\
&\hspace{2cm} \times \int_{x_{1, j} \vee \ldots \vee x_{q, j}}^{\frac{i_j+r_j}{N_j}}du \partial_1K^{H_j'}(u, x_{1, j}) \ldots \partial_1K^{H_j'}(u, x_{q, j}) \\
&= b_{q, \HH} \prod_{j=1}^d f_{i_j, N_j}(x_{1, j}, \ldots, x_{q, j}).
\end{align*}
Then, for the case $d+1$ we have 
\begin{align*}
\Delta Z^{q, \HH}_{[\frac{\ii}{\NN}, \frac{\ii +1}{\NN}]}& = \sum_{\rr \in \{0, 1\}^{d+1}} (-1)^{d+1 - \sum_{i=1}^{d+1}r_i}Z^{q, \HH}_{\frac{\ii+\rr}{\NN}}\\ 
& = \sum_{(r_1, \ldots, r_d) \in \{0, 1\}^d} (-1)^{d - \sum_{i=1}^{d}r_i} Z^{q, \HH}_{\big(\frac{i_1+r_1}{N_1}, \ldots, \frac{i_d+r_d}{N_d}, \frac{i_{d+1}+1}{N_{d+1}} \big)} \\
& \qquad\qquad +  \sum_{(r_1, \ldots, r_d) \in \{0, 1\}^d} (-1)^{d +1 - \sum_{i=1}^{d}r_i} Z^{q, \HH}_{\big(\frac{i_1+r_1}{N_1}, \ldots, \frac{i_d+r_d}{N_d}, \frac{i_{d+1}}{N_{d+1}} \big)}\\
& = \sum_{(r_1, \ldots, r_d) \in \{0, 1\}^d} (-1)^{d - \sum_{i=1}^{d}r_i} \Big( Z^{q, \HH}_{\big(\frac{i_1+r_1}{N_1}, \ldots, \frac{i_d+r_d}{N_d}, \frac{i_{d+1}+1}{N_{d+1}} \big)} - Z^{q, \HH}_{\big(\frac{i_1+r_1}{N_1}, \ldots, \frac{i_d+r_d}{N_d}, \frac{i_{d+1}}{N_{d+1}} \big)}  \Big).
\end{align*}
It belongs to the $q$-Wiener chaos with the kernel  $f_{\ii, \NN}$ given by 
\begin{align*}
f_{\ii, \NN}& = b_{q, \HH} \sum_{(r_1, \ldots, r_d) \in \{0, 1\}^d} (-1)^{d - \sum_{i=1}^{d}r_i} \prod_{j=1}^d \mathbbm{1}_{[0, \frac{i_j + r_j}{N_j}] }(x_{1, j} \vee \ldots \vee x_{q, j})\\
& \hspace{5cm} \times \int_{x_{1, j} \vee \ldots \vee x_{q, j}}^{\frac{i_j+r_j}{N_j}}du \partial_1K^{H_j'}(u, x_{1, j}) \ldots \partial_1K^{H_j'}(u, x_{q, j}) \\ 
& \times \bigg(\int_{x_{1, d+1} \vee \ldots \vee x_{q, d+1}}^{\frac{i_{d+1}+ 1}{N_{d+1}}}du' \partial_1K^{H'_{d+1}}(u', x_{1, d+1}) \ldots \partial_1K^{H'_{d+1}}(u', x_{q, d+1}) \\
& \hspace{2.5cm} - \int_{x_{1, d+1} \vee \ldots \vee x_{q, d+1}}^{\frac{i_{d+1}}{N_{d+1}}}du' \partial_1K^{H'_{d+1}}(u', x_{1, d+1}) \ldots \partial_1K^{H'_{d+1}}(u', x_{q, d+1})\bigg).
\end{align*}
By the induction hypothesis, one gets
$f_{\ii, \NN} =  b_{q, \HH}  \prod_{j=1}^{d+1} f_{i_j, N_j}(x_{1, j}, \ldots, x_{q, j}),$
which is our desired expression.

Next, by applying the product formula (\ref{heq:P1}), we  can write
\begin{equation}\label{heq:dola1}
\Big(\Delta Z^{q, \HH}_{[\frac{\ii}{\NN}, \frac{\ii +1}{\NN}]}\Big)^2 - E\Big[\Big(\Delta Z^{q, \HH}_{[\frac{\ii}{\NN}, \frac{\ii +1}{\NN}]}\Big)^2 \Big]= \sum_{r=0}^{q-1} r!\binom{q}{r}^2 I_{2q-2r}(f_{\ii,\NN} \widetilde{\otimes}_r f_{\ii,\NN}).
\end{equation}
Let us compute the contractions appearing in the right-hand side of (\ref{heq:dola1}). For every $0 \leq r \leq q-1$, we have
{\allowdisplaybreaks
\begin{align}\label{heq:3.1.1}
&(f_{\ii,\NN} \otimes_r f_{\ii,\NN})(\xx_1, \ldots, \xx_{2q-2r}) \nonumber \\ 
& = \int_{([0, 1]^{d})^{  r}}d\aa_1\ldots d\aa_r f_{\ii,\NN}(\xx_1, \ldots, \xx_{q-r}, \aa_1, \ldots, \aa_r) \nonumber\\
&\qquad\qquad\qquad\qquad \times f_{\ii,\NN}(\xx_{q-r+1}, \ldots, \xx_{2q-2r}, \aa_1, \ldots, \aa_r) \nonumber\\
&= b_{q, \HH}^2 \int_{([0, 1]^{d })^{ r}}d\aa_1\ldots d\aa_r \prod_{j=1}^d f_{i_j, N_j}(x_{1,j}, \ldots, x_{q-r,j}, a_{1,j}, \ldots, a_{r,j}) \nonumber\\
&\qquad\qquad\qquad\qquad \times  \prod_{j=1}^d f_{i_j, N_j}(x_{q-r+1,j}, \ldots, x_{2q-2r,j}, a_{1,j}, \ldots, a_{r,j}) \nonumber\\
& = b_{q, \HH}^2 \prod_{j=1}^d (f_{i_j ,N_j} \otimes_r f_{i_j ,N_j})(x_{1,j}, \ldots, x_{2q-2r,j}),
\end{align}
}where 
\begin{align}\label{heq:3.1}
&(f_{i, N} \otimes_r f_{i, N})(x_{1}, \ldots, x_{2q-2r})= (H'(2H' -1))^r \nonumber\\
& \times \bigg\{\mathbbm{1}_{[0, \frac{i_+1}{N}] }(x_{1}\vee \ldots x_{q-r}) \int_{x_{1}\vee \ldots x_{q-r}}^{\frac{i+1}{N}}du \partial_1K^{H'}(u, x_{1})\ldots \partial_1K^{H'}(u, x_{q-r}) \nonumber\\
&\qquad\times\mathbbm{1}_{[0, \frac{i+1}{N}]}(x_{q-r+1}\vee \ldots x_{2q-2r}) \int_{x_{q-r+1}\vee \ldots x_{2q-2r}}^{\frac{i+1}{N}}du' \partial_1K^{H'}(u', x_{q-r+1})\ldots \nonumber \\
&\hspace{8cm}\ldots \partial_1K^{H'}(u', x_{2q-2r}) |u-u'|^{(2H'-2)r} \nonumber\\
&\quad -\mathbbm{1}_{[0, \frac{i+1}{N}] }(x_{1}\vee \ldots x_{q-r}) \int_{x_{1}\vee \ldots x_{q-r}}^{\frac{i+1}{N}}du \partial_1K^{H'}(u, x_{1})\ldots \partial_1K^{H'}(u, x_{q-r})\nonumber\\
&\qquad\times\mathbbm{1}_{[0, \frac{i}{N}]}(x_{q-r+1}\vee \ldots x_{2q-2r}) \int_{x_{q-r+1}\vee \ldots x_{2q-2r}}^{\frac{i}{N}}du' \partial_1K^{H'}(u', x_{q-r+1})\ldots  \nonumber\\
&\hspace{8cm} \ldots \partial_1K^{H'}(u', x_{2q-2r}) |u-u'|^{(2H'-2)r} \nonumber\\
&\quad -  \mathbbm{1}_{[0, \frac{i}{N}] }(x_{1}\vee \ldots x_{q-r}) \int_{x_{1}\vee \ldots x_{q-r}}^{\frac{i}{N}}du \partial_1K^{H'}(u, x_{1})\ldots \partial_1K^{H'}(u, x_{q-r}) \nonumber\\
&\qquad\times\mathbbm{1}_{[0, \frac{i+1}{N}]}(x_{q-r+1}\vee \ldots x_{2q-2r}) \int_{x_{q-r+1}\vee \ldots x_{2q-2r}}^{\frac{i+1}{N}}du' \partial_1K^{H'}(u', x_{q-r+1})\ldots \nonumber \\
&\hspace{8cm}\ldots \partial_1K^{H'}(u', x_{2q-2r}) |u-u'|^{(2H'-2)r} \nonumber\\
&\quad+\mathbbm{1}_{[0, \frac{i}{N}] }(x_{1}\vee \ldots x_{q-r}) \int_{x_{1}\vee \ldots x_{q-r}}^{\frac{i}{N}}du \partial_1K^{H'}(u, x_{1})\ldots \partial_1K^{H'}(u, x_{q-r}) \nonumber\\
&\qquad\times\mathbbm{1}_{[0, \frac{i}{N}]}(x_{q-r+1}\vee \ldots x_{2q-2r}) \int_{x_{q-r+1}\vee \ldots x_{2q-2r}}^{\frac{i}{N}}du' \partial_1K^{H'}(u', x_{q-r+1})\ldots \nonumber\\
&\hspace{8cm}\ldots \partial_1K^{H'}(u', x_{2q-2r}) |u-u'|^{(2H'-2)r}\bigg\}.
\end{align}
(See \cite[page 10]{hTudorViens1} for a detailed computation of the expression (\ref{heq:3.1}).) Moreover, since $Z^{q,\HH}$ is $\HH$-self-similar and has stationary increments, one has
$$\Delta Z^{q, \HH}_{[\frac{\ii}{\NN}, \frac{\ii +1}{\NN}]} \overset{(d)}{=} \NN^{-\HH} \Delta Z^{q, \HH}_{[\ii, \ii+1]}\overset{(d)}{=} \NN^{-\HH}Z^{q, \HH}_{[0,\bf{1}]}.$$
It follows that 
$$E\bigg[\NN^{2\HH}\Big(\Delta Z^{q, \HH}_{[\frac{\ii}{\NN}, \frac{\ii +1}{\NN}]}\Big)^2\bigg] = E[Z^{q, \HH}(\textbf{1})^2] = 1.$$
As a consequence, we have 
\begin{equation}\label{heq:VN}
V_{\NN} = F_{2q, \NN} + c_{2q-2}F_{2q-2, \NN} + \ldots + c_4F_{4, \NN} + c_2F_{2, \NN}.
\end{equation}
where
$c_{2q-2r} = r! \binom{q}{r}^2$, $r=0,\ldots, q-1$,
are the combinator constants coming from the product formula, and
\begin{equation}\label{heq:3.4}
F_{2q-2r, \NN}: = \NN^{2\HH -1}I_{2q-2r}\bigg(\sum_{\ii=0}^{\NN-1} f_{\ii,\NN} \widetilde{\otimes}_r f_{\ii,\NN}\bigg),
\end{equation}
for the kernels $f_{\ii,\NN} \otimes_r f_{\ii,\NN}$ computed in (\ref{heq:3.1.1})-(\ref{heq:3.1}).

\subsection{Evaluating the $L^2(\Omega)$-norm}

Set
\begin{equation}\label{heq:3.5}
c_{1, \HH}= \frac{2! 2^d b_{q, \HH}^4 (\HH'(2\HH'-1))^{2q}}{(4\HH' - 3)(4\HH'-2)[(2\HH'-2)(q-1)+1]^2[(\HH'-1)(q-1) + 1]^2}.
\end{equation}
We claim that 
\begin{equation}\label{heq:30}
\lim_{\NN \to \infty} E[c_{1, \HH}^{-1}\NN^{2(2-2\HH')}c_2^{-2}V_{\NN}^2] = 1.
\end{equation}

Let us prove (\ref{heq:30}). Due to the orthogonality property for Wiener chaoses of different orders, it is sufficient to evaluate the $L^2(\Omega)$-norm of each multiple Wiener-It\^{o} integrals appearing in the chaotic decomposition (\ref{heq:VN}) of $V_\NN$. Let us start with the double integral:
$$F_{2, \NN}= \NN^{2\HH-1}I_2\bigg(\sum_{\ii=0}^{\NN-1} f_{\ii,\NN} \otimes_{q-1} f_{\ii,\NN}\bigg).$$
Since the kernel $\sum_{\ii=0}^{\NN-1} f_{\ii,\NN} \otimes_{q-1} f_{\ii,\NN}$ is symmetric, one has
\begin{align*}
E[F_{2, \NN}^2]& = 2!N^{4\HH-2} \bigg\| \sum_{\ii=0}^{\NN-1} f_{\ii,\NN} \otimes_{q-1} f_{\ii,\NN}\bigg\|_{L^2(([0, 1]^{d})^{ 2})}^2\\ 
&=2!\NN^{4\HH-2} \sum_{\ii, \kk=0}^{\NN-1}\left\langle  f_{\ii,\NN} \otimes_{q-1} f_{\ii,\NN},  f_{\kk,\NN} \otimes_{q-1} f_{\kk,\NN} \right\rangle_{L^2(([0, 1]^{d})^{ 2})}.
\end{align*}
Let us now compute the scalar products in the above expression. By using (\ref{heq:3.1.1}), (\ref{heq:3.1}), by applying Fubini's theorem and by noting that $\int_0^{u \wedge v} \partial_1 K^{H'}(u, a) \partial_1 K^{H'}(v, a)da = H'(2H' -1)|u-v|^{2H'-2}$, it is easy to verify that
\eject
\begin{align*}
&\left\langle  f_{\ii,\NN} \otimes_{q-1} f_{\ii,\NN},  f_{\kk,\NN} \otimes_{q-1} f_{\kk,\NN} \right\rangle_{L^2(([0, 1]^{d})^{ 2})}\\ 
&=b_{q, \HH}^4 \prod_{j=1}^d \left\langle f_{i_j,N_j} \otimes_{q-1} f_{i_j,N_j}, f_{k_j,N_j} \otimes_{q-1} f_{k_j,N_j} \right\rangle_{L^2([0, 1]^2)}\\
& = b^4_{q, \HH} (\HH'(2\HH'-1))^{2q} \prod_{j=1}^d \int_{\frac{i_j}{N_j}}^{\frac{i_j+1}{N_j}}du_j\int_{\frac{i_j}{N_j}}^{\frac{i_j+1}{N_j}}dv_j \int_{\frac{k_j}{N_j}}^{\frac{k_j+1}{N_j}}du'_j \int_{\frac{k_j}{N_j}}^{\frac{k_j+1}{N_j}}dv'_j \\
& \hspace{4.5cm}\times |u_j-v_j|^{(2H'_j-2)(q-1)}|u'_j-v'_j|^{(2H'_j-2)(q-1)}\\
&\hspace{4.5cm}\times|u_j-u'_j|^{2H'_j-2}|v_j-v'_j|^{2H'_j-2},
\end{align*}
(see, e.g., \cite[page 11]{hTudorViens1}). The change of variables $u' = (u-\frac{i}{N})N$ for each $u_j, u'_j, v_j, v'_j$ with $j$ from $1$ to $d$ yields
\begin{align}\label{hf1}
E[F_{2, \NN}^2] &= 2 b^4_{q, \HH} (\HH'(2\HH'-1))^{2q}\NN^{4\HH-2}\NN^{-4}\NN^{-(2\HH'-2)2q} \nonumber\\ 
&\times  \sum_{\ii, \kk=0}^{\NN-1}\prod_{j=1}^d \int_0^1du_j\int_0^1dv_j \int_0^1du'_j \int_0^1dv'_j  
|u_j-v_j|^{(2H'_j-2)(q-1)}|u'_j-v'_j|^{(2H'_j-2)(q-1)} \nonumber\\
&\hspace{5cm}\times|u_j-u'_j + i_j -k_j|^{2H'_j-2}|v_j-v'_j + i_j -k_j|^{2H'_j-2}.
\end{align}
Now, we split the sum $\sum_{\ii, \kk=0}^{\NN-1}$ appearing in $E[F_{2, \NN}^2]$ just above into
\begin{equation}\label{hsum}
\sum_{\ii, \kk=0}^{\NN-1} = \sum_{\substack{\ii, \kk= 0 \\ \exists  1\leq j \leq d: i_j = k_j}}^{\NN-1} + \sum_{\substack{\ii, \kk= 0 \\ \forall  j: i_j \ne k_j}}^{\NN-1} .
\end{equation}
For the first term in the right-hand side of (\ref{hsum}), without loss of generality, let us assume that $ i_1 = k_1, \ldots, i_m = k_m$ for some $1 \leq m < d$ and $i_j \ne k_j$ for all $ m+1 \leq j \leq d$. Then,
\eject
\begin{align*}
&\NN^{-2}\sum_{\substack{\ii, \kk= 0 \\ i_1 = k_1, \ldots, i_m = k_m}}^{\NN-1}\prod_{j=1}^d \int_{[0, 1]^4}du_jdv_jdu'_jdv'_j |u_j-v_j|^{(2H'_j-2)(q-1)}|u'_j-v'_j|^{(2H'_j-2)(q-1)}\\
&\hspace{5.5cm}\times|u_j-u'_j + i_j -k_j|^{2H'_j-2}|v_j-v'_j + i_j -k_j|^{2H'_j-2}\\
&= \prod_{j=1}^m N_j^{-1}  \int_{[0,1]^4}du_jdv_jdu'_jdv'_j (|u_j-v_j||u'_j-v'_j|)^{(2H'_j-2)(q-1)}(|u_j-u'_j ||v_j-v'_j |)^{2H'_j-2}\\
& \qquad\times \sum_{\substack{i_{m+1}, k_{m+1} =0\\ i_{m+1} \ne k_{m+1}}}^{N_j-1} \ldots \sum_{\substack{i_d, k_d =0\\ i_d \ne k_d}}^{N_j-1} \prod_{j=m+1}^d \int_{[0,1]^4}du_jdv_jdu'_jdv'_j (|u_j-v_j||u'_j-v'_j|)^{(2H'_j-2)(q-1)}\\
& \hspace{5cm}\times N_j^{-2}|u_j-u'_j + i_j -k_j|^{2H'_j-2}|v_j-v'_j + i_j -k_j|^{2H'_j-2}.
\end{align*}
By switching sum and product in the above expression, we arrive 
{\allowdisplaybreaks
\begin{align*}
& \prod_{j=1}^m N_j^{-1}  \int_{[0,1]^4}du_jdv_jdu'_jdv'_j (|u_j-v_j||u'_j-v'_j|)^{(2H'_j-2)(q-1)}(|u_j-u'_j ||v_j-v'_j |)^{2H'_j-2}\\
& \qquad\times \prod_{j= m+1}^d \bigg(\sum_{\substack{i_j, k_j = 0 \\i_j \ne k_j}}^{N_j-1} \int_{[0,1]^4}du_jdv_jdu'_jdv'_j |u_j-v_j|^{(2H'_j-2)(q-1)}|u'_j-v'_j|^{(2H'_j-2)(q-1)}\\
& \qquad\qquad\qquad\times N_j^{-2}|u_j-u'_j + i_j -k_j|^{2H'_j-2}|v_j-v'_j + i_j -k_j|^{2H'_j-2}\bigg)\\
&=  \prod_{j=1}^m N_j^{-1}  \int_{[0,1]^4}du_jdv_jdu'_jdv'_j (|u_j-v_j||u'_j-v'_j|)^{(2H'_j-2)(q-1)}(|u_j-u'_j ||v_j-v'_j |)^{2H'_j-2}\\
& \qquad\times \prod_{j=m+1}^d \bigg( \int_{[0,1]^4}du_jdv_jdu'_jdv'_j |u_j-v_j|^{(2H'_j-2)(q-1)}|u'_j-v'_j|^{(2H'_j-2)(q-1)}\\
& \qquad\qquad\qquad\times 2N_j^{-2}\sum_{\substack{i_j, k_j= 0 \\ i_j > k_j}}^{N_j-1}|u_j-u'_j + i_j -k_j|^{2H'_j-2}|v_j-v'_j + i_j -k_j|^{2H'_j-2}\bigg).
\end{align*}
}One has that
\begin{align*}
&N^{-2}\sum_{\substack{i, k= 0 \\ i > k}}^{N-1}|u-u' + i -k|^{2H'-2}|v-v' + i -k|^{2H'-2}\\ 
&= N^{2(2H'-2)}\frac{1}{N}\sum_{n=1}^{N}\Big(1-\frac{n}{N}\Big)\Big|\frac{u-u'}{N}+ \frac{n}{N}\Big|^{2H'-2}\Big|\frac{v-v'}{N}+\frac{n}{N}\Big|^{2H'-2}
\end{align*}
 is asymptotically equivalent to 
$N^{2(2H'-2)}\int_0^1(1-x)x^{4H'-4}dx = N^{2(2H'-2)}\frac{1}{(4H' - 3)(4H'-2)}.$
It follows that
\begin{align*}
&\NN^{-2}\sum_{\substack{\ii, \kk= 0 \\ i_1 = k_1, \ldots, i_m = k_m}}^{\NN-1}\prod_{j=1}^d \int_{[0, 1]^4}du_jdv_jdu'_jdv'_j |u_j-v_j|^{(2H'_j-2)(q-1)}|u'_j-v'_j|^{(2H'_j-2)(q-1)}\\
&\hspace{4.5cm}\times|u_j-u'_j + i_j -k_j|^{2H'_j-2}|v_j-v'_j + i_j -k_j|^{2H'_j-2}\\
 &\approx \prod_{j=1}^m N_j^{-1}  \int_{[0,1]^4}du_jdv_jdu'_jdv'_j |u_j-v_j|^{(2H'_j-2)(q-1)}|u'_j-v'_j|^{(2H'_j-2)(q-1)}\\
&\hspace{5.5cm}\times|u_j-u'_j |^{2H'_j-2}|v_j-v'_j |^{2H'_j-2}\\
& \times \prod_{j=m+1}^d 2N_j^{2(2H'_j-2)}\frac{1}{(4H'_j - 3)(4H'_j-2)} \bigg( \int_{[0,1]^2}du_jdv_j |u_j-v_j|^{(2H'_j-2)(q-1)}\bigg)^2.
\end{align*}
Since $2(2-2H'_j) - 1 < 0$ for all $j$, one gets, as $\NN \to \infty$,
\begin{align}\label{hf2}
&\NN^{2(2-2\HH'_j)} \times \NN^{-2}\sum_{\substack{\ii, \kk= 0 \\ \exists  1\leq j \leq d: i_j = k_j}}^{\NN-1}\prod_{j=1}^d \int_{[0, 1]^4}du_jdv_jdu'_jdv'_j |u_j-v_j|^{(2H'_j-2)(q-1)}|u'_j-v'_j|^{(2H'_j-2)(q-1)}\nonumber\\
&\hspace{4.5cm}\times |u_j-u'_j + i_j -k_j|^{2H'_j-2}|v_j-v'_j + i_j -k_j|^{2H'_j-2} \longrightarrow  0.
 \end{align}
Similarly for the second term in (\ref{hsum}), that is, when $i_j \ne k_j$ for all $1 \leq j \leq d$, we have
\begin{align*}
&\NN^{-2}\sum_{\substack{\ii, \kk= 0 \\ i_j \ne k_j, \text{ } \forall j}}^{\NN-1}\prod_{j=1}^d \int_{[0, 1]^4}du_jdv_jdu'_jdv'_j |u_j-v_j|^{(2H'_j-2)(q-1)}|u'_j-v'_j|^{(2H'_j-2)(q-1)}\\
&\hspace{4.5cm}\times|u_j-u'_j + i_j -k_j|^{2H'_j-2}|v_j-v'_j + i_j -k_j|^{2H'_j-2}\\
 &\approx \prod_{j=1}^d  N_j^{2(2H'_j-2)}\frac{2}{(4H'_j - 3)(4H'_j-2)}\bigg( \int_{[0,1]^2}du_jdv_j |u_j-v_j|^{(2H'_j-2)(q-1)}\bigg)^2\\
 &= \prod_{j=1}^d N_j^{2(2H'_j-2)}\frac{2}{(4H'_j - 3)(4H'_j-2)[(2H'_j-2)(q-1)+1]^2[(H'_j-1)(q-1) + 1]^2}.
\end{align*}
It follows that
\begin{align}\label{hf3}
&\NN^{2(2-2\HH')} \times \NN^{-2}\sum_{\substack{\ii, \kk= 0 \\ i_j \ne k_j, \text{ } \forall j}}^{\NN-1}\prod_{j=1}^d \int_{[0, 1]^4}du_jdv_jdu'_jdv'_j |u_j-v_j|^{(2H'_j-2)(q-1)}|u'_j-v'_j|^{(2H'_j-2)(q-1)}\nonumber\\
&\hspace{5.5cm}\times|u_j-u'_j + i_j -k_j|^{2H'_j-2}|v_j-v'_j + i_j -k_j|^{2H'_j-2}\nonumber\\
 &\longrightarrow \prod_{j=1}^d  \frac{2}{(4H'_j - 3)(4H'_j-2)[(2H'_j-2)(q-1)+1]^2[(H'_j-1)(q-1) + 1]^2}.
\end{align}

To conclude that 
\begin{equation}\label{heq:3.6}
\lim_{\NN \to \infty} E[c_{1, \HH}^{-1}\NN^{2(2-2\HH')}F_{2, \NN}^2] = 1,
\end{equation}
we use the expression  (\ref{hf1}) for $E[F_{2,\NN}^2]$.
The first sum in (\ref{hsum}) goes to zero according to (\ref{hf2}), whereas the second sum goes to
the quantity in (\ref{hf3}).
Going back to the definition (\ref{heq:3.5}) of $c_{1,{\bf H}}$, we arrive to the desired conclusion (\ref{heq:3.6}).

Let us now consider the remaining terms $F_{4, \NN}, \ldots, F_{2q, \NN}$ in the chaos decomposition (\ref{heq:VN}). 
Using that $\|\widetilde{g}\|_{L^2} \leq \|g\|_{L^2}$ for any square integrable function $g$, one can write, for every $ 0 \leq r \leq q-2$,
\begin{align*}
E[F_{2q-2r, \NN}^2]& = \NN^{4\HH-2}(2q-2r)! \,\bigg\| \sum_{\ii=0}^{\NN-1} f_{\ii,\NN} \widetilde{\otimes}_r f_{\ii,\NN}\bigg\|^2_{L^2(([0, 1]^{d})^{2q-2r})}\\ 
 & \leq \NN^{4\HH-2}(2q-2r)!\,\bigg\| \sum_{\ii=0}^{\NN-1}f_{\ii,\NN} \otimes_r f_{\ii,\NN}\bigg\|^2_{L^2(([0, 1]^{d})^{2q-2r})}\\ 
&= (2q-2r)!\NN^{4\HH-2} \sum_{\ii, \kk=0}^{\NN-1}\left\langle  f_{\ii,\NN} \otimes_r f_{\ii,\NN},  f_{\kk,\NN} \otimes_r f_{\kk,\NN} \right\rangle_{L^2(([0, 1]^{d})^{2q-2r})}.
\end{align*}
Proceeding as above, we obtain 
\begin{align*}
&\left\langle  f_{\ii,\NN} \otimes_r f_{\ii,\NN},  f_{\kk,\NN} \otimes_r f_{\kk,\NN} \right\rangle_{L^2([0, 1]^{d \cdot (2q-2r)})}\\ 
& = b^4_{q, \HH} (\HH'(2\HH'-1))^{2q} \prod_{j=1}^d \int_{\frac{i_j}{N_j}}^{\frac{i_j+1}{N_j}}du_j\int_{\frac{i_j}{N_j}}^{\frac{i_j+1}{N_j}}dv_j \int_{\frac{k_j}{N_j}}^{\frac{k_j+1}{N_j}}du'_j \int_{\frac{k_j}{N_j}}^{\frac{k_j+1}{N_j}}dv'_j \\
& \hspace{4.5cm}\times |u_j-v_j|^{(2H'_j-2)r}|u'_j-v'_j|^{(2H'_j-2)r}\\
&\hspace{4.5cm}\times|u_j-u'_j|^{(2H'_j-2)(q-r)}|v_j-v'_j|^{(2H'_j-2)(q-r)}.
\end{align*}
Using the change of variables $u' = (u-\frac{i}{N})N$ for each $u_j, u_j, v_j, v'_j$ with $j=1, \ldots, d$, one obtains
\begin{align*}
E[F_{2q-2r, \NN}^2] &\leq (2q-2r)! b^4_{q, \HH} (\HH'(2\HH'-1))^{2q}\NN^{4\HH-2}\NN^{-4}\NN^{-(2\HH'-2)2q}\\ 
&\times  \sum_{\ii, \kk=0}^{\NN-1}\bigg(\prod_{j=1}^d \int_0^1du_j\int_0^1dv_j \int_0^1du'_j \int_0^1dv'_j |u_j-v_j|^{(2H'_j-2)r}|u'_j-v'_j|^{(2H'_j-2)r}\\
&\hspace{2cm}\times|u_j-u'_j + i_j -k_j|^{(2H'_j-2)(q-r)}|v_j-v'_j + i_j -k_j|^{(2H'_j-2)(q-r)}\bigg).
\end{align*}
Switching sum and product in the above expression, one obtains
\begin{align}\label{hKhochiu}
E[F_{2q-2r, \NN}^2] &\leq (2q-2r)! b^4_{q, \HH} (\HH'(2\HH' -1))^{2q}\NN^{-2} \nonumber\\
&\qquad\times\prod_{j=1}^d\int_{[0, 1]^4}du_jdv_jdu'_jdv'_j |u_j-v_j|^{(2H'_j-2)r}|u'_j-v'_j|^{(2H'_j-2)r}\nonumber\\ 
&\qquad\times \bigg( \sum_{i_j, k_j=0}^{N_j-1}|u_j-u'_j + i_j -k_j|^{(2H'_j-2)(q-r)}|v_j-v'_j + i_j -k_j|^{(2H'_j-2)(q-r)}\bigg).
\end{align}
Note that the above sum $\sum_{i_j, k_j=0}^{N_j-1}$ can be divided into two parts: the diagonal part with $i_j=k_j$ and the non-diagonal part with $i_j\ne k_j$. It is easily seen that the non-diagonal part is dominant. Indeed, the diagonal part in the right-hand side of (\ref{hKhochiu}) is equal to
\begin{align*}
&(2q-2r)! b^4_{q, \HH} (\HH'(2\HH' -1))^{2q}\NN^{-1}\prod_{j=1}^d\int_{[0, 1]^4}du_jdv_jdu'_jdv'_j\\
&\qquad\times |u_j-v_j|^{(2H'_j-2)r}|u'_j-v'_j|^{(2H'_j-2)r}|u_j-u'_j |^{(2H'_j-2)(q-r)}|v_j-v'_j|^{(2H'_j-2)(q-r)}.
\end{align*}
and it tends to zero since $(2H'_j-2)r >- 1$ and $(2H'_j-2)(q-r) > -1$. Thus, in order to find a bound of $E[F_{2q-2r, \NN}^2]$ in (\ref{hKhochiu}), we have to study the following sum
\begin{equation}\label{hKhochiu1}
\frac{1}{N^2}\sum_{\substack{i, k=0 \\ i \ne k}}^{N-1}|u-u' + i -k|^{(2H'-2)(q-r)}|v-v' + i -k|^{(2H'-2)(q-r)}
\end{equation}
for all $q \geq 2$ and $r =0, \ldots, q-2$, when $u, u', v, v' \in [0, 1]$. In (\ref{hKhochiu1}), one has set $H' = 1 + \frac{H-1}{q}$ with $H > \frac12$. We now analyse the behavior of (\ref{hKhochiu1}) according to the following three cases: $H > \frac34, H <\frac34$ and $H =\frac34$.

$\bullet$ If $H > \frac34$, then (\ref{hKhochiu1}) is equal to
$$ N^{(2H'-2)(2q-2r)}\frac{2}{N}\sum_{n=1}^{N}\Big(1-\frac{n}{N}\Big)\Big|\frac{u-u'}{N}+ \frac{n}{N}\Big|^{(2H'-2)(q-r)}\Big|\frac{v-v'}{N}+\frac{n}{N}\Big|^{(2H'-2)(q-r)}.$$
By multiplying (\ref{hKhochiu1}) by $N^{(2-2H')(2q-2r)}$ one has
\begin{align*}
&N^{(2-2H')(2q-2r)}\times \frac{1}{N^2}\sum_{\substack{i, k=0 \\ i \ne k}}^{N-1}|u-u' + i -k|^{(2H'-2)(q-r)}|v-v' + i -k|^{(2H'-2)(q-r)}\\ 
& = \frac{2}{N}\sum_{n=1}^{N}\Big(1-\frac{n}{N}\Big)\Big|\frac{u-u'}{N}+ \frac{n}{N}\Big|^{(2H'-2)(q-r)}\Big|\frac{v-v'}{N}+\frac{n}{N}\Big|^{(2H'-2)(q-r)}\\
&\approx 2 \int_0^1 (1-x)x^{2(2H' -2)(q-r)}dx < \infty \qquad \text{since } H  > \frac34.
\end{align*}

$\bullet$ If $H <\frac34$, (\ref{hKhochiu1}) is bounded by
$$\frac{1}{N}\sum_{r\in\mathbb{Z}\setminus\{0\}} |u-u' + r|^{(2H'-2)(2q-2r)}|v-v' + r|^{(2H'-2)(2q-2r)} = O(\frac{1}{N}).$$

$\bullet$ If $H =\frac34$, following the same route as in the case $H < \frac34$, we arrive to $(\ref{hKhochiu1}) = O(\frac{\log{\NN}}{\NN})$. 

Now, we go back to (\ref{hKhochiu}). From the analysis of (\ref{hKhochiu1}), we conclude that
$$E[F_{2q-2r, \NN}^2] = \begin{cases}
O(\NN^{-(2\HH' -2)(2q-2r)})& \text{if } H \in (\frac34, 1)\\
O(\NN^{-1})& \text{if } H \in (\frac12, \frac34)\\
O(\frac{\log{\NN}}{\NN})& \text{if } H =\frac34\\
\end{cases}
$$
Therefore, for all $0 \leq r \leq q-2$ and as $\NN \to \infty$, one has
\begin{equation}\label{heq:3.7}
\lim_{\NN \to \infty} E[\NN^{2(2-2\HH')}F_{2q-2r, \NN}^2] = 0.
\end{equation}
Thus, from (\ref{heq:3.6}), (\ref{heq:3.7}) and the orthogonality of Wiener chaos, we obtain (\ref{heq:30}).

\subsection{Concluding the proof of Theorem \ref{hmain}}

Thanks to (\ref{heq:3.7}), in order to understand the asymptotic behavior of  the normalized sequence of $V_{\NN}$, it is enough to analyse the convergence of the term

\begin{equation}\label{heq:3.12}
\NN^{2-2\HH'}F_{2, \NN} = I_2\bigg(\NN^{2\HH-1}\NN^{2-2\HH'}\sum_{\ii=0}^{\NN-1}f_{\ii,\NN} \otimes_{q-1} f_{\ii,\NN} \bigg),
\end{equation}
with
{\allowdisplaybreaks
\begin{align*}
&f_{\ii,\NN} \otimes_{q-1} f_{\ii,\NN} (\xx_1, \xx_2)= b_{q, \HH}^2 \prod_{j=1}^d(f_{i_j, N_j} \otimes_{q-1} f_{i_j, N_j})(x_{1,j}, x_{2,j})\\
& = b_{q, \HH}^2 (\HH'(2\HH'-1))^{q-1}\\
&\qquad \times \prod_{j=1}^d\bigg(\mathbbm{1}_{[0, \frac{i_j}{N_j}]}(x_{1,j}) \mathbbm{1}_{[0, \frac{i_j}{N_j}]}(x_{2,j}) \int_{\frac{i_j}{N_j}}^{\frac{i_j+1}{N_j}}du\int_{\frac{i_j}{N_j}}^{\frac{i_j+1}{N_j}}du'  \partial_1K^{H'_j}(u, x_{1,j}) \\
&\hspace{7cm} \times \partial_1K^{H'_j}(u', x_{2,j})|u-u'|^{(2H'_j-2)(q-1)}\\
&\hspace{1.2cm}+ \mathbbm{1}_{[0, \frac{i_j}{N_j}]}(x_{1,j}) \mathbbm{1}_{[\frac{i_j}{N_j}, \frac{i_j+1}{N_j}]}(x_{2,j}) \int_{\frac{i_j}{N_j}}^{\frac{i_j+1}{N_j}}du\int_{x_{2,j}}^{\frac{i_j+1}{N_j}}du'  \partial_1K^{H'_j}(u, x_{1,j})\\
&\hspace{7cm}\times \partial_1K^{H'_j}(u', x_{2,j})|u-u'|^{(2H'_j-2)(q-1)}\\
&\hspace{1.2cm}+\mathbbm{1}_{[\frac{i_j}{N_j}, \frac{i_j+1}{N_j}]}(x_{1,j}) \mathbbm{1}_{[0, \frac{i_j+1}{N_j}]}(x_{2,j}) \int_{x_{1,j}}^{\frac{i_j+1}{N_j}}du\int_{\frac{i_j}{N_j}}^{\frac{i_j+1}{N_j}}du'  \partial_1K^{H'_j}(u, x_{1,j})\\
&\hspace{7cm} \times\partial_1K^{H'_j}(u', x_{2,j})|u-u'|^{(2H'_j-2)(q-1)}\\
&\hspace{1.2cm}+\mathbbm{1}_{[\frac{i_j}{N_j}, \frac{i_j}{N_j}]}(x_{1,j}) \mathbbm{1}_{[\frac{i_j}{N_j}, \frac{i_j+1}{N_j}]}(x_{2,j}) \int_{x_{1,j}}^{\frac{i_j+1}{N_j}}du\int_{x_{2,j}}^{\frac{i_j+1}{N_j}}du'  \partial_1K^{H'_j}(u, x_{1,j}) \\
&\hspace{7cm} \times\partial_1K^{H'_j}(u', x_{2,j})|u-u'|^{(2H'_j-2)(q-1)}\bigg).
\end{align*}
}Among the four terms in the right-hand side of the above expression, only the first one is not asymptotically negligible in $L^2(\Omega)$ as $\NN \to \infty$, see \cite[page 14 and 15]{hTudorViens1} or follow the lines of \cite{hTudorViens2} for details. Furthermore, by the isometry property for multiple Wiener-It\^{o} integrals, evaluating the $L^2(\Omega)$-limit  of a sequence belonging to the second Wiener chaos is equivalent to evaluating the $L^2(([0, 1]^{d})^{2})$-limit of the sequence of their corresponding symmetric kernels. Therefore, we are left to find the limit of $f_2^{\NN}$ in $L^2(([0, 1]^{d})^2)$, where
\eject
\begin{align*}
f_2^{\NN}(\xx_1, \xx_2):&= \NN^{2\HH-1} \NN^{2-2\HH'}b_{q, \HH}^2 (\HH'(2\HH'-1))^{q-1}\\
& \times \sum_{\ii =0}^{\NN -1}\bigg(  \prod_{j=1}^d \mathbbm{1}_{[0, \frac{i_j}{N_j}]}(x_{1,j}) \mathbbm{1}_{[0, \frac{i_j}{N_j}]}(x_{2,j}) \int_{\frac{i_j}{N_j}}^{\frac{i_j+1}{N_j}}du\int_{\frac{i_j}{N_j}}^{\frac{i_j+1}{N_j}}du'  \partial_1K^{H'_j}(u, x_{1,j}) \\
&\hspace{6cm} \times \partial_1K^{H'_j}(u', x_{2,j})|u-u'|^{(2H'_j-2)(q-1)}\bigg)\\
&= \NN^{2\HH-1} \NN^{2-2\HH'}b_{q, \HH}^2 (\HH'(2\HH'-1))^{q-1}\\
& \times \prod_{j=1}^d\bigg(   \sum_{i_j =0}^{N_j-1} \mathbbm{1}_{[0, \frac{i_j}{N_j}]}(x_{1,j}) \mathbbm{1}_{[0, \frac{i_j}{N_j}]}(x_{2,j}) \int_{\frac{i_j}{N_j}}^{\frac{i_j+1}{N_j}}du\int_{\frac{i_j}{N_j}}^{\frac{i_j+1}{N_j}}du'  \partial_1K^{H'_j}(u, x_{1,j}) \\
&\hspace{6cm} \times \partial_1K^{H'_j}(u', x_{2,j})|u-u'|^{(2H'_j-2)(q-1)}\bigg).
\end{align*}
According to \cite[Theorem 3.2]{hTudorViens1}, it is shown that for each $j$ from $1$ to $d$, the following quantity
\begin{align*}
&N_j^{2H_j-1} N_j^{2-2H'_j} \sum_{i_j=1}^{N_j-1} \mathbbm{1}_{[0, \frac{i_j}{N_j}]}(x_{1,j}) \mathbbm{1}_{[0, \frac{i_j}{N_j}]}(x_{2,j}) \int_{\frac{i_j}{N_j}}^{\frac{i_j+1}{N_j}}du\int_{\frac{i_j}{N_j}}^{\frac{i_j+1}{N_j}}du'  \partial_1K^{H'_j}(u, x_{1,j}) \\
&\hspace{6cm} \times \partial_1K^{H'_j}(u', x_{2,j})|u-u'|^{(2H'_j-2)(q-1)}
\end{align*}
converges in $L^2(\mathbb{R}^2)$  to the kernel of a standard Rosenblatt process with self-similarity $2H'_j-1$ at time $1$ (up to an explicit multiplicative constant). Since the kernel of the Rosenblatt sheet has the form of a tensor product from $1$ to $d$ of the kernel of the Rosenblatt process, (see (\ref{hDef2})), 
it follows that $f_2^{\NN}$ converges to the kernel of a Rosenblatt sheet with self-similarity parameter $2\HH'-1$ evaluated at time $\textbf{1}$ up to a constant.
Therefore, the double Wiener-It\^{o} integral $\NN^{2-2\HH'}F_{2, \NN}$ in (\ref{heq:3.12}) converges in $L^2(\Omega)$ to a Rosenblatt sheet $R^{2\HH'-1}_{\textbf{1}}$ with self-similarity parameter $2\HH'-1$ evaluated at time $\textbf{1}$, which leads to the convergence of $\NN^{2-2\HH'}c_2^{-1}V_\NN$ to the same limit (up to a constant). In order to find the explicit constant, we use the fact that $\lim_{\NN \to \infty}E[(c_{1, \HH}^{-\frac12} \NN^{2-2\HH'}c_2^{-1}V_\NN)^2] = E[(R^{2\HH'-1}_{\textbf{1}})^2] =1$ to eventually obtain that $c_{1, \HH}^{-\frac12} \NN^{2-2\HH'}c_2^{-1}V_\NN$ converges in $L^2(\Omega)$ to the Rosenblatt sheet $R^{2\HH'-1}_{\textbf{1}}$ as $\NN \to \infty$ with $c_2 = q!q$.

\section*{Acknowledgements} 
I thank the authors of \cite{hReveillac3} for sharing their paper in progress with me. Another thank goes to my advisor, Prof. Ivan Nourdin, for his very careful review of the paper as well as for his comments and corrections. Finally, I deeply thank an anonymous referee for a very careful and thorough reading of this work, and for her/his
constructive remarks.

 \chapter[Statistical inference for Vasicek-type model driven by Hermite processes]{Statistical inference for Vasicek-type model driven by Hermite processes}

\begin{center}
Ivan Nourdin, T. T. Diu Tran\\
Universit\'{e} du Luxembourg
\end{center}

\begin{tcolorbox}
This article is submitted for publication to \textit{Stochastic Process and Applications}. arXiv: 1712.05915.
\end{tcolorbox}

\begin{center}
\textbf{Abstract}
\end{center}

Let $(Z^{q, H}_t)_{t \geq 0}$ denote a Hermite process of order $q \geq 1$ and self-similarity parameter $H \in (\frac{1}{2}, 1)$. This process is $H$-self-similar, has stationary increments and exhibits long-range dependence. When $q=1$, it corresponds to the fractional Brownian motion,  whereas it is not Gaussian as soon as $q\geq 2$. 

In this paper, we deal with the following Vasicek-type model driven by $Z^{q, H}$:
\[
X_0=0,\quad dX_t = a(b - X_t)dt +dZ_t^{q, H}, \qquad t \geq 0,
\]
where $a > 0$ and $b \in \mathbb{R}$ are considered as unknown drift parameters. We provide estimators for $a$ and $b$ based on continuous-time observations. For all possible values of $H$ and $q$, we prove strong consistency and we 
analyze the asymptotic fluctuations. 

\section{Introduction}

Our aim in this paper is to introduce and analyze a {\it non-Gaussian} extension of the fractional model considered in the seminal paper \cite{bCR} of Comte and Renault  (see also Chronopoulou and Viens \cite{bCV}, as well as the motivations and references therein) and used by these authors to model a situation where, unlike the classical Black-Scholes-Merton model, the volatility exhibits long-memory. More precisely, we deal with the parameter estimation problem for a Vasicek-type process $X$,  defined as the unique (pathwise) solution to
\begin{equation}\label{bfracV}
X_0=0,\quad dX_t = a(b-X_t)dt + dZ^{q, H}_t, \,\,t\geq 0,
\end{equation}
where $Z^{q,H}$ is a Hermite process of order $q\geq 1$ and Hurst parameter $H\in (\frac12,1)$.
Equivalently, $X$ is the process given explicitly by 
\begin{equation}\label{bfracV2}
X_t = b(1-e^{-at}) +  \int_0^t e^{-a(t-s)}dZ^{q, H}_s,
\end{equation}
where the integral with respect to $Z^{q, H}$ must be understood in the Riemann-Stieltjes sense.
In (\ref{bfracV}) and (\ref{bfracV2}), parameters $a>0$ and $b\in\R$ are considered as  (unknown) real parameters. 

Hermite processes $Z^{q, H}$ of order $q\geq 2$ form a class of genuine non-Gaussian generalizations of the celebrated fractional Brownian motion (fBm), this latter corresponding to the case $q=1$. Like the fBm, they are self-similar, have stationary increments and  exhibit long-range dependence. Their main noticeable difference with respect to fBm is that they are {\it not} Gaussian.
For more details about this family of processes, we refer the reader to Section \ref{bsec:hermite}.

As we said, the goal of the present paper is to propose suitable estimators for $a$ and $b$ in (\ref{bfracV})-(\ref{bfracV2}), and  to study their asymptotic properties (that is, their consistency as well as their fluctuations around the true value of the parameter) when a continuous record of observation for $X$ is available. 
Our main motivation behind this study is to understand whether the {\it Gaussian feature} of the fractional Brownian motion $B^H$ really matters  when estimating the unknown parameters in the fractional Vasicek model considered in Comte-Renault \cite{bCR} and given by
\begin{equation}\label{bFBM-V}
X_0=0,\quad dX_t = a(b-X_t)dt + dB^H_t, \,\,t\geq 0.
\end{equation}
More precisely, we look for an answer to the following question.
Do our estimators for $a$ and $b$ have the same asymptotic behavior when $q=1$ (fBm case, model (\ref{bFBM-V})) and $q\geq 2$ (non-Gaussian case, model (\ref{bfracV}))?
If the answer to this question appears to be no, it means that assuming the noise is Gaussian (like done in Comte-Renault \cite{bCR}) is not an insignificant hypothesis. On the contrary, if the results obtained for $q=1$ 
and $q\geq 2$ happen to be of the same nature, one could conclude that the Hermite Vasicek model (\ref{bfracV}) displays some kind of universality with respect to the order $q$, and then working under the Gaussian assumption for the noise is actually not a loss of generality, as far as statistical inference for parameters $a$ and $b$ is concerned.

Let us now describe in more details the results we have obtained.
\begin{Definition}
Recall from (\ref{bfracV})-(\ref{bfracV2}) the definition of the Vasicek-type process $X=(X_t)_{t\geq 0}$ driven by the Hermite process $Z^{q, H}$.  
Assume that $q\geq 1$ and $H\in(\frac12,1)$ are known, whereas $a>0$ and $b\in\R$
are unknown.
Suppose that we continuously observe $X$ over the time interval $[0,T]$, $T>0$.
Then, we define estimators for $a$ and $b$ as follows:
\begin{eqnarray}
\widehat{a}_T&=&\left(\frac{\alpha_T}{H\Gamma(2H)}\right)^{-\frac1{2H}}, \quad\mbox{where $\alpha_T=\frac{1}{T}\int_0^T X_t^2dt - \left(\frac{1}{T}\int_0^T X_tdt\right)^2$},\label{balpha_T}\\
\widehat{b}_{T}&=&\frac{1}{T}\int_0^T X_tdt.\notag
\end{eqnarray}
\end{Definition}

In order to describe the asymptotic behavior of $(\widehat{a}_T,\widehat{b}_T)$ when $T\to\infty$, we first need to define a random variable, called $G_\infty$, which is $Z^{q,H}$-measurable. It is the object of the following proposition.
\begin{Proposition}\label{bfinfini}
Assume either ($q=1$ and $H>\frac34$) or $q\geq 2$. Fix $T>0$, and let
$U_T=(U_T(t))_{t\geq 0}$ be the process defined as $U_T(t)=\int_0^t e^{-T(t-u)}dZ^{q, H}_u$.
Finally, define the random variable $G_T$ by
\[
G_T = T^{\frac{2}{q}(1-H)+2H}\int_0^1 \big(U_T(t)^2 - \E[U_T(t)^2])dt.
\]
Then $G_T$ converges in $L^2(\Omega)$ to a limit written $G_\infty$.
Moreover, $G_\infty/B_{H,q}$ is distributed according to the Rosenblatt distribution of parameter $1-\frac{2}{q}(1-H)$, where
\begin{equation}\label{bbeta}
B_{H,q}=\frac{H(2H-1)}{\sqrt{(H_0-\frac{1}{2})(4H_0-3)}}\,
\times \frac{\Gamma(2H+\frac2q(1-H))}{2H+\frac2q(1-H)-1},\quad\mbox{with }
H_0=1-\frac{1-H}{q}.
\end{equation}
(The definition of the Rosenblatt distribution is recalled in Definition \ref{brosenblatt}.)
\end{Proposition}

We can now describe the asymptotic behavior of $(\widehat{a}_T,\widehat{b}_T)$ as
$T\to\infty$.
\begin{Theorem}\label{bmain}
Let $X=(X_t)_{t\geq 0}$ be given by (\ref{bfracV})-(\ref{bfracV2}), where $Z^{q, H}=(Z^{q, H}_t)_{t\geq 0}$ is a Hermite process of order $q\geq 1$ and parameter $H\in(\frac12,1)$, and where $a>0$ and $b\in\R$
are (unknown) real parameters.
The following convergences take place as $T\to\infty$.
\begin{enumerate}
\item{} {\rm [Consistency]} 
$
(\widehat{a}_T,\widehat{b}_{T})\overset{\rm a.s.}{\to} (a,b).
$
\item{} {\rm [Fluctuations]}
They depend on the values of $q$ and $H$.
\begin{itemize}
\item {\rm (Case $q=1$ and $H<\frac34$)} 
\begin{eqnarray}
\left(\sqrt{T}\{\widehat{a}_T-a\},
T^{1-H}\{\widehat{b}_{T}-b\}
\right)&\overset{\rm law}{\to}&\left(-\frac{a^{1+4H}\sigma_H}{2H^2\Gamma(2H)}\,N,
\frac{1}{a}N'\right),\label{bG1}
\end{eqnarray}
where $N,N'\sim \mathcal{N}(0,1)$ are independent and $\sigma_H$ is given by 
\begin{equation}\label{bsigma}
\sigma_H=\frac{2H-1}{H\Gamma(2H)^2}\,\sqrt{\int_\R\left(
\int_{\R_+^2}e^{-(u+v)}|u-v-x|^{2H-2}dudv
\right)^2dx}.
\end{equation}
\item {\rm (Case $q=1$ and $H=\frac34$)} 
\begin{equation}\label{34}
\left(\sqrt{\frac{T}{\log T}}\{\widehat{a}_T-a\},T^{\frac14}\big\{\widehat{b}_T - b\} \right)
\to
\left(\frac{3}4\sqrt{\frac{a}\pi}\, N,\frac{1}{a}N'\right),
\end{equation}
where $N,N'\sim \mathcal{N}(0,1)$ are independent.
\item {\rm (Case $q=1$ and $H>\frac34$)} 
\begin{eqnarray}
\left(T^{2(1-H)}\{\widehat{a}_T-a\},T^{1-H}\big\{\widehat{b}_T - b\} \right)\overset{\rm law}{\to}
\left(-\frac{a^{2H-1}}{2H^2\Gamma(2H)}\Big(G_\infty-(B^H_1)^2\Big),\frac{1}{a}B^H_1\right),\notag\\
\label{bR4}
\end{eqnarray}
where $B^H=Z^{1,H}$ is the fractional Brownian motion and $G_\infty$ is defined in Proposition \ref{bfinfini}.

\item {\rm (Case $q\geq 2$ and any $H$)} 
\begin{eqnarray}
\left(T^{\frac2q(1-H)}\{\widehat{a}_T-a\},T^{1-H}\big\{\widehat{b}_T - b\} \right)
\overset{\rm law}{\to}
\left(-\frac{a^{1-\frac2q(1-H)}}{2H^2\Gamma(2H)}\,G_\infty,\frac{1}{a}Z^{q, H}_1\right),
\label{bR}
\end{eqnarray}
where $G_\infty$ is defined in Proposition \ref{bfinfini}.
\end{itemize}
\end{enumerate}
\end{Theorem}

As we see from our Theorem \ref{bmain}, strong consistency for $\widehat{a}_T$ and $\widehat{b}_T$ always holds, {\it irrespective} of the values of $q$ (and $H$). That is, when one is only interested in the first order approximation for $a$ and $b$, Vasicek-type model (\ref{bfracV})-(\ref{bfracV2}) displays a kind of universality with respect to the order $q$ of the underlying Hermite process.
But, as point 2 shows, the situation becomes different when one looks at the fluctuations, that is, when one seeks to construct asymptotic confidence intervals: they heavily depend on $q$ (and $H$).

The rest of the paper is structured as follows. Section 4.2 presents some
basic results about multiple Wiener-It\^o integrals and Hermite processes, as well as some other facts which are used throughout the paper. The proof of Proposition 4.1.2 is then given in Section 4.3. 
Section 4.4 is devoted to the proof of 
the consistency part of Theorem \ref{bmain}, whereas
the fluctuations are analyzed in Section 4.5.

\section{Preliminaries}

\subsection{Multiple Wiener-It\^{o} integrals}

Let $B=\big\{B(h),\,h\in L^2(\R)\big\}$ be a Brownian field defined on a probability space $(\Omega,\mathcal{F},\mathbb{P})$, that is, a centered Gaussian family
satisfying $\E[B(h)B(g)]=\langle h,g\rangle_{L^2(\R)}$ for any $h,g\in L^2(\R)$.

For every $q\geq 1$, the $q$th Wiener chaos $\mathcal{H}_q$ is defined as the closed linear subspace of $L^2(\Omega)$ generated by the family of random variables $\{H_q(B(h)),\,h\in L^2(\R),\,\|h\|_{L^2(\R)}=1\}$, where $H_q$ is the $q$th Hermite polynomial ($H_1(x)=x$, $H_2(x)=x^2-1$, $H_3(x)=x^3-3x$, and so on).

The mapping $I^B_q(h^{\otimes q})=H_q(B(h))$ can be extended to a linear isometry between $L^2_s(\R^q)$ (= the space of symmetric square integrable functions of $\R^q$, equipped with the modified
norm $\sqrt{q!}\|\cdot\|_{L^2(\R^q)}$) and
the $q$th Wiener chaos  $\mathcal{H}_q$. When $f\in L^2_s(\R^q)$, the random variable $I^B_q(f)$ is called the {\it multiple Wiener-It\^o integral of $f$ of order $q$}; equivalently, 
one may write
\begin{equation}\label{bmultipleiqbf}
I_q^B(f) = \int_{\mathbb{R}^q} f(\xi_1, \ldots, \xi_q) dB_{\xi_1}\ldots dB_{\xi_q}.
\end{equation}

Multiple Wiener-It\^o integrals enjoy many nice properties. We refer to \cite{bNourdinPeccatibook} or \cite{bNualartbook} for a comprehensive list of them. Here, we only recall the orthogonality relationship, the isometry formula and the hypercontractivity property. 

First, the {\it orthogonality relationship} (when $p\neq q$) or {\it isometry formula} (when $p=q$) states that, if $f \in L^2_s(\mathbb{R}^p)$ and $g \in L^2_s(\mathbb{R}^q)$ with $p,q\geq 1$, then
\begin{equation}\label{bisom**}
\E[I_p^B(f)I_q^B(g)] =
\begin{cases}
 & p! \big\langle f, g \big\rangle_{L^2(\mathbb{R}^p)} \qquad\text{if } p=q\\
& 0 \qquad\qquad\qquad\quad \text{ if } p \ne q.
\end{cases}
\end{equation}

Second, the {\it hypercontractivity property} reads as follows: for any $q\geq 1$, any $k\in[2,\infty)$ and any $f\in L^2_s(\R^q)$,
 \begin{equation}\label{beq:hypercontractivity1}
\E[|I_q^B(f)|^k]^{1/k} \leq (k-1)^{q/2}\E[|I_q^B(f)|^2]^{1/2}.
 \end{equation}
As a consequence, for any $q\geq 1$ and any $k\in[2,\infty)$, there exists a constant $C_{k,q}>0$ such that, for any $F \in \oplus_{l=1}^q \mathcal{H}_l$, we have
 \begin{equation}\label{beq:hypercontractivity2}
\E[|F|^k]^{1/k} \leq C_{k, q}\,\sqrt{\E[F^2]}.
 \end{equation}

\subsection{Hermite processes}\label{bsec:hermite}

We now give the definition and present some basic properties of Hermite processes. We refer the reader to the recent book \cite{bTudorbook} for any missing proof and/or any unexplained notion.

\begin{Definition}
\textit{The Hermite process} $(Z_t^{q, H})_{t \geq 0}$ of order $q \geq 1$ and self-similarity parameter $H \in (\frac{1}{2}, 1)$ is defined as
\begin{equation}\label{bH1}
Z^{q, H}_t= c(H, q) \int_{\mathbb{R}^q} \bigg( \int_0^t \prod_{j=1}^q(s- \xi_j)_+^{H_0 - \frac{3}{2}}ds\bigg) dB_{\xi_1}\ldots dB_{\xi_q},
\end{equation}
where 
\begin{equation}\label{bH2}
c(H, q) = \sqrt{\frac{H(2H - 1)}{q! \beta^q(H_0 - \frac{1}{2}, 2-2H_0)}} \quad \text{and} \quad H_0 = 1+\frac{H-1}{q} \in \left(1-\frac{1}{2q}, 1\right).
\end{equation}
(The integral (\ref{bH1}) is a multiple Wiener-It\^{o} integral of order $q$ of the form (\ref{bmultipleiqbf}).)
\end{Definition}

The positive constant $c(H, q)$ in (\ref{bH2}) has been chosen to ensure that $\E[(Z_1^{q, H})^2] = 1$. 

\begin{Definition}
A random variable which has the same law as $Z^{q, H}_1$ is called a \textit{Hermite random variable} of order $q$ and parameter $H$.
\end{Definition}

Hermite process of order $q=1$ is nothing but the fractional Brownian motion. 
It is the only Hermite process to be Gaussian (and that one could have defined for $H\leq\frac12$ as well). Hermite process of order $q=2$ is called \textit{the Rosenblatt process}.

\begin{Definition}\label{brosenblatt}
A random variable which has the same law as $Z^{2, H}_1$ is called a \textit{Rosenblatt random variable} of parameter $H$.
\end{Definition}

Except for Gaussianity, Hermite processes of order $q \geq 2$ share many properties with the fractional Brownian motion (corresponding to $q=1$). We list some of them in the next statement.

\begin{Proposition}
The Hermite process $Z^{q,H}$ of order $q\geq 1$ and Hurst parameter $H\in (\frac12,1)$ enjoys the following properties.
\begin{itemize} 
\item{} {\rm [Self-similarity]} For all  $c > 0, (Z^{q, H}_{ct})_{t \geq 0} \overset{law}{=}  (c^H Z^{q, H}_t)_{t \geq 0}$. 
\item {}{\rm [Stationarity of increments]} For any $h >0$, $(Z^{q, H}_{t+h} - Z^{q, H}_h)_{t \geq 0} \overset{law}{=} (Z^{q, H}_t)_{t \geq 0}$.
\item {}{\rm [Covariance function]} For all $s, t \geq 0$, $\E[Z_t^{q, H}Z_s^{q,H}]= \frac{1}{2}(t^{2H} + s^{2H} - |t-s|^{2H})$.
\item  {}{\rm [Long-range dependence]} $\sum_{n=0}^\infty |\E[Z_1^{q, H}(Z_{n+1}^{q, H} - Z_n^{q, H})]| = \infty$.
\item {}{\rm [H\"{o}lder continuity]} For any $\zeta \in (0, H)$ and any compact interval $[0,T]\subset\R_+$, $(Z^{q, H}_t)_{t\in [0,T]}$ admits a version with H\"{o}lder continuous sample paths of order $\zeta$.
\item {}{\rm [Finite moments]} For every $p \geq 1$, 
there exists a constant $C_{p,q}>0$ such that
$\E[|Z^{q, H}_t|^p] \leq C_{p, q} t^{pH}$ for all $t\geq 0$.
\end{itemize}
\end{Proposition}

\subsection{Wiener integral with respect to Hermite processes}

The Wiener integral of a deterministic function $f$ with respect to a Hermite process $Z^{q, H}$, which we denote by $\int_{\mathbb{R}}f(u)dZ^{q, H}_u$, has been constructed by Maejima and Tudor in \cite{bMaejimaTudor}. 

Below is a very short summary of what will is needed in the paper about those integrals.
The stochastic integral $\int_{\mathbb{R}}f(u)dZ^{q, H}_u$  is  well-defined for any $f$ belonging to the space $|\mathcal{H}|$
of functions $f: \mathbb{R} \to \mathbb{R}$ such that 
\[
\int_{\mathbb{R}}\int_{\mathbb{R}} |f(u)f(v)||u-v|^{2H-2}dudv < \infty.
\]
We then have, for any  $f,g\in|\mathcal{H}|$, that
\begin{equation}\label{bisometry}
\E\bigg[\int_{\mathbb{R}}f(u)dZ^{q, H}_u\int_{\mathbb{R}}g(v)dZ^{q, H}_u\bigg] = H(2H-1)\int_{\mathbb{R}}\int_{\mathbb{R}}f(u)g(v)|u - v|^{2H-2}dudv.
\end{equation}
Another important and useful property is that, whenever $f \in |\mathcal{H}|$, the stochastic integral $\int_{\mathbb{R}}f(u)dZ^{q, H}_u$ admits the following representation as a multiple Wiener-It\^{o} integral of the form (\ref{bmultipleiqbf}):
\begin{equation}\label{beq:13}
\int_{\mathbb{R}}f(u)dZ^{q, H}_u = c(H, q) \int_{\mathbb{R}^q}\bigg(\int_{\mathbb{R}}f(u) \prod_{j=1}^q(u- \xi_j)_+^{H_0 - \frac{3}{2}}du\bigg)dB_{\xi_1}\ldots dB_{\xi_q},
\end{equation}
with $c(H, q)$ and $H_0$ given in (\ref{bH2}). 

\subsection{Existing limit theorems}

To the best of our knowledge, only a few limit theorems have been already obtained in the litterature for quadratic functionals of the Hermite process. Here we mainly focus on one of them, because it is the one that we will need in order to study the fluctuations of $(\widehat{a}_T,\widehat{b}_T)$ in Theorem \ref{bmain}. 
To state it, we define 
\begin{equation}\label{beq:X}
Y_t = \int_0^t e^{-a(t-u)}dZ^{q, H}_u, \qquad t \geq 0.
\end{equation}

The following result has been obtained by the second-named author in \cite{bDiu}.
\begin{Proposition}
Let $Y$ be given by (\ref{beq:X}), with either $q\geq 2$ or ($q=1$ and $H>\frac34$). Then, 
as $T\to\infty$,
\begin{equation}\label{bdiu}
T^{\frac{2}{q}(1-H)-1}\int_0^T \big(Y_t^2-\E[Y_t^2]\big)dt
\overset{\rm law}{\to} B_{H,q}\, a^{-2H-\frac2q(1-H)}\times  R^{H'},
\end{equation}
where $R^{H'}$ is distributed according to a Rosenblatt random variable of parameter $H'=1-\frac{2}{q}(1-H)$ and $B_{H,q}$ is given by (\ref{bbeta}).
\end{Proposition}

Along the proof of Theorem \ref{bmain}, we will also make use of another result, which has been shown in \cite{bNNZ}.
\begin{Proposition}
Let $Y$ be given by (\ref{beq:X}), with $q=1$ and $H\in\big(\frac12,\frac34\big)$. Then, 
as $T\to\infty$,
\begin{equation}\label{bivan}
T^{-\frac12}\int_0^T \big(Y_t^2-\E[Y_t^2]\big)dt
\overset{\rm law}{\to} a^{2H}\sigma_H\,N,
\end{equation}
where $\sigma_H$ is given by (\ref{bsigma}) and $N\sim \mathcal{N}(0,1)$.
\end{Proposition}
Relying on the seminal Peccati-Tudor criterion on asymptotic {\it joint} normality (see, e.g., \cite[Theorem 6.2.3]{bNourdinPeccatibook}) and since $T^{-\frac12}\int_0^T \big(Y_t^2-\E[Y_t^2]\big)dt$ (resp. $T^{-H}B^H_T$) belongs to the second (resp. first) Wiener chaos, we have even more than (\ref{bivan}) for free, namely
\begin{equation}\label{bivanbis}
\left(T^{-\frac12}\int_0^T \big(Y_t^2-\E[Y_t^2]\big)dt, T^{-H}B^H_T\right)
\overset{\rm law}{\to} (a^{2H}\sigma_H\,N, N'),
\end{equation}
where $N,N'\sim N(0,1)$ are independent.

Finally, in the critical case $q=1$ and $H=\frac34$, we will rely on the following result,
 established by Hu, Nualart and Zhou in \cite[Theorem 5.4]{bHuNualartZhou}.

\begin{Proposition}
Let $Y$ be given by (\ref{beq:X}), with $q=1$ and $H=\frac34$. Then, 
as $T\to\infty$,
\begin{equation}\label{bnunu}
(T\log T)^{-\frac12}\int_0^T \big(Y_t^2-\E[Y_t^2]\big)dt
\overset{\rm law}{\to} \frac{27}{64a^2}\,N,
\end{equation}
where  $N\sim N(0,1)$.
\end{Proposition} 
Similarly to (\ref{bivanbis}) and for exactly the same reason, we have even more than (\ref{bnunu}) for free, namely: 
\begin{equation}\label{bnunubis}
\left((T\log T)^{-\frac12}\int_0^T \big(Y_t^2-\E[Y_t^2]\big)dt
, T^{-\frac34}B^{\frac34}_T\right)
\overset{\rm law}{\to} (\frac{27}{64a^2}\,N, N'),
\end{equation}
where $N,N'\sim N(0,1)$ are independent.

\subsection{A few other useful facts}
In this section, we let $X$ be given by (\ref{bfracV2}), with $a>0$, $b\in\R$ and $Z^{q, H}$ a Hermite process of order $q\geq 1$ and Hurst parameter $H\in(\frac12,1)$.
We can write
\begin{equation}\label{bdecompoX}
X_t=h(t)+Y_t,
\quad\mbox{
 where
$h(t)=b(1-e^{-at})$ and $Y$ is given by (\ref{beq:X}).}
\end{equation}
The following limit, obtained as a consequence of the isometry property (\ref{bisometry}),  will be used many times throughout the sequel:
\begin{eqnarray}
\E[Y_T^2]&=&H(2H-1)\int_{[0,T]^2} e^{-a(T-u)}e^{-a(T-v)}|u-v|^{2H-2}dudv\notag\\
&=&H(2H-1)\int_{[0,T]^2} e^{-a\,u}e^{-a\,v}|u-v|^{2H-2}dudv\notag\\
&\to&H(2H-1)\int_{[0,\infty)^2} e^{-a\,u}e^{-a\,v}|u-v|^{2H-2}dudv\notag\\
&&=a^{-2H}H\Gamma(2H)<\infty. 
\label{byt2}
\end{eqnarray}
Identity (\ref{byt2}) comes from
\begin{eqnarray}\label{beq:sao}
&&(2H-1)\int_{[0,\infty)^2}e^{-a(t+s)}|t-s|^{2H-2}dsdt \nonumber \\
&=&
a^{-2H}(2H-1)\int_{[0,\infty)^2}e^{-(t+s)}|t-s|^{2H-2}dsdt
= a^{-2H}\Gamma(2H) ,
\end{eqnarray}
see, e.g., Lemma 5.1 in Hu-Nualart \cite{bHN} for the second equality.
In particular, we note that 
\begin{equation}\label{bO1}
\E[Y_T^2]=O(1)\quad\mbox{ as $T\to\infty$}.
\end{equation}
Another simple but important fact that will be used is the following identity:
\begin{equation}\label{bivan2}
\int_0^T Y_tdt = \frac{1}{a}(Z^{q, H}_T-Y_T),
\end{equation}
which holds true since
\begin{eqnarray*}
\int_0^T Y_tdt&=&\int_0^T \left( \int_0^t e^{-a (t-u)}dZ^{q, H}_u\right)dt=\int_0^T \left(\int_u^T e^{-a(t-u)}dt\right)dZ^{q, H}_u=\frac{1}{a}(Z^{q, H}_T-Y_T).
\end{eqnarray*}

\section{Proof of Proposition \ref{bfinfini}}

We are now ready to prove Proposition \ref{bfinfini}.

We start by showing that $G_T$ converges well in $L^2(\Omega)$. In order to do so, we will check that the Cauchy
criterion is satisfied.
According to (\ref{beq:13}), we can write $U_T(t)=c(H,q)I_q(g_T(t,\cdot))$, where
\[
g_T(t,\xi_1,\ldots,\xi_q)=\int_0^t e^{-T(t-v)}\prod_{j=1}^q (v-\xi_j)_+^{H_0-\frac32}dv. 
\]
As a result,  we can write, thanks to \cite[identity (3.25)]{bNourdinRosinski},
\begin{eqnarray*}
&&{\rm Cov}(U_S(s)^2,U_T(t)^2) \\
&=& c(H,q)^4
\sum_{r=1}^q \binom{q}{r}^2\bigg\{
q!^2\|g_S(s,\cdot)\otimes_r g_T(t,\cdot)\|^2
+r!^2(2q-2r)!\|g_S(s,\cdot)\widetilde{\otimes}_r g_T(t,\cdot)\|^2\bigg\},
\end{eqnarray*}
implying in turn that
\begin{eqnarray*}
&&\E[G_TG_S]\\
&=&(ST)^{\frac{2}{q}(1-H)+2H}
\int_{[0,1]^2}{\rm Cov}(U_S(s)^2,U_T(t)^2)dsdt\\
&=&c(H,q)^4(ST)^{\frac{2}{q}(1-H)+2H}
\sum_{r=1}^q \binom{q}{r}^2
q!^2 \int_{[0,1]^2}\|g_S(s,\cdot)\otimes_r g_T(t,\cdot)\|^2dsdt\\
&&+c(H,q)^4(ST)^{\frac{2}{q}(1-H)+2H}
\sum_{r=1}^q \binom{q}{r}^2 r!^2(2q-2r)! \int_{[0,1]^2}\|g_S(s,\cdot)\widetilde{\otimes}_r g_T(t,\cdot)\|^2dsdt.
\end{eqnarray*}
To check the Cauchy criterion for $G_T$, we are thus left to show the existence, 
for any $r\in\{1,\ldots,q\}$, of
\begin{eqnarray}
&&\lim_{S,T\to\infty}(ST)^{\frac{2}{q}(1-H)+2H}
 \int_{[0,1]^2}\|g_S(s,\cdot)\otimes_r g_T(t,\cdot)\|^2dsdt \label{blimit1}\\
 &\mbox{and}&\lim_{S,T\to\infty}(ST)^{\frac{2}{q}(1-H)+2H}
 \int_{[0,1]^2}\|g_S(s,\cdot)\widetilde{\otimes}_r g_T(t,\cdot)\|^2dsdt. \label{blimit2}
\end{eqnarray}
Using that $\int_{\mathbb{R}} (u-x)_+^{H_0-\frac32}(v-x)_+^{H_0-\frac32}du =c_H\,|v-u|^{2H_0-2}$ with $c_H$ a constant depending only on $H$ and whose value can change from one line to another,  we have
\begin{eqnarray*}
&&\big(g_S(s,\cdot)\otimes_r g_T(t,\cdot)\big)(x_1,\ldots,x_{2q-2r})\\
&=&c_{H}\int_{0}^s \int_0^t  |v-u|^{(2H_0-2)r}
e^{-S(s-u)}e^{-T(t-v)}\prod_{j=1}^{q-r} (u-x_j)_+^{H_0-\frac32}\prod_{j=q-r+1}^{2q-2r}(v-y_j)_+^{H_0-\frac32}dudv.
\end{eqnarray*}
Now, let $\sigma,\gamma$ be two permutations of $\mathfrak{S}_{2q-2r}$, and
write $g_S(s,\cdot)\otimes_{\sigma,r} g_T(t,\cdot)$ to indicate
the function 
\[(x_1,\ldots,x_{2q-2r})\mapsto 
\big(g_S(s,\cdot)\otimes_r g_T(t,\cdot)\big)(x_{\sigma(1)},\ldots,x_{\sigma(2q-2r)}).
\]
We can write, for some integers $a_1,\ldots,a_4$ satisfying $a_1+a_2=a_3+a_4=q-r$ (and whose exact value is useless in what follows),
\begin{eqnarray*}
&&\big\langle
g_S(s,\cdot)\otimes_{\sigma,r} g_T(t,\cdot),
g_S(s,\cdot)\otimes_{\gamma,r} g_T(t,\cdot)
\big\rangle
\\
&=&c_{H}\int_{0}^s \int_0^t
\int_{0}^s \int_0^t 
 |v-u|^{(2H_0-2)r}
 |z-w|^{(2H_0-2)r}
 |u-w|^{(2H_0-2)a_1}
\\
 &&
\hskip2cm \times 
 |u-z|^{(2H_0-2)a_2}
 |v-w|^{(2H_0-2)a_3}
 |u-z|^{(2H_0-2)a_4}  
  \\
 &&\hskip2cm \times
e^{-S(s-u)}e^{-T(t-v)}
e^{-S(s-w)}e^{-T(t-z)}
dudvdwdz.
\end{eqnarray*}
We deduce that
\begin{eqnarray*}
&&(ST)^{\frac{2}{q}(1-H)+2H}
 \int_{[0,1]^2}\big\langle
g_S(s,\cdot)\otimes_{\sigma,r} g_T(t,\cdot),
g_S(s,\cdot)\otimes_{\gamma,r} g_T(t,\cdot)
\big\rangle dsdt\\
&=&c_H(ST)^{\frac{2}{q}(1-H)+2H}
 \int_{[0,1]^2}\left(
\int_{0}^s \int_0^t
\int_{0}^s \int_0^t 
 |v-u|^{(2H_0-2)r}
 |z-w|^{(2H_0-2)r}
\right.
\\
 &&
\hskip3cm \times 
 |u-w|^{(2H_0-2)a_1}
  |u-z|^{(2H_0-2)a_2}
 |v-w|^{(2H_0-2)a_3}
 |v-z|^{(2H_0-2)a_4}  
  \\
 &&\left.
 \hskip3cm \times
e^{-S(s-u)}e^{-T(t-v)}
e^{-S(s-w)}e^{-T(t-z)}
dudvdwdz\right)
dsdt\\
&=&c_H(ST)^{\frac{2}{q}(1-H)+2H}
 \int_{[0,1]^2}\left(
\int_{0}^s \int_0^t
\int_{0}^s \int_0^t 
 |v-u-t+s|^{(2H_0-2)r}
 |z-w+t-s|^{(2H_0-2)r}
\right.
\\
 &&
\hskip3cm \times 
 |u-w|^{(2H_0-2)a_1}
  |u-z+t-s|^{(2H_0-2)a_2}
 |v-w-t+s|^{(2H_0-2)a_3}
  \\
 &&\left.
 \hskip3cm \times
 |v-z|^{(2H_0-2)a_4}  
 e^{-Su}e^{-Tv}
e^{-Sw}e^{-Tz}
dudvdwdz\right)
dsdt\\
&=&c_HS^{\frac{2}q(1-H)(1+a_1-q)}
T^{\frac{2}q(1-H)(1+a_4-q)}\\
&&\hskip1.5cm\times
 \int_{[0,1]^2}\left(
\int_{0}^{Ss} \int_0^{Tt}
\int_{0}^{Ss} \int_0^{Tt} 
\left|\frac{v}T-\frac{u}{S}-t+s\right|^{(2H_0-2)r}
\left|\frac{z}T-\frac{w}S+t-s\right|^{(2H_0-2)r}
\right.
\\
 &&
\hskip3cm \times 
 |u-w|^{(2H_0-2)a_1}
 \left| \frac{u}S-\frac{z}T+t-s\right|^{(2H_0-2)a_2}
 \left| \frac{v}T-\frac{w}S-t+s\right|^{(2H_0-2)a_3}
  \\
 &&\left.
 \hskip3cm \times
 |v-z|^{(2H_0-2)a_4}  
 e^{-u}e^{-v}
e^{-w}e^{-z}
dudvdwdz\right)
dsdt.
\end{eqnarray*}
It follows that
\[
\lim_{S,T\to\infty} (ST)^{\frac{2}{q}(1-H)+2H}
 \int_{[0,1]^2}\big\langle
g_S(s,\cdot)\otimes_{\sigma,r} g_T(t,\cdot),
g_S(s,\cdot)\otimes_{\gamma,r} g_T(t,\cdot)
\big\rangle dsdt
\]
exists whatever $r$ and $a_1,\ldots,a_4$ such that $a_1+a_2=a_3+a_4=q-r$. Note that this limit is always zero, except when $r=1$, $a_1=a_4=q-1$ and $a_2=a_3=0$, in which case it is given by
\[
c_H \int_{[0,1]^2}
\left|t-s\right|^{4H_0-4}dtds\times 
\left(\int_{\R_+^2}
 |u-w|^{(2H_0-2)(q-1)}e^{-(u+w)}dudw\right)^2<\infty.
\]
Since
\[
g_S(s,\cdot)\widetilde{\otimes}_{r} g_T(t,\cdot) =
\frac{1}{(2q-2r)!}\sum_{\sigma\in\mathfrak{S}_{2q-2r}}
g_S(s,\cdot)\otimes_{\sigma,r} g_T(t,\cdot)
\]
the existence of the two limits (\ref{blimit1})-(\ref{blimit2}) follow, implying in turn the existence of $G_\infty$.
\medskip

Now, let us check the claim about the distribution of $G_\infty$.
Let $\widetilde{Y}_t=U_1(t)$, that is,
$\widetilde{Y}_t = \int_0^t e^{-(t-u)}dZ^{q, H}_u$, $t\geq 0$.
 By a scaling argument,
it is straightforward to check that
$(\widetilde{Y}_{tT})_{t\geq 0} \overset{\rm law}{=} T^H (U_T(t))_{t\geq 0}$
for any fixed $T>0$. As a result,
\[
T^{\frac{2}{q}(1-H)-1}\int_0^T (\widetilde{Y}_t^2-\E[\widetilde{Y}_t^2])dt = 
T^{\frac{2}{q}(1-H)}\int_0^1  (\widetilde{Y}_{tT}^2-\E[\widetilde{Y}_{tT}^2])dt 
\overset{\rm law}{=} G_T.
\]
Using (\ref{bdiu}), we deduce that $G_T/B_{H,q}$ converges in law
to the Rosenblatt distribution of parameter $1-\frac{2}{q}(1-H)$, hence the claim.\qed

\section{Proof of the consistency part in Theorem \ref{bmain}}

The consistency part of Theorem \ref{bmain} is directly obtained as a consequence of the following two propositions.

\begin{prop}
Let $X$ be given by (\ref{bfracV})-(\ref{bfracV2}) with $a>0$, $b\in\R$, $q\geq 1$ and $H\in(\frac12,1)$.
As $T\to\infty$, one has
\begin{equation}\label{blm41}
\frac{1}{T}\int_0^T X_tdt \to b\quad \mbox{a.s.}
\end{equation}
\end{prop}
{\it Proof}.
We use (\ref{bfracV2}) to write
\[
\frac{1}{T}\int_0^T X_t dt = \frac{b}{T}\int_0^T (1-e^{-at}) dt + \frac{1}{T}\int_0^T Y_t dt.
\]
Since it is straightforward that $\frac{b}{T}\int_0^T (1-e^{-at}) dt\to b$, we are left to show that $\frac{1}{T}\int_0^T Y_t dt\to 0$ almost surely.

By (\ref{bivan2}), one can write, for any integer $n\geq 1$,
\[
\E\left[\left(\frac{1}{n}\int_0^n Y_t dt\right)^2\right]\leq \frac{2}{a^2n^2}\big(\E[(Z^{q, H}_n)^2]+\E[Y_n^2]\big) =O(n^{2H-2}),
\]
where the last equality comes from the $H$-selfsimilarity property of $Z^{q, H}$ as well as (\ref{bO1}). 
Since $\frac{1}{T}\int_0^T Y_t dt$ belongs to the $q$th Wiener chaos, it enjoys the 
hypercontractivity property (\ref{beq:hypercontractivity1}). As a result, 
for all $p >\frac{1}{1-H}$ and $\lambda > 0$,
\begin{align*}
\sum_{n=1}^\infty \Prob\bigg(\Big|\frac{1}{n}\int_0^n Y_t dt\Big| > \lambda \bigg) 
&\leq \frac{1}{\lambda^p}\sum_{n=1}^\infty \E\bigg[\Big|\frac{1}{n}\int_0^n Y_t dt\Big|^p\bigg]
& \leq \frac{{\rm cst}(p)}{\lambda^p}\sum_{n=1}^\infty 
\E\bigg[\Big(\frac{1}{n}\int_0^nY_t dt\Big)^2\bigg]^{p/2}\\
& \leq   \frac{{\rm cst}(p)}{\lambda^p}\sum_{n=1}^\infty n^{-(1-H)p}<\infty.
\end{align*}
We deduce from the Borel-Cantelli lemma that
$
\frac{1}{n}\int_0^nY_t dt \rightarrow 0
$
almost surely as $n \to \infty$. 

Finally, fix $T > 0$ and let $n=\lfloor T\rfloor$ be its integer part. We can write
\begin{equation}\label{bbla}
\frac{1}{T}\int_0^TY_t dt = \frac{1}{n}\int_0^nY_t dt + \frac{1}{T}\int_n^TY_t dt + \bigg(\frac{1}{T} - \frac{1}{n}\bigg)\int_0^nY_t dt.
\end{equation}
We have just proved above that $\frac{1}{n}\int_0^nY_t dt$ tends to zero almost surely as $n \to \infty$. We now consider the second and third terms in (\ref{bbla}). 
We have, almost surely as $T\to\infty$,
 \[
 \bigg|\bigg(\frac{1}{T} - \frac{1}{n}\bigg)\int_0^nY_t dt \bigg| = \bigg(1- \frac{n}{T}\bigg) \bigg|\frac{1}{n}\int_0^nY_t dt \bigg| \leq \bigg|\frac{1}{n}\int_0^nY_t dt \bigg| \rightarrow 0,
 \]
and
\[
\bigg|\frac{1}{T}\int_n^TY_t dt\bigg| \leq \frac{1}{n} \int_n^{n+1}|Y_t|dt.
\]

To conclude, it remains to prove that $\frac{1}{n} \int_n^{n+1}|Y_t|dt \to 0$ almost surely as $n \to \infty$. Using (\ref{bO1}) we have, for all fixed $\lambda > 0$, 
\begin{align*}
\mathbb{P} \bigg\{ \frac{1}{n} \int_n^{n+1}|Y_t|dt > \lambda \bigg\} & \leq \frac{1}{\lambda^2} \E\bigg[\bigg( \frac{1}{n} \int_n^{n+1}|Y_t|dt\bigg)^2\bigg] \\ 
& \leq \frac{1}{\lambda^2n^2}  \int_n^{n+1} \int_n^{n+1}  \sqrt{\E[Y_s^2]}\sqrt{\E[Y_t^2]}dsdt =O(n^{-2}).
\end{align*}
Hence, as $n \to \infty$, the Borel-Cantelli lemma applies and implies that $\frac{1}{n} \int_n^{n+1}|Y_t|dt$ goes to zero almost surely. 
This completes the proof of (\ref{blm41}).\qed

\begin{prop}
Let $X$ be given by (\ref{bfracV})-(\ref{bfracV2}) with $a>0$, $b\in\R, q\geq 1$ and $H \in (\frac12, 1)$. As $T\to\infty$, one has
\begin{equation}\label{blm42}
\frac{1}{T}\int_0^T X_t^2dt \to b^2+a^{-2H}H\Gamma(2H)\quad \mbox{a.s.}
\end{equation}
\end{prop}
\noindent
{\it Proof}. 
We first use (\ref{bfracV2}) to write
\[
\frac{1}{T}\int_0^T X_t^2dt = \frac{1}{T}\int_0^T h(t)^2dt + \frac{2}{T}\int_0^T h(t)Y_tdt+\frac{1}{T}\int_0^T Y_t^2dt.
\]
We now study separately the three terms in the previous decomposition.
More precisely we will prove that,  as $T\to\infty$,
\begin{eqnarray}
\frac{1}{T}\int_0^T h(t)^2dt&\to& b^2,\label{bdollar1}\\ 
\frac{1}{T}\int_0^T h(t)Y_tdt&\to& 0\quad\mbox{a.s.}\label{bdollar2}\\ 
\frac{1}{T}\int_0^T Y_t^2dt&\to& a^{-2H}H\Gamma(2H)\quad\mbox{a.s.}\label{bdollar3},
\end{eqnarray}
from which (\ref{blm42}) follows immediately.\\

\underline{First term}.
By Lebesgue dominated convergence, one has
\[
\frac{1}{T}\int_0^T h(t)^2dt = \int_0^1 h(Tt)^2dt =  b^2 \int_0^1 (1- e^{-aTt})^2 dt \rightarrow b^2,
\]
that is, (\ref{bdollar1}) holds.

\bigskip

\underline{Second term}. 
First, we claim that
\begin{equation}\label{bconv}
T^{-H}\int_0^T h(t)Y_tdt \overset{\rm law}{\to} \frac{b}{a} Z^{q, H}_1.
\end{equation}
Indeed, let us decompose:
\begin{eqnarray*}
\int_0^T h(t)Y_tdt=b\int_0^T (1-e^{-a t})Y_tdt
=b\int_0^T Y_tdt - b\int_0^T e^{-a t}\,Y_tdt.
\end{eqnarray*}
Using (\ref{bO1}) in the last line, we can write
\begin{eqnarray}
&&\int_0^T e^{-at}\,Y_tdt = \int_0^T e^{-at}\left(\int_0^t e^{-a(t-s)}dZ^{q, H}_s\right)dt \notag\\
&=& \int_0^T \left(\int_s^T e^{-a(2t-s)}dt\right)dZ^{q, H}_s 
= \frac{1}{2a}\int_0^T (e^{-a(2T- s)}-e^{-as})dZ^{q, H}_s \notag\\
&=&\frac{1}{2a}\left( e^{-a T}Y_T - \int_0^T e^{-as}dZ^{q, H}_s\right)\notag\\
&\to&-\frac{1}{2a}\int_0^\infty e^{-a s}dZ^{q, H}_s\quad\mbox{in $L^2(\Omega)$ as $T\to\infty$.}\label{bblabla}
\end{eqnarray}
The announced convergence (\ref{bconv}) is a consequence of (\ref{bivan2}), (\ref{bblabla}) and the selfsimilarity of $Z^{q, H}$.
Now, relying on the Borel-Cantelli lemma and the fact that $\int_0^T h(t)Y_tdt$ enjoys the
hypercontractivity property, it is not difficult to deduce from (\ref{bconv}) that
(\ref{bdollar2}) holds.\\
 
\underline{Third term}. Firstly, let us write, as $T\to\infty$,
\begin{align}\label{bhou1}
\frac{1}{T}\int_0^T \E[Y_t^2]dt& = H(2H-1) \frac{1}{T}\int_0^Tdt \int_0^t\int_0^t dudv e^{-a u}e^{-a v}|u-v|^{2H-2} \nonumber \\
& = H(2H-1) \int_0^1 dt \int_0^{Tt}\int_0^{Tt} dudv e^{-a u}e^{-a v}|u-v|^{2H-2} \nonumber\\
&\longrightarrow  H(2H-1)  \int_0^\infty \int_0^\infty dudv e^{-a u}e^{-a v}|u-v|^{2H-2} \nonumber\\
&\quad\quad=a^{-2H}H\Gamma(2H). 
\end{align}
To conclude the proof of (\ref{bdollar3}), we are thus left to show that :
\begin{equation}\label{beq:hou2}
\frac{1}{T}\int_0^T (Y_t^2 - \E[Y_t^2])dt \rightarrow 0 \text{ a.s.}
\end{equation}

First, we claim that, as $n \in \mathbb{N}^*$ goes to infinity,
\begin{equation}\label{beq:as1}
G_n:=\frac{1}{n}\int_0^n (Y_t^2 - \E[Y_t^2])dt \to 0 \text{ a.s.}
\end{equation}

Indeed, for all fixed $\lambda > 0$ and $p \geq 1$ we have, by the hypercontractivity property (\ref{beq:hypercontractivity2}) for $G_n$ belonging to a finite sum of Wiener chaoses,
\begin{align*}
\mathbb{P}\{ |G_n| > \lambda \}&\leq \frac{1}{\lambda^p} \E[ |G_n|^p ] \leq  \frac{\text{cst}(p)}{\lambda^p} \E[ G_n^2 ]^{p/2}.
\end{align*}

If ($q \geq 1$ and $H>\frac34$) or $q\geq 2$, combining (\ref{bdiu}) with, e.g., \cite[Lemma 2.4]{bNourdinPoly} leads to
\begin{equation}\label{bazerty}
\sup_{T >0}  
\E\bigg[ 
\bigg(
T^{\frac2q(1-H)-1}\int_0^T (Y_t^2 - \E[Y_t^2])dt 
\bigg)^2 
\bigg] < \infty
\end{equation}
(note that one could also prove (\ref{bazerty}) directly),
implying in turn that 
$
\mathbb{P}\{ |G_n| > \lambda \}
=O(n^{-\frac{2p}{q}(1-H)});$
choosing $p$ so that $\frac{2p}{q}(1-H) > 1$ leads to 
$ \sum_{n=1}^\infty  \mathbb{P}\{ |G_n| > \lambda \} < \infty$,
and so our claim (\ref{beq:as1}) follows from Borel-Cantelli lemma.

If $q=1$ and $H<\frac34$, the same reasoning (but using this time (\ref{bivan}) instead of (\ref{bdiu})) leads exactly to the same conclusion (\ref{beq:as1}).

Now, fix $T > 0$ and consider its integer part $n=\lfloor T\rfloor$. One has
\begin{equation}\label{bblablabla}
G_T = G_n + \frac{1}{T}\int_n^T(Y_t^2 - \E[Y_t^2])dt + \bigg(\frac{1}{T} - \frac{1}{n}\bigg)\int_0^n(Y_t^2 - \E[Y_t^2])dt.
\end{equation}
We have just proved above that $G_n$ tends to zero almost surely as $n \to \infty$. We now consider the third term in (\ref{bblablabla}). We have, using (\ref{beq:as1}):
\begin{align*}
\bigg|\frac{1}{T}- \frac{1}{n}\bigg|\bigg|\int_0^n(Y_t^2 - \E[Y_t^2])dt\bigg|&= \bigg(1-\frac{n}{T}\bigg)\bigg|\frac{1}{n}\int_0^n(Y_t^2 - \E[Y_t^2])dt\bigg|
\leq  |G_n| \rightarrow 0 \quad \text{a.s}.
\end{align*}
Finally, as far as the second term in (\ref{bblablabla}) is concerned, we have
\[
\bigg|\frac{1}{T}\int_n^T(Y_t^2- \E[Y_t^2])dt\bigg|  \leq \frac{1}{n}\int_n^{n+1}|Y_t^2 - \E[Y_t^2]|dt.
\]
To conclude, it thus remains to prove that, as $n \to \infty$,
\begin{equation}\label{beq:10}
F_n:=\frac{1}{n}\int_n^{n+1}|Y_t^2 - \E[Y_t^2]|dt \to 0 \qquad \text{a.s}.
\end{equation}
By hypercontractivity and (\ref{byt2}), one can write
\begin{align*}
 \text{Var}(Y_t^2)  \leq \text{cst}(q)(\E[Y_t^2])^2 \leq \text{cst}(q)  a^{-4H}H^2\Gamma(2H)^2.
\end{align*}
Thus, $\sup_{t} \text{Var}(Y_t^2) < \infty,$ and it follows that
\[
\E[F_n^2] = \frac{1}{n^2}\int_n^{n+1}\int_n^{n+1}\E\big[ \big|Y_t^2 - \E[Y_t^2]\big|\big|Y_s^2 - \E[Y_s^2]\big| \big]dsdt =O(n^{-2}).
\]
Hence $ \sum_{n=1}^\infty  \mathbb{P}\{ |F_n| > \lambda \} \leq \sum_{n=1}^\infty \frac{1}{\lambda^2}E[F_n^2] < \infty$ for all $\lambda > 0$,
and Borel-Cantelli lemma leads to (\ref{beq:10}) and concludes the proof of (\ref{bdollar3}).\qed

\section{Proof of the fluctuation part in Theorem \ref{bmain}}

We now turn to
the proof of the part of Theorem \ref{bmain} related to fluctuations.
We start with the fluctuations of $\widehat{b}_T$, which are easier compared to
$\widehat{a}_T$.\\

\underline{Fluctuations of $\widehat{b}_T$}. Using first (\ref{bdecompoX}) and then (\ref{bivan2}),
we can write
\begin{equation}\label{bbt}
T^{1-H}\big\{\widehat{b}_T - b\} =T^{1-H}\left\{\frac1T\int_0^T Y_tdt - \frac{b}T\int_0^T e^{-at}dt \right\} 
=\frac{Z^{q, H}_T}{aT^H} +O(T^{-H}),
\end{equation}
which will be enough to conclude, see the end of the present section.

\medskip

\underline{Fluctuations of $\widehat{a}_T$}. 
As a preliminary step, we first concentrate on the asymptotic behavior, as $T\to\infty$, of the random quantity 
\[
\ell_T:=\alpha_T -a^{-2H}H\Gamma(2H),
\] 
where $\alpha_T$ is given by (\ref{balpha_T}).
Since $X_t=h(t)+Y_t$, see (\ref{bdecompoX}),  we have
\begin{eqnarray*}
\ell_T 
=A_T+B_T+2C_T+D_T-E_T^2-2E_TF_T-F_T^2,
\end{eqnarray*}
where
\begin{eqnarray*}
A_T&=&\frac{1}{T}\int_0^T (Y_t^2-\E[Y_t^2])dt,\quad
B_T=\frac{1}{T}\int_0^T \E[Y_t^2]dt - a^{-2H}H\Gamma(2H)\\
C_T&=&\frac{1}{T}\int_0^T Y_t h(t)dt,\quad D_T=\frac{1}{T}\int_0^T h^2(t)dt,\quad
E_T=\frac{1}{T}\int_0^T Y_tdt,\quad F_T=\frac{1}{T}\int_0^T h(t)dt.
\end{eqnarray*}
We now treat each of these terms separately.
\medskip

\noindent
{\it Term $B_T$}. Recall from (\ref{byt2}) that
\begin{eqnarray*}
\E[Y_t^2]&=&H(2H-1)\int_{[0,T]^2}e^{-a(u+v)}|u-v|^{2H-2}dudv\\
&\to& H(2H-1) \int_{[0,\infty)^2}e^{-a(u+v)}|u-v|^{2H-2}dudv = a^{-2H}H\Gamma(2H).
\end{eqnarray*} 
As a result,
\begin{eqnarray*}
|B_T|&=&\left|\frac{1}{T}\int_0^T \E[Y_t^2]dt - a^{-2H}H\Gamma(2H)\right|\\
&\leq&\frac{H(2H-1)}{T}\int_0^T dt \int_{[0,\infty)^2\setminus[0,t]^2} dudv\,e^{-a (u+v)}|u-v|^{2H-2}\\
&\leq&\frac{2H(2H-1)}{T}\int_0^T dt\int_t^\infty dv\,e^{-a v}\int_0^\infty du\,e^{-a u}{\bf 1}_{\{v\geq u\}}(v-u)^{2H-2}\\
&\leq&\frac{2H(2H-1)}{T}\int_0^\infty dt\int_t^\infty dv\,e^{-a v}\int_0^v du\,u^{2H-2}\\
&=&\frac{2H}{T}\int_0^\infty dt\int_t^\infty dv\,e^{-a v}v^{2H-1}
=\frac{2H}{T}\int_0^\infty e^{-a v}v^{2H}dv=O(\frac1T).
\end{eqnarray*} 

\medskip
\noindent
{\it Term $C_T$}. We can write
\[
C_T=\frac{1}{T}\int_0^T Y_t h(t)dt =\frac{b}T\int_0^T (1-e^{-a t})Y_tdt
=b\left(\frac1T\int_0^T Y_tdt - \frac1T\int_0^T e^{-a t}Y_tdt\right).
\]
But
\begin{eqnarray*}
\frac1T\int_0^T e^{-a t}Y_tdt&=&
\frac1T\int_0^T e^{-a t}\left(\int_0^t e^{-a(t-s)}dZ^{q, H}_s\right)dt\\
&=&\frac1T\int_0^T e^{a s}\left(
\int_s^T e^{-2a t}dt
\right)dZ^{q, H}_s=\frac1{2a T}\left(\int_0^T e^{-a s}dZ^{q, H}_s - e^{-a T}Y_T\right).
\end{eqnarray*}
Using  (\ref{bO1}) and $\int_0^T e^{-a s}dZ^{q, H}_s\to \int_0^\infty e^{-a s}dZ^{q, H}_s$ in $L^2(\Omega)$, we deduce that
\[
C_T = b\,E_T+O(\frac1T).
\]
\medskip
\noindent
{\it Term $D_T$}. It is straightforward to check that 
\[
D_T=\frac{1}{T}\int_0^T h^2(t)dt = \frac{b^2}{T}\int_0^T (1-e^{-a t})^2dt=b^2+O(\frac1T).
\]

\medskip
\noindent
{\it Term $E_T$}. Thanks to (\ref{bO1}) and (\ref{bivan2}), we have
\[
E_T = \frac1T\int_0^T Y_Tdt =  \frac{1}{a\,T}(Z^{q, H}_T-Y_T)=
\frac{Z^{q, H}_T}{a\,T} + O(\frac1T).
\]
Since $Z^{q, H}_T\overset{\rm law}{=}T^HZ^{q, H}_1$ by selfsimilarity, we deduce
\[
E_T^2
=\frac{(Z^{q, H}_T)^2}{a^2T^2}+ O(T^{H-2}). 
\]
\medskip
\noindent
{\it Term $F_T$}. Similarly to $D_T$, it is straightforward to check that
\[
F_T=\frac{1}{T}\int_0^T h(t)dt = \frac{b}{ T}\int_0^T (1-e^{-a t})dt=b+O(\frac1T).
\]

\medskip
\noindent
Combining everything together, we eventually obtain that
\begin{equation}\label{bcomb}
\ell_T = A_T-\frac{(Z^{q, H}_T)^2}{a^2T^2}+O(T^{-1}).
\end{equation}

\medskip

\underline{Fluctuations of $(\widehat{a}_T,\widehat{b}_T)$}. 
We first rely on the scaling property satisfied by $Z^{q, H}$ to obtain that
\[
\Big( T^{\frac2q(1-H)}A_T, 
T^{-H}Z^{q, H}_T \Big)
\overset{\rm law}{=}
\left(
a^{-\frac2q(1-H)-2H}\,G_{aT},
Z^{q, H}_1
\right)\overset{L^2}{\to}
\left(a^{-\frac2q(1-H)-2H}\,G_{\infty},
Z^{q, H}_1\right).
\]
If ($q=1$ and $H>\frac34$) or $q\geq 2$, a Taylor expansion yields
\begin{eqnarray*}
T^{\frac2q(1-H)}\{\widehat{a}_T-a\}&=&T^{\frac2q(1-H)}\,a\left[\left(1+\frac{a^{2H}\,
\ell_T }{H\Gamma(2H)}\right)^{-\frac1{2H}}-1\right]\\
&=& -\frac{a^{1+2H}}{2H^2\Gamma(2H)}\,\left(T^{\frac2q(1-H)}A_T
-T^{\frac2q(1-H)-2}\frac{(Z^{q, H}_T)^2}{a^2}\right)+o(1),
\end{eqnarray*}
implying in turn
\begin{eqnarray*}
&&\left(T^{\frac2q(1-H)}\{\widehat{a}_T-a\},T^{1-H}\big\{\widehat{b}_T - b\} \right)\\
&=&
\left(-\frac{a^{1+2H}}{2H^2\Gamma(2H)}\,\left[T^{\frac2q(1-H)}A_T
-T^{2(\frac1q-1)(1-H)}\frac{(T^{-H}Z^{q, H}_T)^2}{a^2}\right],\frac{T^{-H}Z^{q, H}_T}{a}\right)+o(1),
\end{eqnarray*}
so that
\begin{eqnarray*}
&&\left(T^{\frac2q(1-H)}\{\widehat{a}_T-a\},T^{1-H}\big\{\widehat{b}_T - b\} \right)\\
& \overset{\rm law}{\to}&
\left\{
\begin{array}{lll}
\left(-\frac{a^{1-\frac2q(1-H)}}{2H^2\Gamma(2H)}\,G_\infty,\frac{Z^{q, H}_1}{a}\right)&\mbox{ if $q\geq 2$}\\
\quad\\
\left(-\frac{a^{2H-1}}{2H^2\Gamma(2H)}\,(G_\infty- (B^H_1)^2),\frac{B^H_1}{a}\right)&\mbox{ if $q=1$ and $H>\frac34$}
\end{array}
\right.,
\end{eqnarray*}
as claimed.

If $q=1$ and $H< \frac34$, we write $B^H$ instead of $Z^{1,H}$ for simplicity. 
We deduce from (\ref{bcomb}) and $T^{-2}(B^{H}_T)^2\overset{\rm law}{=} T^{2H-2}(B^{ H}_1)^2$ that 
$
\sqrt{T}\ell_T = \sqrt{T}A_T+o(1),
$
so that
\begin{eqnarray*}
\left(\sqrt{T}\{\widehat{a}_T-a\},T^{1-H}\big\{\widehat{b}_T - b\} \right)
=
\left(-\frac{a^{1+2H}}{2H^2\Gamma(2H)}\,\sqrt{T}A_T,\frac{T^{-H}B^H_T}{a}\right)+o(1),
\end{eqnarray*}
implying in turn by (\ref{bivanbis}) that
\begin{eqnarray*}
\left(\sqrt{T}\{\widehat{a}_T-a\},T^{1-H}\big\{\widehat{b}_T - b\} \right)\to
\left(-\frac{a^{1+4H}\sigma_H}{2H^2\Gamma(2H)}\, N,\frac{N'}{a}\right),
\end{eqnarray*}
where $N,N'\sim N(0,1)$ are independent, as claimed.

Finally, if $q=1$ and $H= \frac34$,
we deduce again from (\ref{bcomb}) and $T^{-2}(B^{3/4}_T)^2\overset{\rm law}{=} T^{-\frac12}(B^{3/4}_1)^2$ that 
\[
\sqrt{\frac{T}{\log T}}\ell_T = \sqrt{\frac{T}{\log T}}A_T+o(1),
\]
so that, using (\ref{bnunubis}),
\begin{eqnarray*}
\left(\sqrt{\frac{T}{\log T}}\{\widehat{a}_T-a\},T^{\frac14}\big\{\widehat{b}_T - b\} \right)
&=&
\left(-\frac{16a^{\frac52}}{9\sqrt{\pi}}\,\sqrt{\frac{T}{\log T}}A_T,\frac{T^{-\frac34}B^H_T}{a}\right)+o(1)\\
&\to&
\left(\frac{3}4\sqrt{\frac{a}\pi}\, N,\frac{N'}{a}\right),
\end{eqnarray*}
where $N,N'\sim N(0,1)$ are independent.

\qed

 \chapter[Fisher information and multivariate Fourth Moment Theorem] {Fisher information and multivariate Fourth Moment Theorem} 

\begin{center}
T. T. Diu Tran\\
Universit\'{e} du Luxembourg
\end{center}

\begin{center}
\textbf{Abstract}
\end{center}

Under the \textit{non-degenerate} condition in the sense of Malliavin calculus, the convergence in Fisher information distance of random vectors whose components are multiple stochastic integrals to any Gaussian random vectors is actually equivalent to convergence of only the fourth moments of each component. Furthermore, by using the multivariate de Bruijn's identity and score vector-function, we show that the relative entropy is bounded from above by Fisher information, thus extending the result in \cite{tIvanPS}. The equivalence between several forms of convergence (namely, convergence in law, total variation distance, relative entropy, Fisher information distance and uniform convergence for densities) for sequences of \textit{uniformly non-degenerate} random vectors having chaotic components is then given.

\section{Preliminaries}

Fix an integer $d \geq 1$. Throughout the paper, we consider a $d$-dimensional square-integrable and centered random vector $F = (F_1, \ldots, F_d )$ with invertible covariance matrix $C > 0$ (i.e $a^TCa > 0, \forall a \in \mathbb{R}^d \setminus \{0\}$). Assume that the law of $F$ admits a density $f = f_F$ with respect to the Lebesgue measure with support in $\mathbb{R}^d$. Let $Z = (Z_1, \ldots, Z_d)$ be a $d$-dimensional centered Gaussian vector which has the same covariance matrix as $F$ and admits the density $\phi = \phi_d( .; C)$. Without loss of generality, we may and will assume that the vectors $F$ and $Z$ are stochastically independent. 

\subsection{Malliavin operators}

Let $\mathfrak{H}$ be a real separable Hilbert space with inner product $\left\langle ., . \right\rangle$. A centered Gaussian family $X = \{X(h): h \in \mathfrak{H}\}$, defined on a probability space $(\Omega, \mathcal{F}, P)$, is called \textit{isonormal Gaussian process} over $\mathfrak{H}$ if $E[X(h_1)X(h_2)] = \left\langle h, g \right\rangle_{\mathfrak{H}}$ for every $h_1, h_2 \in \mathfrak{H}$.

Let $\mathcal{S}$ be the set of all smooth cylindrical random variables of the form
$$F = g(X(h_1), \ldots, X(h_n)),$$
where $n \geq 1, h_i \in \mathfrak{H}$, and $g$ is infinitely differentiable such that all its partial derivatives have at most polynomial growth. The \textit{Malliavin derivative} of $F$ with respect to $X$ is the element of $L^2(\Omega; \mathfrak{H})$ defined by
$$DF = \sum_{i=1}^n \frac{\partial g}{\partial x_i}(X(h_1), \ldots, X(h_n))h_i.$$
In particular, $DX(h) = h$. For any $m \geq 1$ and $p \geq 1$, we denote by $\mathbb{D}^{m ,p}$ the closure of $\mathcal{S}$ with respect to the norm
$$\| F \|_{m, p}^p = E[|F|^p] + \sum_{j=1}^m E[\| D^jF\|_{\mathfrak{H}^{\otimes j}}^p].$$
The Malliavin derivaive $D$ satisfies the \textit{chain rule}: if $\varphi: \mathbb{R}^n \to \mathbb{R}$ is in $\mathcal{C}^1_b$ and if $F_1, \ldots, F_n$ are in $\mathbb{D}^{1, 2}$, then $\varphi(F_1, \ldots, F_n) \in \mathbb{D}^{1, 2}$ and we have
$$ D\varphi(F_1, \ldots, F_n) = \sum_{i=1}^n \frac{\partial \varphi}{\partial x_i} (F_1, \ldots, F_n)DF_i.$$
The \textit{divergence operator} $\delta$ is defined as the adjoint of the derivaive operator $D$. A random element $u \in L^2(\Omega, \mathfrak{H})$ belongs to the domain of the divergence operator $\delta$, denoted $Dom(\delta)$, if and only if it satisfies
$$|E[\left\langle DF, u \right\rangle_{\mathfrak{H}}]| \leq c_u \sqrt{E[F^2]} \qquad\text{for any} F \in \mathcal{S}.$$
If $u \in Dom(\delta)$, then $\delta(u)$ is defined by the so-called \textit{integration by parts formula}:
$$E[F\delta(u)] = E[\left\langle DF, u \right\rangle_{\mathfrak{H}}],$$
for every $F\in \mathbb{D}^{1, 2}$.

\subsection{Peccati-Tudor theorem for vector-valued multiple stochastic integrals}

Let us state the following very useful result roughly asserting that, for vectors of multiple integrals, joint convergence is actually equivalent to componentwise convergence.

\begin{Theorem} (\cite{tPeccatiTudor2005}) 
Let $d \geq 2$ and $q_1, \ldots, q_d \geq 1$ be some fixed integers. Consider vectors
$$F_n = (F_{n, 1}, \ldots F_{n, d}) = (I_{q_1}(f_{n,1}), \ldots I_{q_d}(f_{n,d})), \quad n \geq 1,$$
with $f_{n, i} \in \mathfrak{H}^{\odot q_i}$. Let $N \sim \mathcal{N}_d(0, C)$ with $\text{det }(C) > 0$ and assume that
$$\lim_{n \to \infty} E[F_{n,i}F_{n,j}] = C_{i,j}, \quad 1 \geq i,j \leq d.$$
Then, as $n \to \infty$, the following two assertions are equivalent:
\begin{enumerate}[a)]
\item $F_n$ converges in law to $N$;
\item For every $1 \leq i \leq d$, $F_{n, i}$ converges in law to $\mathcal{N}(0, C_{i,i})$.
\end{enumerate}
\end{Theorem}

\section{Main results}

We now state and prove our main results, already presented in the Introduction of this thesis. We keep and use the same notation and definitions as in the Introduction.
 
\subsection{Multivariate de Bruijn's identity and upper bound for relative entropy}

Let us recall first the following elementary inequality: if $A, B$ are two $d\times d$ symmetric matrices, and if $A$ is semi-positive definite, then
\begin{equation}\label{teq:2.40}
\lambda_{\min}(B) \times \text{tr}(A) \leq \text{tr}(AB) \leq \lambda_{\max}(B) \times \text{tr}(A),
\end{equation}
where $\lambda_{\min}(B)$ and $\lambda_{\max}(B)$ stand for the minimum and maximum eigenvalues of $B$, respectively. Observe that $\lambda_{\max}(B) = \|B\|_{op}$, the operator norm of $B$. 

\begin{Proposition} Let $F$ and $Z$ be independent random vectors with the same covariance matrix $C$. Then,
\begin{equation}\label{teq:new1}
D(F\| Z) \leq \|C\|_{op} \times \frac{1}{2} \text{tr}(C^{-1}J_{st}(F)) = \|C\|_{op} \times \frac{1}{2}\Big(\text{tr} (J(F)) - \text{tr}(J(Z))\Big).
\end{equation}
As a consequence,
\begin{equation}\label{teq:new2}
\|f - \phi \|^{2}_{L^1(\mathbb{R}^d)} = 4(d_{TV}(F, Z))^2 \leq 2D(F\|Z) \leq \| C\|_{op}\text{tr}(C^{-1}J_{st}(F)).
\end{equation}
\end{Proposition}

\begin{proof}
Write $\text{tr}(J_{st}(F_t)) = \text{tr}(C^{-1}J_{st}(F_t)C)$ and apply (\ref{teq:2.40}) to $A = C^{-1}J_{st}(F_t)$ and $B =C$, in order to obtain that
$$\text{tr}(J_{st}(F_t)) \leq \| C\|_{op} \times \text{tr}(C^{-1}J_{st}(F_t)) = \| C\|_{op} \Big(\text{tr}(J(F_t)) - \text{tr}(C^{-1}) \Big).$$
Furthermore, from definition (\ref{deq:3.50}), it is easily seen that (see appendix for details)
$$ \rho_t(F_t) = E[\sqrt{t}\rho_F(F) + \sqrt{1-t}\rho_Z(Z) | F_t].$$
Then, by Jensen's inequality, the independence of $F$ and $Z$ and the fact that $E[\rho_F(F)] = 0$ and $ J(Z) = C^{-1}$, we have
\begin{align*}
\text{tr}(J(F_t))& = \sum_{j=1}^d E[(\rho_{t, j}(F_t))^2]\\ 
& = \sum_{j=1}^d E\bigg[\Big(  E[\sqrt{t}\rho_{F, j}(F) + \sqrt{1-t}\rho_{Z, j}(Z) | F_t]  \Big)^2\bigg]\\
& \leq \sum_{j=1}^d E\bigg[  E[\Big(\sqrt{t}\rho_{F, j}(F) + \sqrt{1-t}\rho_{Z, j}(Z)\Big)^2 | F_t]  \bigg] \qquad\text{(Jensen)}\\
& = \sum_{j=1}^d E\bigg[  \Big(\sqrt{t}\rho_{F, j}(F) + \sqrt{1-t}\rho_{Z, j}(Z)\Big)^2   \bigg]\\
& =  \sum_{j=1}^d E\bigg[  \Big(\sqrt{t}\rho_{F, j}(F) \Big)^2   \bigg] +  \sum_{j=1}^d E\bigg[  \Big( \sqrt{1-t}\rho_{Z, j}(Z)\Big)^2   \bigg]\\
& = t \text{ tr}(J(F)) + (1-t) \text{ tr}(J(Z))\\
& = t (\text{ tr}(J(F)) - \text{ tr}(C^{-1})) + \text{ tr}(C^{-1}).
\end{align*}
Putting everything together, we obtain
\begin{align*}
D(F \| Z)& = \int_0^1 \frac{\text{tr}(J_{st}(F_t))}{2t}dt \leq \| C\|_{op}\frac{1}{2} \int_0^1 \frac{\text{tr}(J(F_t)) - \text{tr}(C^{-1})}{t}dt \\
& \leq \| C\|_{op}\frac{1}{2}(\text{ tr}(J(F)) - \text{ tr}(C^{-1})) = \| C\|_{op}\frac{1}{2} \text{tr}(C^{-1}J_{st}(F)).
\end{align*}
Our proof is complete.
\end{proof}

%

\subsection{Convergence in the sense of Fisher information in the multivariate Fourth Moment Theorem}

The next theorem (Theorem \ref{tThm1}) will rely on the following notion of non-degeneracy.

\begin{Definition}(\cite[Sec.5]{tHLNualart})
A random vector $F = (F_1, \ldots, F_d)$ in $\mathbb{D}^\infty$ is called \textit{non-degenerate} if its \textit{Malliavin matrix} $\gamma_F = (\left\langle DF_i, DF_j \right\rangle_{\mathfrak{H}})_{1 \leq i,j \leq d}$ is invertible a.s. and $(\text{det}\gamma_F)^{-1} \in \cap_{p \geq 1}L^p(\Omega)$.
\end{Definition}
Recall the notation: $C = (E[F_iF_j])_{1 \leq i, j \leq d}$ covariance matrix of $F$ and $Q: = \text{diag}(q_1, \ldots, q_d)$.

\begin{Theorem} \label{tThm1}
Let $F = (F_1, \ldots, F_d) = (I_{q_1}(f_1), \ldots, I_{q_d}(f_d))$ be a non-degenerate random vector with $1 \leq q_1 \leq \ldots \leq q_d$ and $f_i \in \mathfrak{H}^{\odot q_i}$. Let $\gamma_F$ be the Malliavin matrix of $F$. Then, for any real number $p > 12$,
\begin{equation}\label{teq:result1}
\text{tr}(C^{-1}J_{st}(F)) \leq \text{cst}(C, Q, d) \|(\text{det}\gamma_F)^{-1}\|_{p}^4 \sum_{j=1}^d \Big\| \|DF_j\|_{\mathfrak{H}}^2 - q_jc_{jj}\Big\|_{L^2(\Omega)}.
\end{equation}
\end{Theorem}

\begin{proof}
Denote by $\gamma_F^{-1} = ((\gamma_F^{-1})_{ij})_{1 \leq i, j \leq d}$ the inverse matrix of the Malliavin matrix $\gamma_F$. Note that
 $$D((\gamma_F^{-1})_{ij}DF_j) = D(\gamma_F^{-1})_{ij} \otimes DF_j + (\gamma_F^{-1})_{ij}D^2F_j.$$
Applying H\"{o}lder's inequality with $\frac{1}{p} + \frac{1}{q} = 1$ we have
$$E\Big[\| (\gamma_F^{-1})_{ij}DF_j \|_{\mathfrak{H}}\Big] \leq E\Big[| (\gamma_F^{-1})_{ij}| \|DF_j \|_{\mathfrak{H}}\Big]  \leq E\Big[| (\gamma_F^{-1})_{ij}|^p\Big]^{\frac{1}{p}}E\Big[\|DF_j \|_{\mathfrak{H}}^q\Big]^{\frac{1}{q}}$$
and 
\begin{align*}
&E\Big[\| D((\gamma_F^{-1})_{ij}DF_j) \|_{\mathfrak{H}^{\otimes 2}}^2\Big]\\
& \leq E\Big[\| D(\gamma_F^{-1})_{ij}\|_{\mathfrak{H}}^2\|DF_j \|_{\mathfrak{H}}^2\Big] + E\Big[ |(\gamma_F^{-1})_{ij}|^2\|D^2F_j \|^2_{\mathfrak{H}^{\otimes 2}}\Big] \\ 
& \leq E\Big[\| D(\gamma_F^{-1})_{ij}\|_{\mathfrak{H}}^{2p}\Big]^{\frac{1}{p}}E\Big[\|DF_j \|_{\mathfrak{H}}^{2q}\Big]^{\frac{1}{q}} +E\Big[| (\gamma_F^{-1})_{ij}|^{2p}\Big]^{\frac{1}{p}}E\Big[\|D^2F_j \|_{\mathfrak{H}^{\otimes 2}}^{2q}\Big]^{\frac{1}{q}}.
\end{align*}
Recall from \cite[Lemma 5.6]{tHLNualart} that, for any real number $p >1, \|\gamma_F^{-1}\|_p \leq c \|(\text{det} \gamma)^{-1}\|_{2p}$ where the constant $c$ depends on $q_1, \ldots, q_d$ and $C$. Combining this with the hypercontractivity property of Wiener chaos,
 we conclude that $(\gamma_F^{-1})_{ij}DF_j \in \mathbb{D}^{1, 2}$. The Meyer inequality eventually yields that $(\gamma_F^{-1})_{ij}DF_j \in \text{Dom}\delta$.

Let $ \varphi : \mathbb{R}^n \to \mathbb{R}$ be a test function. We have, on one hand, for all $1 \leq i \leq d$
$$\left\langle DF_i, D\varphi(F) \right\rangle_{\mathfrak{H}} = \sum_{j=1}^d \partial_j \varphi(F) \left\langle DF_i, DF_j \right\rangle_{\mathfrak{H}}.$$
It follows that
\begin{equation}\label{teq:31}
\partial_j \varphi(F) = \sum_{k=1}^d (\gamma_F^{-1})_{jk}\left\langle DF_k, D\varphi(F) \right\rangle_{\mathfrak{H}}.
\end{equation}
Therefore,
\begin{align*}
 - E[\rho_{F, i}(F)\varphi (F)] =  E[\partial_i \varphi(F)] &=\sum_{j=1}^d E\Big[ \left\langle D\varphi(F), (\gamma_F^{-1})_{ij}DF_j  \right\rangle_{\mathfrak{H}}\Big]\\
& =  \sum_{j=1}^d E\Big[ \delta\Big((\gamma_F^{-1})_{ij}DF_j\Big)\varphi(F) \Big]
\end{align*}
Hence
$$\rho_{F, i}(F) = -E\bigg[\delta\Big(\sum_{j=1}^d (\gamma_F^{-1})_{ij}DF_j\Big) \Big| F\bigg].$$
On the other hand, denote by $C^{-1} = (c^{-1}_{ij})_{1\leq i, j \leq d}$ the inverse matrix of $C = (E[F_iF_j])_{1 \leq i,j \leq d}$. Since $\delta DF_j = q_jF_j$, we can write
$$(C^{-1}F)_i= \sum_{j=1}^d c^{-1}_{ij}F_j= \sum_{j=1}^d c^{-1}_{ij} \frac{1}{q_j}\delta DF_j = \delta \Big(\sum_{j=1}^d c^{-1}_{ij} \frac{1}{q_j}DF_j\Big).$$
Therefore,
$$\rho_{F, i}(F) + (C^{-1}F)_i = - \sum_{j=1}^d E\bigg[\delta\bigg( (\gamma_F^{-1})_{ij}DF_j - c^{-1}_{ij}\frac{1}{q_j}DF_j\bigg)\bigg|F\bigg].$$
From (\ref{deq:3.46}) and by Jensen inequality, we have
\begin{align*}
\text{tr}(C^{-1}J_{st}(F))& = \sum_{i=1}^d E\Big[\big(\rho_{F, i}(F) + (C^{-1}F)_i \big)^2\Big]\\ 
& = \sum_{i=1}^d E\bigg[\bigg(  \sum_{j=1}^d E\bigg[\delta\bigg( (\gamma_F^{-1})_{ij}DF_j - c^{-1}_{ij}\frac{1}{q_j}DF_j\bigg)\bigg|F\bigg]  \bigg)^2\bigg]\\
& \leq d \sum_{i=1}^d E\bigg[ \sum_{j=1}^d \bigg( E\bigg[\delta\bigg( (\gamma_F^{-1})_{ij}DF_j - c^{-1}_{ij}\frac{1}{q_j}DF_j\bigg)\bigg|F\bigg]  \bigg)^2\bigg]\\
& = d \sum_{i, j=1}^d E\bigg[ \delta\bigg( (\gamma_F^{-1})_{ij}DF_j - c^{-1}_{ij}\frac{1}{q_j}DF_j\bigg)^2 \bigg]\\
& = d \sum_{i, j=1}^d E\bigg[ \delta\bigg(DF_j\Big( (\gamma_F^{-1})_{ij} - \frac{c^{-1}_{ij}}{q_j}\Big)\bigg)^2 \bigg].
\end{align*}
Now, use the Meyer inequality to get that
\begin{align*}
&E\bigg[ \delta\bigg(DF_j\Big( (\gamma_F^{-1})_{ij} - \frac{c^{-1}_{ij}}{q_j}\Big)\bigg)^2 \bigg]  \leq  \text{cst} \bigg\|\Big((\gamma_F^{-1})_{ij} -  \frac{c^{-1}_{ij}}{q_j}\Big)DF_j\bigg{\|}_{\mathbb{D}^{1,2}}^2\\
& \leq \text{cst} \bigg(E\bigg[\Big\|\Big((\gamma_F^{-1})_{ij} -  \frac{c^{-1}_{ij}}{q_j}\Big)DF_j \Big\|_{\mathfrak{H}}^2\bigg] + E\bigg[\Big\|D\Big(\Big((\gamma_F^{-1})_{ij} -  \frac{c^{-1}_{ij}}{q_j}\Big)DF_j\Big)\Big\|_{\mathfrak{H}^{\otimes 2}}^2\bigg] \bigg).
\end{align*}
Notice that
\begin{align*}
&E\bigg[\Big\|\Big((\gamma_F^{-1})_{ij} -  \frac{c^{-1}_{ij}}{q_j}\Big)DF_j \Big\|_{\mathfrak{H}}^2\bigg]\\
&= E\bigg[\Big| (\gamma_F^{-1})_{ij} - \frac{c^{-1}_{ij}}{q_j}\Big|^2\|DF_j \|_{\mathfrak{H}}^2\bigg]   \leq E\Big[\Big| (\gamma_F^{-1})_{ij}- \frac{c^{-1}_{ij}}{q_j}\Big|^{2p}\Big]^{\frac{1}{p}}E\Big[\|DF_j \|_{\mathfrak{H}}^{2q}\Big]^{\frac{1}{q}}
\end{align*}
and 
\begin{align*}
& E\bigg[\Big\|D\Big(\Big((\gamma_F^{-1})_{ij} -  \frac{c^{-1}_{ij}}{q_j}\Big)DF_j\Big)\Big\|_{\mathfrak{H}^{\otimes 2}}^2\bigg]\\
& \leq 2 E\bigg[\Big\| D\Big((\gamma_F^{-1})_{ij}- \frac{c^{-1}_{ij}}{q_j}\Big)\Big\|_{\mathfrak{H}}^2\|DF_j \|_{\mathfrak{H}}^2\bigg] + 2 E\bigg[ \Big| \Big((\gamma_F^{-1})_{ij}- \frac{c^{-1}_{ij}}{q_j}\Big)\Big|^2\|D^2F_j \|_{\mathfrak{H}^{\otimes 2}}^2\bigg] \\ 
& \leq 2 E\bigg[\Big\| D((\gamma_F^{-1})_{ij}- \frac{c^{-1}_{ij}}{q_j})\Big\|_{\mathfrak{H}}^{2p}\bigg]^{\frac{1}{p}}E\Big[\|DF_j \|_{\mathfrak{H}}^{2q}\Big]^{\frac{1}{q}} \\
& \hspace{5cm} + 2 E\bigg[\Big| \big((\gamma_F^{-1})_{ij}- \frac{c^{-1}_{ij}}{q_j}\big)\Big|^{2p}\Big]^{\frac{1}{p}}E\Big[\|D^2F_j \|_{\mathfrak{H}^{\otimes 2}}^{2q}\Big]^{\frac{1}{q}}.
\end{align*}
Applying the hypercontractivity for $\|DF_j \|_{\mathfrak{H}}^{2q}$ and $\|D^2F_j \|_{\mathfrak{H}^{\otimes 2}}^{2q}$ yields
$$\text{tr}(C^{-1}J_{st}(F)) \leq c\sum_{i,j=1}^d \Big\|(\gamma_F^{-1})_{ij} -  \frac{c^{-1}_{ij}}{q_j}\Big\|_{\mathbb{D}^{1, 2p}}^2 = c \big\|\gamma_F^{-1} - C^{-1}Q^{-1}\big\|_{1, 2p}^2.$$
Finally, recall from \cite[Lemma 5.7]{tHLNualart} that for all $p > 12$,
$$\big\|\gamma_F^{-1} - C^{-1}Q^{-1}\big\|_{1, 2p} \leq c \|(\text{det}\gamma_F)^{-1}\|^4_{p} \sum_{j=1}^d \Big\| \|DF_j\|_{\mathfrak{H}}^2 - q_jc_{jj}\Big\|_{L^2(\Omega)}, $$
where the constant $c$ depends only on $C, Q$ and $d$. This completes the proof.
\end{proof}
%

\subsection{Equivalence between several forms of convergence}

The following statement contains a set of equivalence between several forms of convergence for vectors of multiple integrals, under a uniform non-degeneracy condition.

\begin{corollary} 
Let $d \geq 2$ and let $q_1, \ldots, q_d \geq 1$ be some fixed integers. Consider vectors
$$F_n = (F_{1, n}, \ldots, F_{d, n}) = (I_{q_1}(f_{1, n}), \ldots, I_{q_d}(f_{d, n}) ), \quad n \geq 1,$$
with $f_{i, n} \in \mathfrak{H}^{\odot q_i}$. Let $C = (c_{ij})_{1 \leq i, j \leq d}$ be a symmetric non-negative definite matrix, and let $Z \sim \mathcal{N}_d(0, C)$. Assume that $F_n$ is \emph{uniformly non-degenerate} (that is, $\gamma_{F_n}$ is invertible a.s. for all $n$ and $\limsup_{n \to \infty} \|(\text{det }\gamma_{F_n})^{-1}\|_{L^p} < \infty $ for all $p >12$) and that
$$\lim_{n\to \infty} E[F_{i,n}F_{j,n}] =c_{ij}, \qquad 1 \leq i, j \leq d.$$
Then, as $n \to \infty$, the following assertions are equivalent:
\begin{enumerate}[(a)]
\item $F_n$ converges in law to $Z$;
\item For every $1 \leq i \leq d$, $F_{i,n}$ converges in law to $\mathcal{N}(0, c_{ii})$;
\item $tr (J(F_n)) \to tr (J(Z))$, that is $F_n$ converges to $Z$ in the sense of Fisher information distance;
\item $D (F_n \| Z) \to 0$;
\item $d_{TV}(F_n, Z) \to 0$;
\item  $\|f_{F_n} - \phi\|_\infty \to 0$, where $f_{F_n}$ and $\phi$ are densities of $F_n$ and $Z$ respectively, that is the uniform convergence of densities.
\end{enumerate}
\end{corollary}

\begin{proof}
Equivalence between $(a)$ and $(b)$ corresponds to the Peccati-Tudor theorem (Theorem 5.1.1) for vector-valued multiple stochastic integrals. Moreover, it follows from $(b)$ that $\|DF_{i, n}\|_{\mathfrak{H}}^2 \to q_ic_{ii}$ in $L^2(\Omega)$ as $n \to \infty$ for all $i = 1, \ldots, n$. On the other hand, since $F_n$ is uniformly non-degenerate, then $\sup_n \|(\text{det }\gamma_{F_n})^{-1}\|^4_p < \infty$. Applying Theorem \ref{tThm1}, we have the convergence in the sense of Fisher information $(c)$. The proof of $(c) \Rightarrow (d)$ follows immediately from the estimation between relative entropy and Fisher information (\ref{teq:new1}). Implication $(d) \Rightarrow (e)$ is proved via the Csisz\'{a}r-Kullback-Pinsker inequality, whereas $(e) \Rightarrow (a)$ comes from the fact that convergence in total variation is stronger then convergence in law.

Finally, the equivalence between $(a), (b)$ and $(f)$ comes from \cite[Theorem 5.2]{tHLNualart}, which asserts that
$$\sup_{x \in \mathbb{R}^d}| f_F(x) - \phi(x) | \leq \text{cst} \bigg(\big| (E[(F_{i}F_{j})_{i,j}] - C\big|+ \sum_{j=1}^d \sqrt{E[F_j^4] - 3(E[F_j^2])^2}\bigg),$$
together with the fact that uniform convergence of densities is stronger than convergence in distribution. Our proof is finished.
\end{proof}

\subsection*{Appendix}

If $U, V$ are independent random variables with score functions $\rho_U$ and $\rho_V$, and if $W=U+V$ with sore function $\rho_W$, then for all $\varphi \in \mathcal{C}_c^\infty(\mathbb{R}^d)$
\begin{align*}
E[\rho_W(W)\varphi(W)]& = -E[\varphi'(W)]\\ 
& = - E[\varphi'(U+V)] = E[\rho_U(U) \varphi(U+V)].
\end{align*}
It follows that
$$\rho_W(W) = E[\rho_U(U)| W].$$
Similarly, from the definition of score function (\ref{deq:3.50}), it is easy to give an exact expression for the score vector-function of $F_t$, see e.g. \cite[Lemma V.2]{tIvanPS}.

\begin{Proposition} One has
\begin{equation}\label{teq:rhoFt}
\rho_t(F_t) = E[\sqrt{t}\rho_F(F) + \sqrt{1-t}\rho_Z(Z) | F_t].
\end{equation}
\end{Proposition}

\begin{proof}
$\forall \varphi \in \mathcal{C}_c^\infty(\mathbb{R}^d)$,
\begin{align*}
E[\rho_t(F_t)\varphi(F_t)]& = - E[\nabla \varphi(F_t)]\\ 
& = -E[\nabla \varphi(\sqrt{t}F + \sqrt{1-t}Z)]\\
& = \frac{1}{\sqrt{t}}E[\rho_F(F) \varphi (\sqrt{t}F + \sqrt{1-t}Z)]\\
\text{or}& = \frac{1}{\sqrt{1-t}}E[\rho_N(N) \varphi (\sqrt{t}F + \sqrt{1-t}Z)].
\end{align*}
It follows that
$$\rho_t(F_t) =  \frac{1}{\sqrt{t}}E[\rho_F(F) | F_t] =  \frac{1}{\sqrt{1-t}}E[\rho_Z(Z) | F_t].$$
Then, 
$$\rho_t(F_t) = t\rho_t(F_t) + (1-t)\rho_t(F_t) = E[\sqrt{t}\rho_F(F) + \sqrt{1-t}\rho_Z(Z) | F_t].$$
\end{proof}
\end{document}